\documentclass[a4paper,12pt]{article}

\pdfoutput=1

\usepackage{amsmath, amssymb, amsthm, mathrsfs}
\usepackage{color}
\usepackage{hyperref}

\newtheorem{thm}{Theorem}[section]
\newtheorem{lem}[thm]{Lemma}
\newtheorem{cor}[thm]{Corollary}
\newtheorem{prop}[thm]{Proposition}

\newtheorem{rem}{Remark}[section]
\newtheorem{defn}{Definition}[section]

\numberwithin{equation}{section}

\allowdisplaybreaks[3]

\title{Gauge transformation for the kinetic derivative nonlinear Schr\"odinger equation on the torus}
\author{Nobu KISHIMOTO\thanks{Research Institute for Mathematical Sciences, Kyoto University, Kyoto 606-8502, JAPAN} ~and~ Yoshio TSUTSUMI\thanks{Institute for Liberal Arts and Sciences, Kyoto University, Kyoto 606-8501, JAPAN}}
\date{\empty}

\begin{document}

\maketitle

\begin{abstract}
We consider the kinetic derivative nonlinear Schr\"odinger equation, which is a one-dimensional nonlinear Schr\"odinger equation with a cubic derivative nonlinear term containing the Hilbert transformation.
In our previous work, we proved small-data global well-posedness of the Cauchy problem on the torus in Sobolev space $H^s$ for $s>1/2$ by combining the Fourier restriction norm method with the parabolic smoothing effect, which is available in the periodic setting.
In this article, we improve the regularity range to $s>1/4$ for the global well-posedness by constructing an effective gauge transformation.
Moreover, we remove the smallness assumption by making use of the dissipative nature of the equation.
\end{abstract}

\renewcommand{\thefootnote}{\fnsymbol{footnote}}
\renewcommand{\thefootnote}{\arabic{footnote}}
%


\section{Introduction}
\label{sec:intr}

We consider the Cauchy problem for the kinetic derivative NLS (nonlinear Schr\"odinger equation) on $\mathbf{T}=\mathbf{R}/2\pi \mathbf{Z}$:
\begin{alignat}{2}
   \partial_tu - i \partial_x^2u &= \alpha \partial_x \big[ |u|^2u\big] +\beta \partial_x\big[\mathcal{H}(|u|^2)u\big] , &\qquad t > 0, ~ &x \in \mathbf{T}, \label{kdnls} \\
   u(0, x) &= \phi (x), & &x \in \mathbf{T} , \label{ic}
\end{alignat}
where $\alpha ,\beta \in \mathbf{R}$ are constants, $\mathcal{H}=\mathcal{F}^{-1}(-i\, \mathrm{sgn}(n))\mathcal{F}$ is the Hilbert transformation.
In the case $\beta =0$, \eqref{kdnls} becomes the standard derivative NLS which has been extensively studied.
In this paper, we focus our attention on the case $\beta <0$.

The equation \eqref{kdnls} models propagation of weakly nonlinear and dispersive Alfv\'en waves in a plasma (see, e.g., \cite{DP77} and \cite{MW88}).
In particular, the kinetic term $\beta \partial_x [\mathcal{H}(|u|^2)u]$ represents the effect of resonance between the wave modulation and the ions, and the constant $\beta$ can be positive, negative or zero according to each physical situation.
The case $\beta <0$ corresponds to the situation where the magnetic field loses energy to resonant ion particles.


\subsection{Dissipativity and parabolicity}

We first observe that the $L^2$ norm of a solution to \eqref{kdnls} satisfies (formally) the following identity:
\begin{equation}\label{id:L2dissipation}
\frac{d}{dt}\| u(t)\|_{L^2}^2\ =\ \beta \big\| D_x^{\frac12}(|u(t)|^2)\big\|_{L^2}^2,\qquad t>0,
\end{equation}
where $D_x:=|\partial_x|=\mathcal{H}\partial_x$.
In particular, the equation has \emph{dissipativity} when $\beta <0$.
Note that this property is available in both the periodic and the non-periodic cases.

We next see \emph{parabolicity} of the equation in the periodic setting.
Here, an important role is played by the resonance function: for a $u\bar{u}u$-type (gauge-invariant, cubic) nonlinearity in the Schr\"odinger equation, it is defined by
\[ \Phi (n_1,n_2,n_3):=(n_1+n_2+n_3)^2-n_1^2+n_2^2-n_3^2=2(n_1+n_2)(n_2+n_3).\]
The resonance function appears as the time oscillation $e^{it\Phi}$ in the first nonlinear term of the Picard iteration scheme.
On the one hand, for non-resonant frequency interactions ($\Phi \neq 0$), one has a kind of smoothing effect through the Duhamel integration.
On the other hand, resonant interactions ($\Phi =0$), which do not have such a smoothing mechanism, sometimes become dominant and affect the first-order (linear) behavior of the solution.
For our problem \eqref{kdnls}, the resonant part of the nonlinearity is calculated as follows:
\begin{align*}
&\sum_{n\in \mathbf{Z}}e^{inx}\frac{1}{2\pi}\sum_{\begin{smallmatrix} n_1+n_2+n_3=n \\ (n_1+n_2)(n_2+n_3)=0\end{smallmatrix}} \Big[ \alpha in +\beta in \big( -i\, \mathrm{sgn}(n_1+n_2)\big) \Big] \hat{u}(n_1)\hat{\bar{u}}(n_2)\hat{u}(n_3)\\
&\quad =\ \frac{\alpha}{\pi} \| u\|_{L^2}^2\partial_xu+\frac{\beta}{2\pi}\| u\|_{L^2}^2D_xu+\text{(terms without loss of derivative)}.
\end{align*}
While the first term on the right-hand side just induces a spatial translation, the second one is a first-order parabolic term and thus totally changes the type of the equation. 
We will indeed make full use of this parabolic structure to obtain our results.
It is unclear whether similar parabolicity is observed in the non-periodic setting.
In fact, the influence of the resonant interactions is not as clear as in the periodic case, because the set of (pure) resonant frequencies $\{ (\xi_1,\xi_2,\xi_3)\in \mathbf{R}^3:\Phi (\xi_1,\xi_2,\xi_3)=0\}$ is of measure zero.


\subsection{Previous results on well-posedness}

The Fourier restriction norm method is known as an ideal tool to effectively capture the aforementioned ``non-resonant smoothing effect'' for general nonlinear dispersive equations.
However, it is also known that the method is not sufficient by itself to overcome the loss of derivative in \eqref{kdnls}.
To see this, we first recall that one can heuristically gain derivatives at most corresponding to $|\Phi|^{1/2}$ by the Fourier restriction norm method.
We then take a nonlinear term $u\bar{u}\partial_xu$ and consider the high-low interaction $u_{\rm{low}}\bar{u}_{\rm{low}}\partial_xu_{\rm{high}}$; i.e., the case where the frequency of $\partial_xu$ is much higher than those of the other two $u$'s.
In this case, we have $|\Phi |\sim |n_{1,\rm{low}}+n_{2,\rm{low}}||n_{3,\rm{high}}|$, and hence $|\Phi|^{1/2}$ amounts only to half a derivative (plus half power of the lower frequencies), which is not enough to recover the loss of derivative in $u_{\rm{low}}\bar{u}_{\rm{low}}\partial_xu_{\rm{high}}$.

For the derivative NLS (the case $\beta =0$), the gauge transformation technique has been used to eliminate this type of unfavorable interaction.
Actually, \eqref{kdnls} with $\beta =0$ has the nonlinear terms $2\alpha |u|^2\partial_xu+\alpha u^2\partial_x\bar{u}$, and the first term $2\alpha |u|^2\partial_xu$ can be removed (formally) by the transformation
\[ u(t,x)\quad \mapsto \quad u(t,x)\exp \Big[ \frac{1}{2i}\partial_x^{-1}(2\alpha |u(t)|^2)\Big] .\]
Now, we observe that the other nonlinear term $u(\partial_x\bar{u})u$ has better non-resonant property: for the high-low interaction $u_{\rm{low}}(\partial_x\bar{u}_{\rm{high}})u_{\rm{low}}$ we have $|\Phi |\sim |n_{2,\rm{high}}|^2$, and hence the loss of derivative can be fully recovered by the Fourier restriction norm method.
The gauge transformation was first used in this context by Takaoka~\cite{Ta99} in the non-periodic case and later adjusted to the periodic problem by Herr~\cite{Her06}, both of who succeeded in proving local well-posedness in $H^{1/2}$ by an iteration argument using the Fourier restriction norm.
However, the standard gauge transformation for the derivative NLS does not work well for the kinetic derivative NLS (the case $\beta \neq 0$) because of the Hilbert transformation.
Indeed, the nonlinear term $\beta \mathcal{H}(|u|^2)\partial_xu$ may be eliminated by a similar gauge transformation, but the term $\beta u\mathcal{H}(\bar{u}\partial_xu)$ can not be removed.
To overcome this difficulty, in this paper we will employ the frequency-localized gauge transformation, which is a technique introduced by Tao~\cite{Tao04} for the Benjamin-Ono equation (see also Ionescu and Kenig~\cite{IK07} and Kenig and Takaoka~\cite{KeTa06}).
In that case, it becomes more difficult to estimate the product of the solution and the gauge factor in the Fourier restriction norm.
We explain it in more detail in Subsection~\ref{subsec:intro-outline}.

Next, we recall the known results on the kinetic derivative NLS.
In our previous paper \cite{KT22}, we proved small-data global well-posedness for \eqref{kdnls} with $\beta <0$ in $H^{\frac12+}(\mathbf{T})$ by combining the Fourier restriction norm method with the parabolic smoothing effect mentioned above instead of the gauge transformation.
In order to estimate (the non-resonant part of) the term $u_{\rm{low}}\bar{u}_{\rm{low}}\partial_xu_{\rm{high}}$, roughly speaking, we used each half of the time (Duhamel) integration for the non-resonant smoothing effect (which recovers half a derivative on $u_{\rm{high}}$ and half a derivative on the product $u_{\rm{low}}\bar{u}_{\rm{low}}$) and for the parabolic smoothing effect (which again recovers half a derivative on $u_{\rm{high}}$ but with the constant proportional to $\| u\|_{L^2}^{-1}$).
Together with the Sobolev embedding, this intuitively gives an estimate (for short time) like
\begin{equation}\label{est:intro0}
\begin{aligned}
&\Big\| \int_0^t U(t-t')[u_{\rm{low}}\bar{u}_{\rm{low}}\partial_xu_{\rm{high}}](t')\,dt'\Big\|_{Z^s(T)}\\
&\quad \lesssim \Big( \inf _{t\in [0,T]}\| u(t)\|_{L^2}\Big) ^{-1}\| \partial_x^{-\frac12}(u_{\rm{low}}\bar{u}_{\rm{low}})\| _{L^\infty_TL^\infty_x}\| u_{\rm{high}}\|_{Z^s(T)}\\
&\quad \lesssim \| \phi \|_{L^2}^{-1}\| u\|_{L^\infty_TH^{\frac14+}}^2\| u\|_{Z^s(T)}.
\end{aligned}
\end{equation}
Here, $U(t)$ is the linear propagator of mixed type (i.e., second-order dispersion and first-order parabolicity), and $Z^s(T)$ is a suitable Fourier restriction norm associated with $U(t)$ and restricted to the time interval $[0,T]$.
In fact, $Z^s(T)$ can be defined as the intersection of the usual Fourier restriction norm space for the Schr\"odinger equation with the space $L^2_TH^{s+\frac12}$ related to the parabolic smoothing effect.
In this argument, however, we can not have a positive power of $T$ in \eqref{est:intro0} because we have to use the full $t$-integrability to recover one derivative.
Note also that the smallness in $H^{\frac14+}$ is not enough to close the estimate because of the factor $\| \phi \|_{L^2}^{-1}$.
We then eliminated this factor by using interpolation $\| u\|_{H^{\frac14+}}^2\leq \| u\|_{L^2}\| u\|_{H^{\frac12+}}$ and imposed smallness in $H^{\frac12+}$, to obtain small-data (forward-in-time) local well-posedness in $H^{\frac12+}(\mathbf{T})$.
The local-in-time solution was extended globally by using the $H^1$ \emph{a priori} estimate for small solutions obtained in \cite{KTmatrix} and the parabolic smoothing property of the solution map. 
In \cite{KT22} we also proved forward-in-time unconditional uniqueness and backward-in-time ill-posedness (non-existence) in $H^s(\mathbf{T})$ for $s>3/2$.
At that time, the solutions from large initial data were constructed only in $H^{\frac32+}$ by the energy method, and the global existence of large solutions was left open.

Lastly, we mention that a nonlinear Schr\"odinger equation with nonlocal derivative nonlinear term which appears very similar to \eqref{kdnls},
\begin{equation}\label{eq:nonlocal}
\partial_tu +i\alpha \partial_x^2u\ =\ \beta u(1+i\mathcal{T}_h)\partial_x(|u|^2)+i\gamma |u|^2u,\qquad \alpha ,\beta,\gamma \in \mathbf{R},
\end{equation}
has been studied; for well-posedness results in the non-periodic case, see de~Moura and Pilod \cite{dMP10}, Barros, de~Moura and Santos \cite{BdMS19} and references therein.
Here, $\mathcal{T}_h$, $0<h<\infty$, is a family of certain nonlocal operators (Fourier multipliers) similar to the Hilbert transformation.
This equation comes from a fluid problem, and the major difference from \eqref{kdnls} is that the nonlocal term is multiplied by $i$ and also the term $i\mathcal{T}_h(|u|^2)\partial_xu$ is absent. 
In fact, this equation conserves the $L^2$ norm and does not have dissipative structure, and therefore it is quite a different problem.
Nevertheless, in this paper we will take an approach similar to that in \cite{dMP10}.


\subsection{Main results}

In this paper, we construct an effective gauge transformation for \eqref{kdnls} and combine it with the methods used in \cite{KT22} to obtain large-data global well-posedness of \eqref{kdnls}--\eqref{ic} in $H^{\frac14+}(\mathbf{T})$.
Unless we take advantage of the dissipative property, it seems difficult to bring down the lower bound of the regularity below $s = 1/2$.
The index $s = 1/2$ is thought of as the lowest possible regularity for the derivative NLS when we do not use the complete integrability (see Takaoka \cite{Ta99} and Herr \cite{Her06} for the proof in low regularity without the complete integrability and see Klaus and Schippa \cite{KS22}, Bahouri and Perelman \cite{BP22}, Killip, Ntekoume and Vi\c{s}an \cite{KNV21}, Harrop-Griffiths, Killip and Vi\c{s}an \cite{H-GKV23} and Harrop-Griffiths, Killip, Ntekoume and Vi\c{s}an \cite{H-GKNV22} for the proof in low regularity with the complete integrability).

Our first result is the following theorem concerning the local well-posedness in $H^s(\mathbf{T})$, $s > 1/4$ except for the origin.
\begin{thm}\label{thm:lwp}
Let $\alpha \in \mathbf{R}$, $\beta <0$, and $s>1/4$.
Then, the Cauchy problem \eqref{kdnls}--\eqref{ic} is locally well-posed in $H^s(\mathbf{T})\setminus \{ 0\}$ forward in time, and the solution is smooth for $t>0$.
\end{thm}
See Theorem~\ref{thm:lwp2} and Corollary~\ref{cor:lwp2} for the precise statement.
Next, we establish \emph{a priori} bounds at the $H^1$ level for smooth solutions of arbitrary size (see Proposition~\ref{prop:aprioriH1large} below), which extends our previous result in \cite{KTmatrix} for small solutions.
Combining these bounds with the regularizing property of the solution map obtained in Theorem~\ref{thm:lwp}, we can show:
\begin{thm}\label{cor:gwp}
Let $\alpha \in \mathbf{R}$, $\beta <0$, and $s>1/4$.
Then, the Cauchy problem \eqref{kdnls}--\eqref{ic} is globally well-posed in $H^s(\mathbf{T})\setminus \{ 0\}$ forward in time, and the solution is smooth for $t>0$.
\end{thm}

While $u\equiv 0$ is a global solution to the Cauchy problem with $\phi =0$, the above theorems lack the continuity of the solution map at the origin to assert the well-posedness in $H^s(\mathbf{T})$.
In fact, the local existence time in Theorem~\ref{thm:lwp} shrinks as $\| \phi \|_{L^2}\to 0$.
This is because our proof relies on the parabolicity of the equation, whose strength is proportional to the squared $L^2$ norm of the solution.
Recall that in \cite{KT22} we were able to overcome this issue by imposing smallness in $H^{\frac12+}$.
Here, we employ the estimate obtained in our recent paper \cite{KTshorttime} to deal with this issue.
In \cite{KTshorttime}, we use a different method (i.e., the short-time Fourier restriction norm method) not relying on the parabolicity to obtain for $s>1/4$ the \emph{a priori} bound 
\[ \| u\| _{L^\infty ([0,T];H^s(\mathbf{T}))}\leq C\| \phi\|_{H^s(\mathbf{T})} \]
for small and smooth solutions on $[0,T]$ with $0<T\leq 1$.
This and an approximation argument using the continuous dependence on initial data away from the origin (given in Theorem~\ref{thm:lwp}) imply the same \emph{a priori} bound for (non-zero) small and rough $H^s$ solutions constructed in the above theorems, which then shows the continuity of the solution map at the origin.
Therefore, we have:
\begin{cor}
Let $\alpha \in \mathbf{R}$, $\beta <0$, and $s>1/4$.
Then, the Cauchy problem \eqref{kdnls}--\eqref{ic} is globally well-posed in $H^s(\mathbf{T})$ forward in time, and the solution is smooth for $t>0$.
\end{cor}


\subsection{Outline of the proof of local well-posedness}
\label{subsec:intro-outline}

Here, we shall sketch the proof of Theorem~\ref{thm:lwp}.
The key idea is to construct an ``effective'' gauge transformation which removes the term $\beta u_{\rm{low}}\mathcal{H}(\bar{u}_{\rm{low}}\partial_xu_{\rm{high}})$.
For this purpose, we decompose the solution as $u=P_+u+P_-u+P_0u$, the positive and the negative frequency parts and the zero mode.
Then, since
\[ P_\pm \big[ u_{\rm{low}}\mathcal{H}(\bar{u}_{\rm{low}}\partial_xu_{\rm{high}})\big] \ \approx \ u_{\rm{low}}\bar{u}_{\rm{low}}P_\pm \mathcal{H} \partial_xu_{\rm{high}} \ = \ \mp iu_{\rm{low}}\bar{u}_{\rm{low}}\partial_xP_\pm u_{\rm{high}},\]
the equation for $P_\pm u$ can be written as 
\[ \partial_t(P_\pm u)=i\partial_x^2(P_\pm u)\mp i\beta |u|^2\partial_x(P_\pm u) + \cdots .\]
By analogy to the derivative NLS case $\beta =0$, we are then led to define the gauge transformations separately for $P_+u$ and $P_-u$ as
\[ v_\pm (t,x):= (P_\pm u)(t,x)\exp \Big[ \frac{1}{2i}\partial_x^{-1}(\mp i\beta |u(t)|^2)\Big] .\]
In fact, for the equation \eqref{eq:nonlocal} a similar gauge transformation was introduced in the same spirit by de~Moura and Pilod~\cite{dMP10}.
Such a frequency-localized gauge transformation including the projection $P_\pm$ was originally developed to study the Benjamin-Ono type equations; see, e.g., Tao~\cite{Tao04}, Ionescu and Kenig~\cite{IK07}, and Kenig and Takaoka~\cite{KeTa06}. 
One may wonder if the parabolic structure of the equation is disrupted by these gauge transformations.
However, we actually define the gauge function as $\exp [\frac{1}{2i}\partial_x^{-1}(\mp i\beta(|u|^2-P_0(|u|^2)))]$ so that it becomes periodic in $x$ (see Herr~\cite{Her06} for such a modification in the derivative NLS case).
This means that we only remove the term $\mp i\beta (|u|^2-P_0(|u|^2))\partial_x(P_\pm u)$.
Then, the remaining part
\[ \mp i\beta P_0(|u|^2)\partial_x(P_\pm u)\ =\ \beta P_0(|u|^2) D_x(P_\pm u) \]
acts exactly as the parabolic term for the equation of $P_\pm u$.
Hence, our gauge transformations reserve the parabolic structure of the kinetic term $\beta \partial_x[\mathcal{H}(|u|^2)u]$ in \eqref{kdnls}.

By the above gauge transformations, we obtain the equations for $v_\pm$ as
\[ \partial_tv_\pm = i\partial_x^2v_\pm +\beta P_0(|u|^2)D_xv_\pm +\mathcal{N}_v^\pm (u),\]
where the nonlinear part $\mathcal{N}_v^\pm (u)$ does no longer include the bad interactions of type $u_{\rm{low}}\bar{u}_{\rm{low}}\partial_xu_{\rm{high}}$.
In fact, the main part of $\mathcal{N}_v^\pm (u)$ is of type $u_{\rm{low}}u_{\rm{low}}\partial_x\bar{u}_{\rm{high}}$.
We estimate it with the Fourier restriction norm to recover one derivative from $1/2$ of $t$-integrability, and also use $(\frac12-\varepsilon)$ $t$-integrability for the parabolic smoothing, which yields the estimate (for short time) like
\begin{align*}
&\Big\| \int_0^t U(t-t')\mathcal{N}_v^\pm (u)(t')\,dt'\Big\|_{Z^s(T)}\\
&\quad \lesssim T^\varepsilon \Big( \inf_{t\in [0,T]}\| u(t)\|_{L^2}\Big) ^{-1+2\varepsilon}\| \partial_x^{-\frac12+\varepsilon}(u^2)\|_{L^\infty_TL^\infty_x}\| u\|_{Z^s(T)}\\
&\quad \lesssim T^{0+}\| \phi \|_{L^2}^{-1+}\| u\|_{L^\infty_TH^{\frac14+}}^2\| u\|_{Z^s(T)}.
\end{align*}
In contrast to the estimate \eqref{est:intro0}, this term is under control for arbitrarily large $u$ at the $H^{\frac14+}$ level by choosing $T$ sufficiently small.
As a result, for $s>1/4$ and the gauge transform $v_\pm$ of any solution $u$ to \eqref{kdnls}--\eqref{ic}, we have an estimate like
\begin{equation}\label{est:intro1}
\| v_\pm\|_{Z^s(T)}\leq C_\phi +T^{0+}C_u,
\end{equation}
where we denote any constant determined from $\| \phi\|_{H^s}$ by $C_\phi$ while that depending also on $\| u\|_{Z^s(T)}$ (and on $\| \phi \|_{L^2}^{-1}$) by $C_u$.

Unlike the derivative NLS case (and similarly to the Benjamin-Ono case), the nonlinear term $\mathcal{N}_v^\pm (u)$ can not be written in terms of $v_\pm$ only.
Therefore, we additionally estimate $u$ in terms of $v_\pm$ using the identity
\begin{equation}\label{id:intro}
u\ =\ P_0u+P_+u+P_-u\ =\ P_0u+e^{-\rho^+(u)}v_++e^{-\rho^-(u)}v_-,
\end{equation}
where we write the gauge transforms as $v_\pm=e^{\rho^\pm (u)}P_\pm u$ and thus $\rho ^\pm (u)$ is roughly the anti-derivative of $|u|^2$.
Let us see how to treat the term $e^{-\rho^+(u)}v_+$: we aim to prove an estimate like
\begin{equation}\label{est:intro2}
\| e^{-\rho^+(u)}v_+-v_+\|_{Z^s(T)}\leq \Big( C_\phi +T^{0+}C_u\Big) \| v_+\|_{Z^s(T)}.
\end{equation}
As observed in the Benjamin-Ono case (see, e.g., Ionescu and Kenig~\cite{IK07} and Molinet~\cite{Mo08}), multiplication by gauge function such as $e^{-\rho^+(u)}$ in the Fourier restriction norm space is usually a delicate issue.
In our space $Z^s(T)$, however, it turns out that multiplication by gauge function shows better behavior thanks to its parabolic part $L^2_TH^{s+\frac12}$. 
In fact, we can find a suitable space $G^s(T)$ admitting the following simple estimates
\[ \| gv\|_{Z^s(T)}\lesssim \| g\|_{G^s(T)}\| v\| _{Z^s(T)},\qquad \| e^{c\rho^\pm (u)}-1\|_{G^s(T)}\lesssim e^{C\| u\|_{Z^{s'}(T)}^2}\| u\|_{Z^{s'}(T)}^2\]
for $s>1/4$, where $s'$ is some index satisfying $s>s'>1/4$.
Then, using the integral equation for $u$ we show an estimate like
\[ \| u -U(t)\phi \|_{Z^{s'}(T)}\lesssim T^{0+}\| u\|_{Z^s(T)}^3,\]
which is sufficient to conclude \eqref{est:intro2}.
Compare this estimate again with \eqref{est:intro0}: we can derive a factor $T^{0+}$ since now $s'<s$ and the derivative loss is strictly less than the first order.
Finally, applying \eqref{est:intro2} and then \eqref{est:intro1} in the expression \eqref{id:intro}, we show
\[ \| u\|_{Z^s(T)}\leq C_\phi +T^{0+}C_u\]
for $s>1/4$ and any solution $u\in Z^s(T)$ to \eqref{kdnls}--\eqref{ic}.
This gives the desired \emph{a~priori} estimate for $u$, and a slight modification of the above argument shows similar \emph{a~priori} estimate for the difference.
We can then construct the solution to \eqref{kdnls}--\eqref{ic} in $Z^{\frac14+}(T)$ by a compactness argument.


\subsection{Related results and open problems}

The dispersive-parabolic type Fourier restriction norm space like $Z^s(T)$ was first introduced by Molinet and Ribaud~\cite{MR02} to study low-regularity well-posedness for the Korteweg-de~Vries-Burgers equation
\[ \partial_tu+\partial_x^3u-\partial_x^2u+u\partial_xu=0. \]
Compared with it, the feature of our equation \eqref{kdnls} is that the parabolic structure is not present in the linear part but hidden in the (resonant) nonlinear part. 
In particular, our parabolic term with non-constant coefficient $\frac{\beta}{2\pi}\| u(t)\|_{L^2}^2$ gives rise to a quasilinear nature of the equation, as observed in non-uniformly continuous dependence on initial data around the origin (see \cite[Propositions~5.3 and 1.3 (iii)]{KT22}).
We also remark that such a hidden parabolic structure in the resonant interaction of semilinear dispersive equations on the torus was first discovered by Tsugawa~\cite{Tsu-p} in the context of fifth order KdV-type equations.
In~\cite{Tsu-p}, Tsugawa studied equations with general nonlinearities in sufficiently regular spaces by the refined energy method, aiming to find a precise condition on the nonlinearities for the appearance of parabolicity. 
Meanwhile, we study the equation \eqref{kdnls} with a specific nonlinearity which presents parabolicity, and we aim to obtain low-regularity well-posedness by exploiting the parabolic smoothing effect in the context of the Fourier restriction norm method and the gauge transformation.

We have mentioned that the gauge transformation behaves better in our solution space $Z^s(T)$ reflecting the parabolic effect than in the standard Fourier restriction norm space.
As a consequence, for \eqref{kdnls} with $\beta <0$ uniqueness in the space $Z^s(T)$ can be proved, which is stronger than the uniqueness assertion in the previous results for the case $\beta =0$ using the gauge transformation and the Fourier restriction norm method (see \cite{Ta99,Her06}, where the uniqueness was proved in the inverse gauge image of the Fourier restriction norm space).
In the case $\beta=0$, it was shown later that the uniqueness holds unconditionally in $C_tH^{\frac12+}_x$ (see Win~\cite{Win08}, Mosincat and Yoon~\cite{MY20} and Kishimoto~\cite{Kis-p}).
For $\beta <0$, the unconditional uniqueness in $H^s$ follows by the classical energy method for $s>3/2$, but it is open for $s\leq 3/2$.

Let us next consider the non-periodic problem.
Although the dissipation is still available for $\beta <0$ in the sense of the $L^2$ decay \eqref{id:L2dissipation}, the explicit parabolic term does not appear, as mentioned before, and hence the Cauchy problem might be well-posed both forward and backward in time (or equivalently, both for $\beta <0$ and for $\beta >0$).
The construction by a classical energy method is available in $H^{\frac32+}(\mathbf{R})$ but only forward in time when $\beta <0$, since the dissipative structure is used in energy estimates. 
Below $H^{3/2}$, one may first think of applying the local smoothing ($L^\infty_xL^2_t$) and the maximal function ($L^2_xL^\infty_t$) estimates for the Schr\"odinger group; this is in fact an approach taken in \cite{BdMS19} and \cite{dMP10} for the equation \eqref{eq:nonlocal}.
However, this strategy does not seem to work well for \eqref{kdnls}.
The obstacle here is the term $\beta \mathcal{H}(|u_{\rm{low}}|^2)\partial_xu_{\rm{high}}$, which is absent in \eqref{eq:nonlocal}.
Indeed, noticing that $\| D_x^{s+\frac12}u\|_{L^\infty_xL^2_t}$ is controlled for $\phi \in H^s$ by the local smoothing, one may try to estimate this term as
\begin{align*}
&\Big\| D_x^s \int_0^te^{i(t-t')\partial_x^2}\mathcal{H}(|u_{\rm{low}}(t')|^2)\partial_xu_{\rm{high}}(t')\,dt'\Big\|_{L^\infty_tL^2_x}\\
&\quad \lesssim \big\| D_x^{s-\frac12}\big[ \mathcal{H}(|u_{\rm{low}}|^2)\partial_xu_{\rm{high}}\big] \big\| _{L^1_xL^2_t} \\
&\quad \lesssim \big\| \mathcal{H}(|u_{\rm{low}}|^2)\big\|_{L^1_xL^\infty_t}\| D_x^{s+\frac12}u_{\rm{high}}\|_{L^\infty_xL^2_t},
\end{align*}
but then the first term in the last line can not be bounded by the maximal function norm $\| u_{\rm{low}}\|_{L^2_xL^\infty_t}^2$ due to the fact that $\mathcal{H}$ is not a bounded operator on $L^1$.
Now, the next idea would be to remove (part of) the term $\beta \mathcal{H}(|u|^2)\partial_xu$ by the gauge transformation
\[ v(t,x):= u(t,x)\exp \Big[ \frac{\beta}{2i}\partial_x^{-1}\mathcal{H}(|u(t)|^2) \Big] .\]
However, it seems difficult to handle the anti-derivative term $v \partial_x^{-1}\mathcal{H}\big[ \mathcal{H}(|v|^2)\partial_x(|v|^2)\big]$ appearing in the gauge transformed equation, due to the unboundedness of $\partial_x^{-1}$ in low frequency.%
\footnote{This gauge transformation approach was actually taken by Peres~de~Moura and Pastor~\cite{MP11}, who claimed small-data local well-posedness of \eqref{kdnls}--\eqref{ic} on $H^{\frac12+}(\mathbf{R})$ for both signs of $\beta$.
However, their estimate of the anti-derivative term seems to have a gap, and we do not know how to fill it.}
(Note that this term is easy to handle in the periodic case, because the Hilbert transformation $\mathcal{H}$ removes the low-frequency part (the zero mode) and $\partial_x^{-1}\mathcal{H}$ becomes a bounded operator from $H^{s-1}$ to $H^s$.)
We also remark that this anti-derivative term comes from the $t$-derivative of the gauge function, and it originates in the non-conservative structure of the equation \eqref{kdnls}, i.e., that the $t$-derivative of the density $|u|^2$ can not be written in the form of $\partial_xF(u)$.
In the end, to our knowledge no result is known on well-posedness for \eqref{kdnls} posed on $\mathbf{R}$ below $H^{3/2}$.
(Nevertheless, in \cite{KTshorttime} we applied the short-time Fourier restriction norm method to prove the \emph{a priori} bound in $H^s(\mathbf{R})$ for $s>1/4$ in the dissipative case $\beta <0$, and this method can also be used to prove local well-posedness for some higher regularity but lower than $3/2$. 
We plan to address this problem in a forthcoming paper.)

Another natural question is long time behavior of the global solutions constructed in this paper.
In the periodic case, for any single-mode initial data $\phi =Ae^{inx}$ ($A\in \mathbf{C}$, $n\in \mathbf{Z}$) the solution to \eqref{kdnls}--\eqref{ic} is explicitly given by
\[ u(t,x)\ =\ Ae^{i(-n^2+\alpha |A|^2n)t}e^{inx}. \]
It can be seen from \eqref{id:L2dissipation}%
\footnote{There might be some weak solutions of very low regularity which do not satisfy (the integrated version of) \eqref{id:L2dissipation}.
We can at least verify \eqref{id:L2dissipation} for the solutions constructed in Theorem~\ref{thm:lwp}.} 
that these single-mode solutions are the only solutions to \eqref{kdnls} with conserved $L^2$ norm.
Hence, in the dissipative case $\beta <0$ it is expected that any solution with more than one mode has asymptotics either decaying to zero or approaching to one of the single-mode solutions.
It seems, however, quite difficult to determine the asymptotic behavior of each solution; it is even not clear whether the latter scenario can happen.
The situation is somewhat simpler in the non-periodic case, where the trivial solution appears to be the only non-dissipative solution.
In this case we still do not have global existence, however.

It is also of interest from both mathematical and physical viewpoints to investigate the non-dissipative (non-kinetic) limit; i.e., whether for initial data in $H^s$ the solution of \eqref{kdnls} with $\beta <0$ converges in $H^s$ to that of the derivative NLS as $\beta \to 0$.
This is indeed a singular limit connecting an equation of dispersive-parabolic type and that of purely dispersive type.
By virtue of the gauge transformation developed in this paper, in principle it is possible to recover the loss of derivative only by the Fourier restriction norm and without using the parabolic smoothing effect.
It is then likely that for $s\geq 1/2$ one can show local well-posedness in $H^s(\mathbf{T})$ uniformly in $\beta \leq 0$ and in $\| \phi\|_{L^2}^{-1}$, by which the non-dissipative limit would be verified.
Below $H^{1/2}$ the problem would become substantially more difficult, since we essentially use the parabolicity to construct the solution for $\beta <0$.
Note that the derivative NLS $\beta =0$ is also solvable below $H^{1/2}$ by virtue of its complete integrability; the best known (global-in-time) well-posedness results are in $L^2(\mathbf{R})$ without any size restriction in the non-periodic case (see Harrop-Griffiths, Killip, Ntekoume and Vi\c{s}an \cite{H-GKNV22}) and in $H^{1/6}(\mathbf{T})$ under the condition $\| \phi\|_{L^2}^2<4\pi/|\alpha|$ in the periodic case (see Killip, Ntekoume and Vi\c{s}an \cite{KNV21}).
In particular, our result for $\beta <0$ in the periodic case is better than the result for $\beta =0$ in that no smallness is required, whereas the result for $\beta=0$ is better in regularity.
It would be very interesting if one could give another proof of the result for $\beta=0$ (or even improve it) without using complete integrability by the analysis of the kinetic derivative NLS with $\beta <0$, or conversely, if a certain property of the derivative NLS coming from its complete integrability could be shown to persist in the case of negative $\beta$.


\subsection{Notation and plan of the paper}

We list the notation used throughout this paper.
Let $a$ and $b$ be two positive parameters.
We write $a \lesssim b$ if $a \leq C b$ for some $C > 0$, where $C$ is independent of $a$ and $b$.
We also write $a \sim b$ if $a \lesssim b$ and $a \gtrsim b$.
When an implicit constant $C$ in $a \lesssim b$ depends on the quantity $\eta$, we write $a \lesssim_\eta b$ to indicate the dependence of $C$ on $\eta$.
We write $a \ll b$ if an implicit constant $C$ in $a \lesssim b$ is very small.
For any small $\varepsilon > 0$, we denote $a \pm \varepsilon$ by $a \pm$, respectively.

The plan of this paper is as follows.
In Section~\ref{sec:gauge}, we describe in detail how to define the gauge transformation which plays a crucial role in the paper.
In Section~\ref{sec:lwp}, we mention the precise statement of the local well-posedness theorem in $H^s\setminus \{ 0\}$, $s > 1/4$ and prove it.
For that purpose, we define the function spaces in Subsection~\ref{subsec:space} and present the linear estimates in Subsection~\ref{subsec:linearest}.
Furthermore, we show the estimates of the gauge transformation and the nonlinearity in Subsections~\ref{subsec:gaugeest} and \ref{subsec:nonlest}, respectively.
In Subsection~\ref{subsec:apriori-lwp}, we show {\it a priori} bounds of the solution and the difference of two solutions, and as a result, in Subsection~\ref{subsec:proof-lwp}, we prove the local well-posedness in $H^s\setminus \{ 0\}$, $s > 1/4$.
In Section~\ref{sec:global}, we extend the $H^1$ {\it a priori} bound of the small solution given in \cite{KTmatrix} to the solution of arbitrary size.
By combining it with the local well-posedness theorem in Section~\ref{sec:lwp}, we then show Theorem~\ref{cor:gwp}, i.e., the global well-posedness in $H^s\setminus \{ 0\}$, $s > 1/4$.
Finally, in Appendix~\ref{sec:appendix}, we give a proof of the local well-posedness for \eqref{kdnls}--\eqref{ic} in the high regularity space.
This is because we construct a solution in low regularity as the limit of smooth solutions with approximated initial data, rather than by a direct fixed point argument.
The proof is based on the classical energy method but does not depend on the parabolicity of the system.

\section{Gauge transformation}
\label{sec:gauge}

In this section, we present the gauge transformation which removes the worst nonlinear interactions of high-low type and does not break the dissipative structure of the nonlinearity (see \eqref{def:vpm} below).
In the proof of local well-posedness, we actually work on a renormalized equation:
\begin{equation}\label{kdnls-r}
\partial_tu - i \partial_x^2u = \alpha \partial_x \big[ |u|^2u\big] -2\alpha P_0(|\phi|^2)\partial_xu+\beta \partial _x\big[ \mathcal{H}(|u|^2)u\big] ,
\end{equation}
where 
\[ P_0f:=\frac{1}{2\pi}\int _0^{2\pi}f(x)\,dx.\]
Well-posedness of the Cauchy problems \eqref{kdnls}--\eqref{ic} and \eqref{kdnls-r}--\eqref{ic} are ``equivalent'' through the spatial translation
\[ u(t,x)\quad \mapsto \quad u\big( t,x-2\alpha P_0(|\phi |^2)t\big) .\] 
\begin{rem}
Since the $L^2$ norm is not conserved in our problem, renormalization by $2\alpha P_0(|\phi |^2)\partial_xu$ is different from subtracting $2\alpha P_0(|u|^2)\partial_xu$.
We do not take the latter approach, because the mapping property of the relevant spatial translation
\[ u(t,x)\quad \mapsto \quad u\Big( t,x-2\alpha \int _0^tP_0(|u(t')|^2)\,dt'\Big) \]
is not clear.
Indeed, it seems difficult to characterize the image of the space $Z^s_1(T)$ (to be defined later) under this map.
\end{rem}

Let $P_\pm :=\mathcal{F}^{-1}\chi _{\{ \pm n>0\}}\mathcal{F}$, where $\chi_A$ is the characteristic function of $A$ for a set $A$.
We note that
\[ 1=P_++P_-+P_0,\quad \mathcal{H}=-iP_++iP_-,\quad D_x=-i\partial_xP_++i\partial_xP_-.\]
For $f :\mathbf{T}\to \mathbf{C}$ satisfying $P_0f=0$, we can define its periodic primitive
\[ \partial _x^{-1}f(x):=\mathcal{F}^{-1}\big[ (in)^{-1}\mathcal{F}f\big] (x)=\frac{1}{2\pi}\int _0^{2\pi}\Big( \int _z^xf(y)\,dy\Big) \,dz .\]
Let $u$ be a solution of \eqref{kdnls-r}, and we begin by observing that
\begin{align*}
P_\pm \partial _x\big[ |u|^2u\big] &=2P_\pm \big[ |u|^2\partial_xu\big] +P_\pm \big[ u^2\partial_x\bar{u}\big] \\
&=2|u|^2\partial_xP_\pm u+\mathcal{R}_1^\pm ,\\
P_\pm \partial_x\big[ \mathcal{H}(|u|^2)u\big] &=P_\pm \big[ \mathcal{H}(\bar{u}\partial_xu)u+\mathcal{H}(u\partial_x\bar{u})u+\mathcal{H}(|u|^2)\partial_xu\big] \\
&=|u|^2\partial_xP_\pm \mathcal{H}u+\mathcal{H}(|u|^2)\partial_xP_\pm u +\mathcal{R}_2^\pm \\
&=\big[ \mp i|u|^2 +\mathcal{H}(|u|^2)\big] \partial _xP_\pm u +\mathcal{R}_2^\pm ,
\end{align*}
where (and hereafter) we use $\mathcal{R}$ to denote the nonlinear terms which can be treated by the standard Fourier restriction norm method.
Precisely,
\begin{align*}
\mathcal{R}_1^\pm &=P_\pm \big[ u^2\partial_x\bar{u}\big] +2[P_\pm ,|u|^2]\partial_xu,\\
\mathcal{R}_2^\pm &=P_\pm \big[ \mathcal{H}(u\partial_x\bar{u})u\big] +P_\pm \big[ u[\mathcal{H},\bar{u}]\partial_xu\big] +[P_\pm ,|u|^2]\mathcal{H}\partial_xu+[P_\pm ,\mathcal{H}(|u|^2)]\partial_xu .
\end{align*}
Therefore, we have
\begin{gather*}
\partial _tP_\pm u-i\partial_x^2P_\pm u \ =\ F[u]\partial _xP_\pm u +\mathcal{R}_3^\pm ,\\
F[u]:=2\alpha |u|^2-2\alpha P_0(|\phi|^2) \mp i\beta |u|^2 +\beta \mathcal{H}(|u|^2),\\
\mathcal{R}_3^\pm :=\alpha \mathcal{R}_1^\pm +\beta \mathcal{R}_2^\pm . 
\end{gather*}
The previous results (see, e.g., Tao~\cite{Tao04} and Kenig--Takaoka~\cite{KeTa06}) suggest that the nonlinear term $F[u]\partial_xP_\pm u$ can be made tamer by the gauge transformation $P_\pm u\mapsto \exp (\frac{1}{2i}\partial_x^{-1}F[u])P_\pm u$.
In the periodic setting, we need to modify the primitive $\partial _x^{-1}F$ to the periodic one $\partial_x^{-1}P_{\neq 0}F$, where $P_{\neq 0}:=1-P_0=\mathcal{F}^{-1}\chi _{\{ n\neq 0\}}\mathcal{F}$.
Note that
\begin{align*}
P_0F[u] &= 2\alpha \big[ P_0(|u|^2) -P_0(|\phi |^2)\big] \mp i\beta P_0(|u|^2),\\
P_{\neq 0}F[u]&=(2\alpha \mp i\beta ) P_{\neq 0}(|u|^2)+\beta \mathcal{H}(|u|^2)\\
&=2\alpha P_{\neq 0}(|u|^2)\mp 2i\beta P_\pm (|u|^2)\ =\ 2i (-i\alpha P_{\neq 0}\mp \beta P_\pm )(|u|^2).
\end{align*}
We thus define the gauge transformation as follows:
\begin{gather}
\begin{gathered}
v_\pm :=e^{\rho ^\pm [u]}P_\pm u, \\
\rho ^\pm [u]:=\partial_x^{-1}P^\pm _{\alpha ,\beta}(|u|^2), \quad P_{\alpha ,\beta}^\pm :=-i\alpha P_{\neq 0}\mp \beta P_\pm .
\end{gathered}\label{def:vpm}
\end{gather}
\begin{rem}
In order to make $v_\pm$ spatially periodic, we use a primitive of $P_{\neq 0}F[u]$ rather than that of $F[u]$.
This suggests that the above gauge transformation will not change the ``resonant'' part of the nonlinearity $P_0F[u]$, which includes the parabolic term 
\[ \mp i\beta P_0(|u|^2)\partial _xP_\pm u\ =\ \beta P_0(|u|^2)D_xP_\pm u.\]
Therefore, the dissipative structure will be taken over by the gauge transformed equation.
\end{rem}

Let us derive the equation for $v_\pm$.
Recall that each of $P_\pm u$ satisfies the equation
\begin{equation}
\begin{split}
(\partial _t-i\partial_x^2)P_\pm u&=\mp i\beta P_0(|u|^2)\partial_xP_\pm u +2\alpha \big[ P_0(|u|^2)-P_0(|\phi |^2)\big]\partial_xP_\pm u \\
&\quad +2iP^\pm _{\alpha ,\beta}(|u|^2)\partial_xP_\pm u +\mathcal{R}_3^\pm .
\end{split}\label{eq:upm}
\end{equation}
Using
\begin{align*}
\partial_tv_\pm  &= e^{\rho ^\pm [u]}\big[ P_\pm u_t+\rho ^\pm _tP_\pm u\big] ,\\
-i\partial_x^2v_\pm &=-ie^{\rho ^\pm [u]}\big[ P_\pm u_{xx} +2\rho ^\pm_xP_\pm u_x+\rho ^\pm_{xx}P_\pm u+(\rho ^\pm _x)^2P_\pm u\big] ,
\end{align*}
and the equation \eqref{eq:upm}, we have
\begin{align*}
&(\partial _t-i\partial_x^2)v_\pm \\
&= \mp i\beta P_0(|u|^2)e^{\rho ^\pm [u]}\partial_xP_\pm u +2\alpha \big[ P_0(|u|^2)-P_0(|\phi |^2)\big] e^{\rho ^\pm [u]}\partial_xP_\pm u\\
&\quad +e^{\rho ^\pm [u]}\big[ 2i P^\pm _{\alpha,\beta}(|u|^2)P_\pm u_x -2i\rho ^\pm _xP_\pm u_x+\mathcal{R}_3^\pm +(\rho ^\pm _t-i\rho ^\pm _{xx}) P_\pm u-i(\rho ^\pm _x)^2 P_\pm u\big] .
\end{align*}
For the first line on the right-hand side, we see that
\begin{align*}
&\mp i\beta P_0(|u|^2)e^{\rho ^\pm [u]}\partial_xP_\pm u +2\alpha \big[ P_0(|u|^2)-P_0(|\phi |^2)\big] e^{\rho ^\pm [u]}\partial_xP_\pm u\\
&\quad =\mp i\beta P_0(|u|^2)\partial_xv_\pm +2\alpha \big[ P_0(|u|^2)-P_0(|\phi |^2)\big] \partial_xv_\pm\\
&\qquad \pm i\beta P_0(|u|^2)P^\pm _{\alpha,\beta}(|u|^2)v_\pm -2\alpha \big[ P_0(|u|^2)-P_0(|\phi |^2)\big] P^\pm _{\alpha,\beta}(|u|^2)v_\pm \\
&\quad =\beta P_0(|\phi |^2)D_xv_\pm \mp 2i\beta P_0(|u|^2) P_\mp \partial_xv_\pm +\big[ P_0(|u|^2)-P_0(|\phi |^2)\big] (2\alpha \partial_xv_\pm +\beta D_x v_\pm ) \\
&\qquad -(2\alpha \mp i\beta ) P_0(|u|^2)P^\pm _{\alpha,\beta}(|u|^2)v_\pm +2\alpha P_0(|\phi |^2)P^\pm _{\alpha,\beta}(|u|^2)v_\pm .
\end{align*}
Moreover, from the equation \eqref{kdnls-r},
\begin{gather*}
\rho ^\pm _x=P^\pm _{\alpha,\beta}(|u|^2),\qquad \rho ^\pm _{xx}= P^\pm _{\alpha,\beta} (\bar{u}u_x+u\bar{u}_x) ,\\[5pt]
\begin{aligned}
\rho ^\pm _t 
&=\partial_x^{-1}P^\pm _{\alpha,\beta}\Big[ 2\mathrm{Re}\, \Big\{ \bar{u}\Big( iu_{xx}+\alpha \partial_x \big[ |u|^2u\big] -2P_0(|\phi |^2)u_x+\beta \partial _x\big[ \mathcal{H}(|u|^2)u\big] \Big) \Big\} \Big] \\
&=\partial_x^{-1}P^\pm _{\alpha,\beta}\Big[ \partial _x\Big(i\bar{u}u_x-iu\bar{u}_x+\frac{3}{2}\alpha |u|^4-2\alpha P_0(|\phi |^2)|u|^2+2\beta \mathcal{H}(|u|^2)|u|^2\Big) \\
&\qquad\qquad\quad -\beta \mathcal{H}(|u|^2)\partial_x(|u|^2)\Big] \\
&=iP^\pm _{\alpha,\beta} (\bar{u}u_x-u\bar{u}_x) +P^\pm _{\alpha,\beta}\Big[ \frac{3}{2}\alpha |u|^4-2\alpha P_0(|\phi |^2)|u|^2+2\beta \mathcal{H}(|u|^2)|u|^2\Big] \\
&\qquad -\beta \partial_x^{-1}P^\pm _{\alpha,\beta}\big[ \mathcal{H}(|u|^2)\partial_x(|u|^2)\big] ,
\end{aligned}
\end{gather*}
and therefore,
\begin{align*}
(\rho ^\pm _t-i\rho ^\pm _{xx}) P_\pm u&\;= \mathcal{R}_4^\pm \\
&:= -2iP^\pm _{\alpha,\beta}(u\bar{u}_x) P_\pm u -\beta \partial_x^{-1}P^\pm _{\alpha,\beta}\big[ \mathcal{H}(|u|^2)\partial_x(|u|^2)\big] \cdot P_\pm u\\
&\;\quad +P^\pm _{\alpha,\beta}\Big[ \frac{3}{2}\alpha |u|^4-2\alpha P_0(|\phi |^2)|u|^2+2\beta \mathcal{H}(|u|^2)|u|^2\Big] \cdot P_\pm u.
\end{align*}
Hence, we have
\begin{align*}
(\partial _t-i\partial_x^2)v_\pm &=\beta P_0(|\phi |^2)D_x v_\pm \\
&\quad \mp 2i\beta P_0(|u|^2) P_\mp \partial_xv_\pm +\big[ P_0(|u|^2)-P_0(|\phi |^2)\big] (2\alpha \partial_xv_\pm +\beta D_x v_\pm ) \\
&\quad -(2\alpha \mp i\beta ) P_0(|u|^2)P^\pm _{\alpha,\beta}(|u|^2)v_\pm +2\alpha P_0(|\phi |^2)P^\pm _{\alpha,\beta}(|u|^2)v_\pm \\
&\quad +e^{\rho ^\pm [u]}\Big( \mathcal{R}_3^\pm +\mathcal{R}_4^\pm -i\big[ P_{\alpha,\beta}^\pm (|u|^2)\big] ^2 P_\pm u\Big) .
\end{align*}
Inserting $\mathcal{R}_3^\pm$ and $\mathcal{R}_4^\pm$, we rewrite it as
\begin{equation}\label{eq:vpm}
\begin{split}
&(\partial _t-i\partial_x^2)v_\pm \ =\ \beta P_0(|\phi |^2)D_x v_\pm +\big[ P_0(|u|^2)-P_0(|\phi |^2)\big] (2\alpha \partial_xv_\pm +\beta D_x v_\pm ) \\
&+e^{\rho ^\pm [u]}\Big( \alpha P_\pm \big( u^2\bar{u}_x\big) +\beta P_\pm \big( \mathcal{H}(u\bar{u}_x)u\big) -2iP^\pm _{\alpha ,\beta}(u\bar{u}_x)P_\pm u\Big) \\
&+e^{\rho ^\pm [u]}\Big( 2\alpha [P_\pm ,|u|^2]u_x+\beta P_\pm \big( u[\mathcal{H},\bar{u}]u_x\big) +\beta [P_\pm ,|u|^2]\mathcal{H}u_x+\beta [P_\pm ,\mathcal{H}(|u|^2)]u_x\Big) \\
&-\beta \partial_x^{-1}P^\pm _{\alpha,\beta}\big( \mathcal{H}(|u|^2)\partial_x(|u|^2)\big) \cdot e^{\rho ^\pm [u]}P_\pm u \mp 2i\beta P_0(|u|^2) \partial_xP_\mp \big[ e^{\rho ^\pm [u]}P_\pm u\big] \\
&+\Big\{ P^\pm _{\alpha,\beta}\Big[ \frac{3}{2}\alpha |u|^4-(2\alpha \mp i\beta ) P_0(|u|^2)|u|^2+2\beta \mathcal{H}(|u|^2)|u|^2\Big] -i\big[ P^\pm _{\alpha ,\beta}(|u|^2)\big] ^2\Big\} e^{\rho ^\pm [u]}P_\pm u.
\end{split}
\end{equation}


\section{Local well-posedness}
\label{sec:lwp}

In this section, we show the local well-posedness results for \eqref{kdnls} by using the gauge transformation described in Section~\ref{sec:gauge}.
In the following, we fix the parameters $\alpha \in \mathbf{R}$, $\beta <0$ in the equation.
The main task is to establish the following theorem for the renormalized Cauchy problem \eqref{kdnls-r} and \eqref{ic}.
\begin{thm}\label{thm:lwp2}
Let $s>1/4$.
Then, the Cauchy problem \eqref{kdnls-r} and \eqref{ic} is locally well-posed in $H^s(\mathbf{T})\setminus \{ 0\}$.
More precisely:
\begin{itemize}
\item For any $R\geq r>0$, there exists $T=T(s,R,r)>0$ such that the following holds.
For any $\phi$ in
\[ B_{H^s}(R,r):=\{ \varphi \in H^s(\mathbf{T}):\| \varphi\|_{H^s}\leq R,\,\| \varphi\|_{L^2}\geq r\} ,\]
there exists a solution $u$ in $Z^s_1(T)\hookrightarrow C([0,T];H^s)\cap L^2(0,T;H^{s+\frac12})$ satisfying 
\[ \| u\| _{Z^s_1(T)}\leq C=C(s,R,r),\qquad \min _{0\leq t\leq T}\| u(t)\|_{L^2}\geq \frac12 \| \phi \|_{L^2}^2.\]
(See the next subsection for the definition of the space $Z^s_\mu(T)$, $\mu \geq 0$.)
\item The solution is unique in $Z^s_1(T)$.
\item For each $R,r>0$, the solution map $\phi \mapsto u$ defined above is Lipschitz continuous from $B_{H^s}(R,r)$ to $Z^s_1(T)$.
\item Assume that $\phi \in B_{H^s}(R,r)$ has additional regularity: $\phi \in H^{s+\theta}(\mathbf{T})$ ($\theta>0$).
Then, the solution $u\in Z^s_1(T)$ constructed above belongs to $C([0,T];H^{s+\theta})\cap L^2(0,T;H^{s+\theta+\frac12})$.
\item The solution $u$ is smooth for $t>0$: $u\in C^\infty ((0,T]\times \mathbf{T})$.
\end{itemize}
\end{thm}

\begin{rem}
As we mentioned in Section~\ref{sec:intr}, continuity of the solution map at the origin (even as a map into $C([0,T];H^s)$) is not obtained in the above theorem.
In fact, the local existence time $T$ given in the theorem shrinks as $r\to 0$, though the solution itself turns out to be extended globally in time (see Theorem~\ref{cor:gwp} and its proof in Section~\ref{sec:global}).
As a result, establishing local well-posedness in the whole space $H^s(\mathbf{T})$ needs the help of \emph{a priori} estimate from \cite{KTshorttime}, which is another tough job.
However, when $s>1/2$ it is expected that the nonlinear estimates in this section can be established without essentially relying on the parabolicity.
If this would be the case, the existence time $T$ could be chosen uniformly for small $r$, and large-data local well-posedness in $H^s(\mathbf{T})$ could be shown directly.
\end{rem}

From Theorem~\ref{thm:lwp2}, we immediately obtain the result on the original equation.
\begin{cor}\label{cor:lwp2}
The same results as in Theorem~\ref{thm:lwp2} hold for the Cauchy problem \eqref{kdnls} and \eqref{ic}, except that we obtain merely continuity of the solution map $\phi \mapsto u$ from $B_{H^s}(R,r)$ to $Z^s_1(T)$.
\end{cor}

Most of the rest of this section is devoted to proving Theorem~\ref{thm:lwp2}, and a proof of Corollary~\ref{cor:lwp2} is given at the end of the section.

\subsection{Formulation and definition of function spaces}
\label{subsec:space}

We compute (the Fourier transform of) the resonant part of the nonlinear terms
\[ \alpha \partial_x \big[ |u|^2u\big] -2\alpha P_0(|\phi |^2)\partial _xu+\beta \partial_x\big[ \mathcal{H}(|u|^2)u\big] \]
of the equation \eqref{kdnls-r} for $u$ as
\begin{align*}
&\frac{1}{2\pi}\Big( \sum _{n_1=n_2,\,n_3=n}+\sum _{n_1=n,\,n_2=n_3}-\sum _{n_1=n_2=n_3=n}\Big) \\
&\qquad\qquad \times in \big[ \alpha -i\beta \mathrm{sgn}(n_1-n_2)\big] \hat{u}(n_1)\bar{\hat{u}}(n_2)\hat{u}(n_3) -2\alpha P_0(|\phi|^2)in\hat{u}(n)\\
&\quad =2\alpha \big[ P_0(|u|^2)-P_0(|\phi |^2)\big] in\hat{u}(n)+\beta P_0(|u|^2) |n|\hat{u}(n)\\
&\qquad +\frac{\beta}{2\pi}\Big( \sum _{n'}\big[ \mathrm{sgn}(n-n')-\mathrm{sgn}(n)\big] |\hat{u}(n')|^2\Big) n\hat{u}(n) -\frac{\alpha}{2\pi}in|\hat{u}(n)|^2\hat{u}(n).
\end{align*}
Note that the second line on the right-hand side consists of nonlinear terms without derivative loss.
Therefore, the equation \eqref{kdnls-r} for $u$ is rewritten as
\begin{gather}
\partial _tu-i\partial_x^2u-\beta P_0(|\phi |^2) D_xu\ =\ \mathcal{N}_0[u,\phi;u] +\mathcal{N}_u[u,u,u],\label{eq:u'}
\end{gather}
where
\begin{align*}
\mathcal{N}_0[u_1,u_2;w]&:=\big[ P_0(|u_1(t)|^2)-P_0(|u_2(t)|^2)\big] \big[ 2\alpha \partial_xw+\beta  D_xw\big] ,\\
\mathcal{N}_u[u_1,u_2,u_3]&:=\mathcal{F}^{-1}\bigg[ \frac{\beta}{2\pi}\Big( \sum _{n'}\big[ \mathrm{sgn}(n-n')-\mathrm{sgn}(n')\big] \hat{u}_1(n')\bar{\hat{u}}_2(n')\Big) n\hat{u}_3(n) \\
&\qquad\qquad -\frac{\alpha}{2\pi}in\hat{u}_1(n)\bar{\hat{u}}_2(n)\hat{u}_3(n)\\
&\qquad\qquad +\frac{1}{2\pi}\sum _{\begin{smallmatrix}n=n_1-n_2+n_3\\ n_2\neq n_1,n_3\end{smallmatrix}}in \big[ \alpha -i\beta \mathrm{sgn}(n_1-n_2)\big] \hat{u}_1(n_1)\bar{\hat{u}}_2(n_2)\hat{u}_3(n_3)\bigg].
\end{align*}
If $u_1,u_2$ are solutions of \eqref{eq:u'} with initial data $\phi _1$ and $\phi _2$, respectively, then $w:=u_1-u_2$ satisfies
\begin{equation}\label{eq:u-diff}
\begin{split}
&\partial _tw-i\partial_x^2w-\beta P_0(|\phi_1|^2)D_xw\\
&\quad =\ \mathcal{N}_0[u_1,\phi _1;w]+\mathcal{N}_0[u_1,u_2;u_2]-2\alpha \big[ P_0(|\phi _1|^2)-P_0(|\phi _2|^2)\big] \partial_xu_2\\
&\qquad +\mathcal{N}_u[u_1,u_1,u_1]-\mathcal{N}_u[u_2,u_2,u_2]. 
\end{split}
\end{equation}
For the gauge transformed equation \eqref{eq:vpm}, we can rewrite it as
\begin{gather}\label{eq:vpm'}
\partial_tv_\pm -i\partial_x^2v_\pm - \beta P_0(|\phi|^2)D_xv_\pm \ =\ \mathcal{N}_0[u,\phi; v_\pm ] +\mathcal{N}_v^\pm [u],
\end{gather}
where
\begin{align}
&\mathcal{N}_v^\pm [u] \ =\ \Big( \alpha P_\pm \big( u^2\bar{u}_x\big) +\beta P_\pm \big( \mathcal{H}(u\bar{u}_x)u\big) -2iP^\pm _{\alpha ,\beta}(u\bar{u}_x)P_\pm u\Big) e^{\rho ^\pm [u]} \label{nonl:ubarx}\\
&\quad +\Big( \beta P_\pm \big( u[\mathcal{H},\bar{u}]u_x\big) +2\alpha [P_\pm ,|u|^2]u_x+\beta [P_\pm ,|u|^2]\mathcal{H}u_x+\beta [P_\pm ,\mathcal{H}(|u|^2)]u_x\Big) e^{\rho ^\pm [u]} \label{nonl:u3x}\\
&\quad \mp 2i\beta P_0(|u|^2) \partial_xP_\mp \big[ e^{\rho ^\pm [u]}P_\pm u\big] \label{nonl:pm}\\
&\quad -\beta \partial_x^{-1}P^\pm _{\alpha,\beta}\big( \mathcal{H}(|u|^2)\partial_x(|u|^2)\big) \cdot e^{\rho ^\pm [u]}P_\pm u \label{nonl:p-1}\\
&\quad +\Big\{ P^\pm _{\alpha,\beta}\Big[ \frac{3}{2}\alpha |u|^4-(2\alpha \mp i\beta ) P_0(|u|^2)|u|^2+2\beta \mathcal{H}(|u|^2)|u|^2\Big] -i\big[ P^\pm _{\alpha ,\beta}(|u|^2)\big] ^2\Big\} e^{\rho ^\pm [u]}P_\pm u.\label{nonl:quintic}
\end{align}
For two solutions $v_{\pm,1},v_{\pm,2}$ of \eqref{eq:vpm} corresponding to solutions $u_1,u_2$ of \eqref{kdnls-r} with initial data $\phi _1,\phi _2$, the difference equation for $w_\pm :=v_{\pm,1}-v_{\pm,2}$ is given by
\begin{gather}\label{eq:vpm-diff}
\begin{aligned}
&\partial _tw_\pm -i\partial_x^2w_\pm - \beta P_0(|\phi_1|^2)D_xw_\pm \\
&\quad =\ \mathcal{N}_0[u_1,\phi _1;w_\pm ]+\mathcal{N}_0[u_1,u_2;v_{\pm ,2}]\\
&\qquad -2\alpha \big[ P_0(|\phi _1|^2)-P_0(|\phi _2|^2)\big] \partial_xv_{\pm ,2} +\mathcal{N}^\pm _v[u_1]-\mathcal{N}^\pm _v[u_2]. 
\end{aligned}
\end{gather}

Each of these equations has the dispersive-parabolic linear part $\partial_t-i\partial_x^2+\mu D_x$ for some constant $\mu \geq 0$ depending on the data.
We write the linear propagator as
\[ U_\mu (t):=e^{t(i\partial_x^2-\mu D_x)},\qquad t\geq 0,\quad \mu \geq 0.\]
Following \cite{MR02,KT22}, we will use the Fourier restriction norms associated with this linear part:
\begin{defn}
For $s,b\in \mathbf{R}$ and $\mu \geq 0$, we define the spaces $X^{s,b}_\mu$, $Y^{s,b}_\mu$ by the following norms:
\begin{align*}
\| u\|_{X^{s,b}_\mu}&=\big\| \langle n\rangle^s\langle i(\tau +n^2)+\mu |n|\rangle^b\tilde{u}(\tau ,n)\big\| _{\ell ^2_nL^2_\tau},\\
\| u\|_{Y^{s,b}_\mu}&=\big\| \langle n\rangle^s\langle i(\tau +n^2)+\mu |n|\rangle^b\tilde{u}(\tau ,n)\big\| _{\ell ^2_nL^1_\tau}.
\end{align*}
We normally drop the subscript $\mu$ in these spaces with $\mu =0$ unless $\mu = 0$ needs to be emphasized.
Then, we define the spaces $Z^{s}_\mu$ for solutions, $N^s_\mu$ for nonlinearities and $G^s$ for the gauge part by the following norms:
\begin{gather*}
\| u\|_{Z^{s}_\mu}=\| u\|_{X^{s,\frac{1}{2}}_\mu}+\| u\|_{Y^{s,0}_\mu},\qquad \| u\|_{N^s_\mu}=\| u\|_{X^{s,-\frac12}_\mu}+\| u\|_{Y^{s,-1}_\mu},\\
\| g\|_{G^s}=\| \langle n\rangle ^{s+1}\tilde{g}\| _{\ell ^2_nL^{\frac43}_\tau}+\| \langle n\rangle ^{s+\frac34}\tilde{g}\| _{\ell ^2_nL^1_\tau }+\| \langle n\rangle ^{s+\frac14}\langle \tau \rangle ^{\frac12}\tilde{g}\| _{\ell ^2_nL^2_\tau} .
\end{gather*}
We also define the restricted spaces $Z^s_\mu(T)$, $N^s_\mu(T)$, $G^s(T)$ of functions on $[0,T]\times \mathbf{T}$ in the usual way.
\end{defn}

\begin{rem}\label{rem:normequiv}
For fixed $s\in \mathbf{R}$, the $Z^s_\mu$ norms for $\mu>0$ are all mutually equivalent and also equivalent to the $Z^s_0\cap L^2(\mathbf{R};H^{s+\frac12}(\mathbf{T}))$ norm.
In fact, for $\mu, \nu >0$ we see that
\begin{align*}
(\tfrac{\mu}{\nu}\wedge 1)^{\frac12}\| u\|_{Z^s_\nu}&\lesssim \| u\|_{Z^s_\mu}\lesssim (\tfrac{\mu}{\nu} \vee 1)^{\frac12}\| u\|_{Z^s_\nu},\\
\| u\|_{Z^s_0}+(\mu \wedge 1)^{\frac12}\| u\|_{L^2_tH^{s+\frac12}_x}&\lesssim \| u\|_{Z^s_\mu}\lesssim \| u\|_{Z^s_0}+\mu^{\frac12}\| u\|_{L^2_tH^{s+\frac12}_x}.
\end{align*} 
\end{rem}

We will use the following properties of the $Z^s_1(T)$ norm.
\begin{lem}\label{lem:Zmu}
Let $s\in \mathbf{R}$ and $T>0$.
\begin{enumerate}
\item[(i)] We have the embeddings $C^1([0,T];H^{s+2})\hookrightarrow Z^s_1(T)\hookrightarrow C([0,T];H^s)$.
Moreover, these embeddings are uniform in $T$ for $0<T<1$.
\item[(ii)] For $u\in C^1([0,T];H^{s+2})$, the map $T'\mapsto \| u\|_{Z^s_1(T')}$ is nondecreasing and continuous on $(0,T]$, and $\lim _{T'\to 0}\| u\|_{Z^s_1(T')}\lesssim \| u(0)\|_{H^s}$.
\end{enumerate}
\end{lem} 

\begin{proof}
(i) Let $u\in C^1([0,T];H^{s+2})$, and take an extension $u^\dagger \in C^1(\mathbf{R};H^{s+2})$ such that $u^\dagger (t)=0$ outside some finite interval and 
\[ \| u^\dagger \|_{C^1(\mathbf{R};H^{s+2})}\lesssim \| u\|_{C^1([0,T];H^{s+2})},\]
with the implicit constant being uniform in $T$.
For instance, we can define
\[ u^\dagger(t):=\psi (t)\times \left\{
\begin{alignedat}{2}
&2u(0)-u(\tfrac{-tT}{T+(-t)}) &\qquad &(t<0),\\
&u(t) &&(t\in [0,T]),\\
&2u(T)-u(\tfrac{T^2}{t}) &&(t>T),
\end{alignedat}\right. \]
where $\psi$ is a bump function. 
Then, we have 
\begin{align*}
\| u^\dagger \|_{Z^s_1}&\lesssim \| u^\dagger \|_{X^{s,1}_1}\lesssim \| u^\dagger \|_{L^2(\mathbf{R};H^{s+2})}+\| \partial_tu^\dagger \|_{L^2(\mathbf{R};H^s)}\\
&\lesssim \| u^\dagger \|_{C^1(\mathbf{R};H^{s+2})}\lesssim \| u\|_{C^1([0,T];H^{s+2})},
\end{align*}
which shows the first embedding.
The second embedding follows from $Y^{s,0}\hookrightarrow C(\mathbf{R};H^s)$.

(ii) By the definition of the restricted norm $Z^s_1(T')$, the map $T'\mapsto \| u\|_{Z^s_1(T')}$ is nondecreasing.
To show continuity, we employ the argument of \cite[Lemma~4.2]{IKT08}.
For given $u\in C^1([0,T];H^{s+2})$, it suffices to prove the following: for any $\varepsilon >0$, there is $\delta >0$ such that if $T_1,T_2\in (0,T]$ satisfy $|\frac{T_2}{T_1}-1|\leq \delta$, then $\| u\|_{Z^s_1(T_2)}\leq \| u\|_{Z^s_1(T_1)}+C\varepsilon$.
To this end, we first fix $N\gg 1$ so that $\| P_{>N}u\|_{C^1([0,T];H^{s+2})}\ll \varepsilon$ and then, by the estimate in (i), $\| P_{>N}u\|_{Z^s_1(T')}\leq \varepsilon$ for any $T'\in (0,T]$.
Hence, we have 
\[ \| u\|_{Z^s_1(T_2)}\leq \| P_{\leq N}u\|_{Z^s_1(T_2)}+\varepsilon. \]
Define $\kappa :=\frac{T_2}{T_1}$ and $D_\kappa u(t,x):=u(\frac{t}{\kappa},x)$.
By the estimate in (i), we have
\begin{align*}
&\big| \| P_{\leq N}u\|_{Z^s_1(T_2)}-\| D_\kappa P_{\leq N}u\|_{Z^s_1(T_2)}\big| \\
&\quad \leq C\| u-D_\kappa u\|_{C^1([0,T_2];H^{s+2})}\\
&\quad =C\sup_{t\in [0,T_2]}\Big( \| u(t)-u(\tfrac{t}{\kappa})\|_{H^{s+2}}+\| (\partial_tu)(t)-\tfrac{1}{\kappa}(\partial_tu)(\tfrac{t}{\kappa})\|_{H^{s+2}}\Big) \\
&\quad \leq \varepsilon \qquad (|\kappa -1|\ll 1).
\end{align*}
Next, we take an extension $u^\ddagger\in Z^s_1$ of $u$ on $[0,T_1]$ satisfying $\| u^\ddagger \|_{Z^s_1}\leq \| u\|_{Z^s_1(T_1)}+\varepsilon$.
Since $D_\kappa P_{\leq N}u^\ddagger$ is an extension of $D_\kappa P_{\leq N}u$ on $[0,\kappa T_1]=[0,T_2]$, we have
\begin{align*}
&\| D_\kappa P_{\leq N}u \|_{Z^s_1(T_2)}\leq \| D_\kappa P_{\leq N}u^\ddagger \|_{Z^s_1}\\
&\quad =\| \langle n\rangle ^s\langle i(\tau +n^2)+|n|\rangle ^{\frac12}\kappa \widetilde{P_{\leq N}u^\ddagger}(\kappa \tau ,n)\|_{\ell ^2_nL^2_\tau} +\| \langle n\rangle ^s\kappa \widetilde{P_{\leq N}u^\ddagger}(\kappa \tau ,n)\|_{\ell ^2_nL^1_\tau}\\
&\quad =\kappa ^{\frac12}\| \langle n\rangle ^s\langle i(\tfrac{\tau}{\kappa} +n^2)+|n|\rangle ^{\frac12}\widetilde{P_{\leq N}u^\ddagger}(\tau ,n)\|_{\ell ^2_nL^2_\tau} +\| \langle n\rangle ^s\widetilde{P_{\leq N}u^\ddagger}(\tau ,n)\|_{\ell ^2_nL^1_\tau}\\
&\quad \leq \kappa ^{\frac12} \Big( (1\vee \tfrac{1}{\kappa}) +CN^2\big| \tfrac{1}{\kappa}-1\big| \Big) ^{\frac12} \| u^\ddagger \|_{Z^s_1}\\
&\quad \leq \kappa ^{\frac12} \Big( (1\vee \tfrac{1}{\kappa}) +CN^2\big| \tfrac{1}{\kappa}-1\big| \Big) ^{\frac12} \big( \| u\|_{Z^s_1(T_1)}+\varepsilon \big) ,
\end{align*}
where we have used the inequality
\begin{align*}
\langle i(\tfrac{\tau}{\kappa} +n^2)+|n|\rangle &\leq \langle \tfrac{i(\tau +n^2)+|n|}{\kappa}\rangle +\big| \tfrac{1}{\kappa}-1\big| (n^2+|n|)\\
&\leq \Big( (1\vee \tfrac{1}{\kappa}) +CN^2\big| \tfrac{1}{\kappa}-1\big| \Big) \langle i(\tau +n^2)+|n|\rangle \qquad (|n|\lesssim N).
\end{align*}
Hence, we have
\[ \| D_\kappa P_{\leq N}u \|_{Z^s_1(T_2)}\leq \| u\|_{Z^s_1(T_1)}+2\varepsilon  \qquad (|\kappa -1|\ll 1). \]
Combining the above estimates, we obtain
\[ \| u\|_{Z^s_1(T_2)}\leq \| u\|_{Z^s_1(T_1)}+4\varepsilon \]
if $|\kappa -1|\leq \delta$, where $\delta>0$ is determined only by $u$ and $\varepsilon$ (and $N=N(u,\varepsilon)$).
Therefore, the continuity is shown.

The last statement follows from $\| u-U_1(t)u(0)\|_{Z^s_1(T')}\to 0$ ($T'\to 0$) by \cite[Lemma~2.1]{KT22} and $\| U_1(t)u(0)\|_{Z^s_1(1)}\lesssim \| u(0)\|_{H^s}$ by Lemma~\ref{lem:linear1} below.
\end{proof}

\subsection{Linear estimates}
\label{subsec:linearest}

We first recall the following linear estimates related to the propagator $U_\mu(t)$.
\begin{lem}[{\cite[Lemmas 2.3, 2.4]{KT22}, \cite[Propositions 2.1, 2.3(a)]{MR02}}]\label{lem:linear1}
For $s\in \mathbf{R}$, $\mu \geq 0$ and $0<T\leq 1$, we have
\[ \| U_\mu(t)\phi\|_{Z^s_\mu(T)}\lesssim \| \phi\|_{H^s},\qquad \Big\| \int_0^t U_\mu (t-t')F(t')\,dt'\Big\|_{Z^s_\mu(T)}\lesssim \| F\|_{N^s_\mu(T)}.\]
\end{lem}
The following lemma gives estimates related to the time cut-off.
\begin{lem}\label{lem:linear2}
Let $\psi \in C^\infty_0(\mathbf{R})$ be a bump function and define $\psi_T(t):=\psi (t/T)$ for $T>0$.
Then, for $s\in \mathbf{R}$, $\mu \geq 0$, $0<T\leq 1$ and $0\leq b\leq \frac12$, we have
\[ \| \psi_T(t)u\|_{X^{s,b}_\mu}\lesssim T^{\frac12-b}\| u\|_{Z^s_\mu},\qquad \| \psi _T(t)u\|_{Z^s_\mu}\lesssim \| u\|_{Z^s_\mu}.\]
Moreover, for $0\leq b'<\frac12$ and $0<\theta <\frac12-b'$, we have
\[ \| \psi_T(t)F\| _{N^s_\mu}\lesssim _{b',\theta}T^{\theta}\| F\|_{X^{s,-b'}_\mu}.\]
\end{lem}
\begin{proof}
The first two estimates were stated in \cite[Lemma 2.5]{KT22} without proof.
We only give a proof for the last estimate, while a similar argument verifies the first two.
It suffices to show
\begin{align}
\| \langle \tau\rangle ^{-\frac12}(\hat{\psi}_T*\hat{f})\|_{L^2}+\| \langle \tau\rangle ^{-1}(\hat{\psi}_T*\hat{f})\|_{L^1}\lesssim _{b',\theta}T^{\theta}\| \langle \tau \rangle ^{-b'}\hat{f}\|_{L^2}\label{est:timecutoff}
\end{align}
for any function $f$ in $t$.
Using
\begin{align*}
\langle \tau \rangle ^{-\frac12}&\lesssim \langle \tau \rangle ^{-\frac12}\frac{\langle \tau \rangle ^{b'}+|\tau_1|^{b'}}{\langle \tau-\tau_1\rangle ^{b'}}=\big( \langle \tau \rangle ^{-(\frac12-b')}+\langle \tau \rangle ^{-\frac12}|\tau_1|^{b'}\big) \langle \tau-\tau_1\rangle ^{-b'},\\
\langle \tau \rangle ^{-1}&\lesssim \big( \langle \tau \rangle ^{-(1-b')}+\langle \tau \rangle ^{-1}|\tau_1|^{b'}\big) \langle \tau-\tau_1\rangle ^{-b'}
\end{align*}
and the H\"older inequality, we have
\[ \text{LHS of \eqref{est:timecutoff}} \lesssim _{b',\theta} \big\| \hat{\psi}_T*(\langle \tau \rangle ^{-b'}\hat{f})\big\|_{L^{\frac{2}{1-2\theta}}}+\big\| (|\tau |^{b'}\hat{\psi}_T)*(\langle \tau \rangle ^{-b'}\hat{f})\big\|_{L^{\frac{2}{1-2(\theta +b')}}}.\]
From the Young inequality, this is bounded by 
\begin{align*}
&\big( \| \hat{\psi}_T\|_{L^{\frac{1}{1-\theta}}}+\| |\tau |^{b'}\hat{\psi}_T\|_{L^{\frac{1}{1-\theta -b'}}}\big) \| \langle \tau \rangle ^{-b'}\hat{f}\|_{L^2} \\
&\quad =T^{\theta}\big( \| \hat{\psi}\|_{L^{\frac{1}{1-\theta}}}+\| |\tau |^{b'}\hat{\psi}\|_{L^{\frac{1}{1-\theta -b'}}}\big) \| \langle \tau \rangle ^{-b'}\hat{f}\|_{L^2}.
\end{align*}
Hence, we have shown \eqref{est:timecutoff}.
\end{proof}

The estimates given in the following lemma are useful in estimating the nonlinear terms.
\begin{lem}\label{lem:Str}
\begin{enumerate}
\item[(i)] For $s\in \mathbf{R}$, $0<T\leq 1$ and $0<\varepsilon \ll 1$, we have
\begin{align}
\| \psi_TF\|_{N^s_0}\lesssim _{\varepsilon}T^{\varepsilon}\| J^sF\|_{L^{1+\varepsilon}_tL^2_x(\mathbf{R}\times \mathbf{T})}. \label{est:Str0}
\end{align}
\item[(ii)] For $s\in \mathbf{R}$, $0<T\leq 1$ and $0<\varepsilon \ll 1$, we have
\begin{align}
&\| J^su\|_{L^4_{t,x}(\mathbf{R}\times \mathbf{T})}\lesssim \| u\|_{X^{s,\frac38}},\label{est:L4} \\
&\| \psi_TF\|_{N^s_0}\lesssim _{\varepsilon}T^{\frac18-\varepsilon}\| J^sF\|_{L^{\frac{4}{3}}_{t,x}(\mathbf{R}\times \mathbf{T})}.\label{est:Str1}
\end{align}
\item[(iii)] For $s\in \mathbf{R}$, $0<T\leq 1$, $0<\varepsilon \ll 1$ and $\mu >0$, we have
\begin{align}
&\| J^{s+\frac18}u\|_{L^4_{t,x}(\mathbf{R}\times \mathbf{T})}\lesssim _{\mu} \| u\|_{X^{s,\frac12}_\mu},\label{est:L4mu} \\
&\| \psi_TF\|_{N^s_\mu}\lesssim _{\varepsilon,\mu}T^{\frac{\varepsilon}{2}}\| J^{s-\frac18+\varepsilon}F\|_{L^{\frac{4}{3}}_{t,x}(\mathbf{R}\times \mathbf{T})}.\label{est:Str1mu}
\end{align}
\end{enumerate}
\end{lem}
\begin{proof}
We may assume $s=0$.

(i) By the H\"older, Young, Hausdorff-Young inequalities in $\tau$, we have
\begin{align*}
&\big\| \langle \tau+n^2\rangle^{-\frac12}(\hat{\psi}_T*_\tau\tilde{F})\big\|_{\ell ^2_nL^2_\tau}+\big\| \langle \tau+n^2\rangle^{-1}(\hat{\psi}_T*_\tau\tilde{F})\big\|_{\ell ^2_nL^1_\tau}\\
&\quad \lesssim _{\varepsilon}\| \hat{\psi}_T*_\tau\tilde{F}\|_{\ell^2_nL^{\frac{1+\varepsilon}{1+\varepsilon-\varepsilon^2}}_\tau}\leq \| \hat{\psi}_T\|_{L^{\frac{1}{1-\varepsilon}}_\tau}\| \tilde{F}\|_{\ell ^2_nL^{\frac{1+\varepsilon}{\varepsilon}}_\tau}.\\
&\quad \lesssim _{\varepsilon}T^{\varepsilon}\| F\|_{L^{1+\varepsilon}_tL^2_x}.
\end{align*}

(ii) \eqref{est:L4} is the well-known $L^4$ Strichartz estimate due to Bourgain.
In view of the dual estimate of \eqref{est:L4}, \eqref{est:Str1} follows from Lemma~\ref{lem:linear2}.

(iii) \eqref{est:L4mu} is a consequence of \eqref{est:L4} and the inequality $\langle i(\tau +n^2)+\mu |n|\rangle \gtrsim_\mu \langle n\rangle$.
Similarly, we have $\| J^{\frac18-\varepsilon}u\|_{L^4_{t,x}}\lesssim_\mu \| u\|_{X^{0,\frac12-\varepsilon}_\mu}$, and by duality, $\| F\|_{X^{0,-\frac12+\varepsilon}_\mu}\lesssim _\mu \| J^{-\frac18+\varepsilon}F\|_{L^{\frac43}_{t,x}}$.
Then, \eqref{est:Str1mu} follows from Lemma~\ref{lem:linear2}.
\end{proof}
\begin{rem}
From the proof, it is clear that the estimates \eqref{est:Str0}, \eqref{est:Str1}, \eqref{est:Str1mu} hold if $\psi_TF$ on the left-hand side is replaced with $\mathcal{F}^{-1}(|\hat{\psi}_T|*_\tau |\tilde{F}|)$.
In the proof of Lemma~\ref{lem:est-v} below, we will use these estimates in this form.
\end{rem}

\subsection{Estimates related to the gauge transformation}
\label{subsec:gaugeest}

In this subsection, we show estimates of the gauge transformation \eqref{def:vpm}.
\begin{lem}\label{lem:ubyv}
For $s>\frac14$ and $\mu >0$, we have
\[ \| gw\| _{Z^s_\mu}\lesssim _{s,\mu}\| g\| _{G^s}\| w\|_{Z^s_\mu}.\]
In particular, for any $T>0$, we have
\[ \| gw\| _{Z^s_\mu(T)}\lesssim _{s,\mu}\| g\| _{G^s(T)}\| w\|_{Z^s_\mu(T)}.\]
\end{lem}
\begin{proof}
The claim is shown by a standard argument using the Young and the H\"older inequalities on the Fourier side.
To estimate the $X^{s,\frac12}_\mu$ norm, we notice the inequality
\[ \langle i(\tau +n^2)+\mu |n|\rangle ^{\frac12}\lesssim _\mu \langle \tau_1\rangle ^{\frac12}+\langle \tau_2+n_2^2\rangle ^{\frac12}+\big( \langle n_1\rangle ^{\frac12}+\langle n_2\rangle ^{\frac12}\big) \langle n_1\rangle ^{\frac12} \]
under the relation $(\tau ,n)=(\tau_1,n_1)+(\tau _2,n_2)$.
This implies 
\begin{align*}
\| gw\|_{X^{s,\frac12}_\mu}&\lesssim _{s,\mu} \| (J^sQ^{\frac12}g)w\| _{L^2_{t,x}}+\|(Q^{\frac12}g)(J^sw)\|_{L^2_{t,x}}+ \| (J^sg)(J^s\Lambda^{\frac12}w)\| _{L^2_{t,x}}\\
&\qquad +\| (J^{s+1}g)w\|_{L^2_{t,x}}+\| (J^{\frac12}g)(J^{s+\frac12}w)\|_{L^2_{t,x}},
\end{align*}
where $J^s,Q^b,\Lambda^b$ are spacetime Fourier multipliers symbols of which are $\langle n\rangle ^s$, $\langle \tau \rangle ^b$ and $\langle \tau+n^2\rangle ^b$ respectively, and we assume that $\tilde{g},\tilde{w}\geq 0$.
The right-hand side of the above inequality is then bounded by
\begin{align*}
&\| \mathcal{F}[J^sQ^{\frac12}g]\|_{\ell^{\frac43+}_nL^2_\tau}\| \tilde{w}\|_{\ell^{\frac43-}L^1}+\| \mathcal{F}[Q^{\frac12}g]\|_{\ell^1L^2}\| \mathcal{F}[J^sw]\|_{\ell^2L^1}+\| \mathcal{F}[J^sg]\|_{\ell ^1L^1}\| \mathcal{F}[J^s\Lambda^{\frac12}w]\|_{\ell^2L^2}\\
&\quad +\| \mathcal{F}[J^{s+1}g]\| _{\ell ^2L^{\frac43}}\| \tilde{w}\|_{\ell^1L^{\frac43}}+\| \mathcal{F}[J^{\frac12}g]\| _{\ell^1L^1}\| \mathcal{F}[J^{s+\frac12}w]\|_{\ell ^2L^2}\\
&\lesssim \| \mathcal{F}[J^{s+\frac14}Q^{\frac12}g]\|_{\ell^2L^2}\| w\|_{Y^{\frac14+,0}}+\| \mathcal{F}[J^{\frac12+}Q^{\frac12}g]\|_{\ell^2L^2}\| w\|_{Y^{s,0}}+\| \mathcal{F}[J^{s+\frac12+}g]\|_{\ell ^2L^1}\| w\|_{X^{s,\frac12}}\\
&\qquad +\| \mathcal{F}[J^{s+1}g]\| _{\ell ^2L^{\frac43}}\| w\|_{Y^{\frac14+,0}}^{\frac12}\| w\|_{X^{\frac34+,0}}^{\frac12}+\| \mathcal{F}[J^{1+}g]\| _{\ell^2L^1}\| w\|_{X^{s+\frac12,0}}\\
&\lesssim _{s,\mu}\| g\|_{G^s}\| w\|_{Z^s_\mu}\qquad (s>\tfrac14).
\end{align*}
The $Y^{s,0}$ norm is estimated similarly, but more easily.
\end{proof}

\begin{lem}\label{lem:gauge}
For $s>\frac14$, there exists $s'<s$ such that for $\mu >0$ we have
\begin{align*}
\| e^{c\partial_x^{-1}P^\pm_{\alpha,\beta}(|u|^2)}-1\| _{G^s(T)}&\lesssim_{s,\mu} e^{C_{s,\mu}\| u\|_{Z^{s'}_\mu(T)}^2}\| u\|_{Z^{s'}_\mu(T)}^2,\\
\| e^{c\partial_x^{-1}P^\pm_{\alpha,\beta}(|u_1|^2)}-e^{c\partial_x^{-1}P^\pm_{\alpha,\beta}(|u_2|^2)}\| _{G^s(T)}&\lesssim_{s,\mu} e^{C_{s,\mu}\| u_j\|_{Z^{s'}_\mu(T)}^2}\| u_j\|_{Z^{s'}_\mu(T)}\| u_1-u_2\|_{Z^{s'}_\mu(T)},
\end{align*}
where $\| u_j\|_{Z^{s'}_\mu(T)}$ means $\| u_1\|_{Z^{s'}_\mu(T)}+\| u_2\|_{Z^{s'}_\mu(T)}$.
\end{lem}
\begin{proof}
In view of the Taylor series expansion
\[ e^{c\partial_x^{-1}P^\pm_{\alpha,\beta}(|u|^2)}-1=\sum_{m=1}^\infty \frac{1}{m!}\big[ c\partial_x^{-1}P^\pm_{\alpha,\beta}(|u|^2)\big] ^m \]
and 
\begin{align*}
&e^{c\partial_x^{-1}P^\pm_{\alpha,\beta}(|u_1|^2)}-e^{c\partial_x^{-1}P^\pm_{\alpha,\beta}(|u_2|^2)}\\
&\quad =\sum _{m=1}^\infty \frac{1}{m!}\Big( \big[ c\partial_x^{-1}P^\pm_{\alpha,\beta}(|u_1|^2)\big] ^m-\big[ c\partial_x^{-1}P^\pm_{\alpha,\beta}(|u_2|^2)\big] ^m\Big) \\
&\quad =\Big( c\partial_x^{-1}P^\pm_{\alpha,\beta}((u_1-u_2)\bar{u}_1)+c\partial_x^{-1}P^\pm_{\alpha,\beta}(u_2(\overline{u_1-u_2}))\Big) \\
&\qquad\qquad \times \sum_{m=1}^\infty \frac{1}{m!}\sum _{l=0}^{m-1}\big[ c\partial_x^{-1}P^\pm_{\alpha,\beta}(|u_1|^2)\big] ^l\big[ c\partial_x^{-1}P^\pm_{\alpha,\beta}(|u_2|^2)\big] ^{m-1-l},
\end{align*}
the claimed estimates follow if we prove
\begin{align}
\| f_1f_2\| _{G^s(T)}&\lesssim_s \| f_1\| _{G^s(T)}\| f_2\|_{G^s(T)} & &(s>-\tfrac14), \label{est:G1}\\
\| \partial_x^{-1}P^\pm_{\alpha,\beta}(u_1\bar{u}_2)\|_{G^s(T)}&\lesssim_{s,\mu} \| u_1\|_{Z^{s'}_\mu(T)}\| u_2\|_{Z^{s'}_\mu(T)}& &(s>\tfrac14,~\mu >0).\label{est:G2}
\end{align}
By considering suitable extensions of functions on $[0,T]\times \mathbf{T}$, it suffices to prove these estimates in $G^s$ and $Z^{s'}_\mu$.

\underline{Proof of \eqref{est:G1}.}
We first observe that $\| \tilde{f}\|_{\ell^1_nL^1_\tau}\lesssim \| \langle n\rangle^{\frac12+}\tilde{f}\|_{\ell^2_nL^1_\tau}\lesssim \| f\|_{G^{-\frac14+}}$.
Then, we use the Young inequality in $(n,\tau)$ as
\begin{align*}
\| \langle n\rangle ^{s+1}(\tilde{f}_1*\tilde{f}_2)\| _{\ell ^2L^{\frac43}}&\lesssim \| \langle n\rangle^{s+1}\tilde{f}_1\|_{\ell^2L^{\frac43}}\| \tilde{f}_2\|_{\ell^1L^1}+\| \tilde{f}_1\|_{\ell^1L^1}\| \langle n\rangle ^{s+1}\tilde{f}_2\|_{\ell^2L^{\frac43}} \\
&\lesssim \| f_1\|_{G^s}\| f_2\|_{G^s}\qquad (s>-\tfrac14),
\end{align*}
and similarly for $\| \langle n\rangle^{s+\frac34}(\tilde{f}_1*\tilde{f}_2)\|_{\ell^2L^1}$.
For the last norm containing $\langle \tau \rangle^{\frac12}$, we need to consider the case in which $\langle n\rangle ^{s+\frac14}$ and $\langle \tau\rangle^{\frac12}$ move onto different functions.
Actually, it can be treated as follows:
\begin{align*}
\| (\langle n\rangle^{s+\frac14}\tilde{f}_1)*(\langle \tau \rangle^{\frac12}\tilde{f}_2)\|_{\ell^2L^2}&\leq \| \langle n\rangle^{s+\frac14}\tilde{f}_1\|_{\ell^{1+}L^1}\| \langle \tau \rangle^{\frac12}\tilde{f}_2\|_{\ell^{2-}L^2}\\
&\lesssim \| \langle n\rangle^{s+\frac34}\tilde{f}_1\|_{\ell^2L^1}\| \langle n\rangle ^{0+}\langle \tau \rangle^{\frac12}\tilde{f}_2\|_{\ell^2L^2}\\
&\lesssim \| f_1\|_{G^s}\| f_2\|_{G^{-\frac14+}}.
\end{align*}

\underline{Proof of \eqref{est:G2}.}
The first two terms in the $G^s$ norm are treated in a way similar to the above:
\begin{align*}
\| \langle n\rangle ^s&(\tilde{u}_1*\bar{\tilde{u}}_2(-\cdot ))\| _{\ell^2L^{\frac43}}
\lesssim \| \langle n\rangle^s\tilde{u}_1\|_{\ell^2L^{1+}}\| \tilde{u}_2\|_{\ell ^1L^{\frac43-}}+\| \tilde{u}_1\|_{\ell ^1L^{\frac43-}}\| \langle n\rangle^s\tilde{u}_2\|_{\ell^2L^{1+}}\\
&\lesssim \| \langle n\rangle^s\tilde{u}_1\|_{\ell^2L^{1+}}\| \langle n\rangle^{\frac12+}\tilde{u}_2\|_{\ell ^2L^{\frac43-}}+\| \langle n\rangle ^{\frac12+}\tilde{u}_1\|_{\ell ^2L^{\frac43-}}\| \langle n\rangle^s\tilde{u}_2\|_{\ell^2L^{1+}},\\
\| \langle n\rangle ^{s-\frac14}&(\tilde{u}_1*\bar{\tilde{u}}_2(-\cdot ))\| _{\ell^2L^1}
\lesssim \| \langle n\rangle^{s-\frac14}\tilde{u}_1\|_{\ell^{\frac43+}L^1}\| \tilde{u}_2\|_{\ell ^{\frac43-}L^1}+\| \tilde{u}_1\|_{\ell ^{\frac43-}L^1}\| \langle n\rangle^{s-\frac14}\tilde{u}_2\|_{\ell^{\frac43+}L^1}\\
&\lesssim \| \langle n\rangle^{s-}\tilde{u}_1\|_{\ell^2L^1}\| \langle n\rangle ^{\frac14+}\tilde{u}_2\|_{\ell ^2L^1}+\| \langle n\rangle ^{\frac14+}\tilde{u}_1\|_{\ell ^2L^1}\| \langle n\rangle^{s-}\tilde{u}_2\|_{\ell^2L^1}\\
&\lesssim _s\| u_1\|_{Y^{s-,0}}\| u_2\|_{Y^{s-,0}}\qquad (s>\tfrac14).
\end{align*}
The former estimate is finished by interpolation inequalities:
\begin{align*}
\| \langle n\rangle ^s\tilde{u}\|_{\ell^2L^{1+}}&\leq \| \langle n\rangle ^{s+\frac12-}\tilde{u}\|_{\ell^2L^2}^{0+}\| \langle n\rangle ^{s-}\tilde{u}\|_{\ell^2L^1}^{1-}\lesssim _{\mu}\| u\|_{Z^{s-}_\mu},\\
\| \langle n\rangle ^{\frac12+}\tilde{u}\|_{\ell^2L^{\frac43-}}&\lesssim \| \langle n\rangle^{\frac34+}\tilde{u}\|_{\ell ^2L^2}^{\frac12-}\| \langle n\rangle ^{\frac14+}\tilde{u}\|_{\ell^2L^1}^{\frac12+}\lesssim _\mu \| u\|_{Z^{\frac14+}_\mu}.
\end{align*}
To estimate the last term in the $G^s$ norm, we first see the inequality
\[ \langle \tau \rangle ^{\frac12}\lesssim \langle \tau_1+n_1^2\rangle ^{\frac12}+\langle \tau_2+n_2^2\rangle ^{\frac12}+\langle n\rangle ^{\frac12}\langle n_1+n_2\rangle^{\frac12} \]
under the relation $(\tau,n)=(\tau_1,n_1)-(\tau_2,n_2)$.
For the contribution from $\langle \tau_1+n_1^2\rangle ^{\frac12}$ (and similarly for $\langle \tau_2+n_2^2\rangle ^{\frac12}$), we proceed as follows:
\begin{align*}
&\| \langle n\rangle ^{s-\frac34}\big[ (\langle \tau +n^2\rangle ^{\frac12}\tilde{u}_1)*\tilde{u}_2(-\cdot )\big] \|_{\ell^2L^2}\\
&\quad \lesssim \| \langle n\rangle ^{s-\frac14+}\big[ (\langle \tau +n^2\rangle ^{\frac12}\tilde{u}_1)*\tilde{u}_2(-\cdot )\big] \|_{\ell^\infty L^2}\\
&\quad \lesssim \| u_1\|_{X^{s-\frac14+,\frac12}}\| u_2\|_{Y^{0,0}}+\| u_1\|_{X^{0,\frac12}}\| u_2\|_{Y^{s-\frac14+,0}}\qquad (s>\tfrac14)\\
&\quad \lesssim \| u_1\|_{Z^{s-\frac14+}_0}\| u_2\|_{Z^{s-\frac14+}_0}.
\end{align*}
For the last case, by symmetry we may assume $|n_1|\geq |n_2|$, then we have
\begin{align*}
\| \langle n\rangle ^{s-\frac34}\langle n\rangle ^{\frac12}\big[ (\langle n\rangle ^{\frac12}\tilde{u}_1)*\tilde{u}_2(-\cdot )\big] \|_{\ell^2L^2}&\lesssim \| \langle n\rangle ^{s+\frac14}\tilde{u}_1\|_{\ell^2L^{\frac43+}}\| \tilde{u}_2\|_{\ell ^1L^{\frac43-}}\\
&\lesssim \| \langle n\rangle ^{s+\frac14}\tilde{u}_1\|_{\ell^2L^{\frac43+}}\| \langle n\rangle ^{\frac12+}\tilde{u}_2\|_{\ell ^2L^{\frac43-}}.
\end{align*}
Applying the latter interpolation inequality in the above and another one
\[ \| \langle n\rangle ^{s+\frac14}\tilde{u}\|_{\ell^2L^{\frac43+}}\leq \| \langle n\rangle ^{s+\frac12-}\tilde{u}\|_{\ell^2L^2}^{\frac12+}\| \langle n\rangle ^{s-}\tilde{u}\|_{\ell^2L^1}^{\frac12-}\lesssim _\mu \| u\|_{Z^{s-}_\mu},\]
we obtain the claim.
This finishes the proof of \eqref{est:G2}.
\end{proof}

Consequently, from the last two lemmas, we obtain the following estimates on the gauge transforms $v_\pm$ in terms of $u$.
\begin{cor}\label{cor:vbyu}
Let $s>\frac14$, $\mu >0$, $0<T\leq 1$ and $u,u_1,u_2\in Z^s_\mu(T)$.
Define $v_\pm$, $v_{1,\pm}$ and $v_{2,\pm}$ respectively by \eqref{def:vpm}.
Then, there exists $C=C(s,\mu)>0$ such that
\begin{gather*}
\| v_\pm\|_{Z^s_\mu(T)}\leq C\big( 1+e^{C\| u\|_{Z^s_\mu(T)}^2}\| u\|_{Z^s_\mu(T)}^2\big) \| u\|_{Z^s_\mu(T)},\\
\| v_{1,\pm}-v_{2,\pm}\|_{Z^s_\mu(T)}\leq C\big( 1+e^{C\| u_j\|_{Z^s_\mu(T)}^2}\| u_j\|_{Z^s_\mu(T)}^2\big) \| u_1-u_2\|_{Z^s_\mu(T)}.
\end{gather*}
\end{cor}

\subsection{Nonlinear estimates}
\label{subsec:nonlest}

The nonlinear terms in the equation \eqref{eq:u'} for $u$ is treated by the following lemma:
\begin{lem}\label{lem:est-u}
Let $s>\frac14$, $\mu >0$, and $0<T\leq 1$.
We have the following estimates.
\begin{enumerate}
\item[(i)] There exists $\delta >0$ such that for any solutions $u,u_1,u_2\in Z^s_\mu (T)$ to \eqref{eq:u'} with the initial data $\phi ,\phi _1,\phi_2$, respectively, and any $w\in Z^\sigma_\mu (T)$ with any $\sigma \in \mathbf{R}$, we have
\begin{gather*}
\| \mathcal{N}_0[u,\phi ;w]\|_{N^\sigma_\mu (T)}\lesssim_{s,\mu} T^{\delta}\| u\|_{Z^s_\mu (T)}^4\|w\|_{Z^\sigma_\mu (T)},\\
\begin{aligned}
&\| \mathcal{N}_0[u_1,u_2;w]\|_{N^\sigma_\mu (T)}\\
&\quad \lesssim_{s,\mu} \Big( \| \phi_j\| _{L^2}\| \phi_1-\phi_2\|_{L^2}+T^{\delta}\| u_j\|_{Z^s_\mu (T)}^3\| u_1-u_2\|_{Z^s_\mu (T)}\Big) \|w\|_{Z^\sigma_\mu (T)}.
\end{aligned}
\end{gather*}
\item[(ii)] For any $u_1,u_2,u_3\in Z^s_\mu (T)$,
\[ \| \mathcal{N}_u[u_1,u_2,u_3]\|_{N^s_\mu (T)}\lesssim_{s,\mu} \| u_1\|_{Z^s_\mu(T)}\| u_2\|_{Z^s_\mu(T)}\| u_3\|_{Z^s_\mu(T)}.\]
Moreover, for any $s'<s$ there exists $\delta >0$ such that
\[ \| \mathcal{N}_u[u_1,u_2,u_3]\|_{N^{s'}_\mu (T)}\lesssim_{s,s',\mu} T^{\delta}\| u_1\|_{Z^s_\mu(T)}\| u_2\|_{Z^s_\mu(T)}\| u_3\|_{Z^s_\mu(T)}.\]
\end{enumerate}
\end{lem}

\begin{proof}
(i) First, we observe that
\begin{gather}\label{est:N0-0}
\| f(t)D_xw\|_{N^\sigma_\mu(T)}\lesssim _\mu \| f\|_{L^\infty_t([0,T])} \| w\|_{Z^\sigma_\mu(T)}
\end{gather}
for any bounded function $f$ on $[0,T]$ and $w\in Z^\sigma_\mu(T)$ with $\sigma \in \mathbf{R}$, $\mu >0$, $T>0$ (see \cite[proof of Proposition~3.2]{KT22} for a proof).
Then, it suffices to prove
\begin{gather*}
\sup _{t\in [0,T]}\big| \| u(t)\|_{L^2}^2-\| \phi \|_{L^2}^2\big| \lesssim _{s,\mu}T^{\delta}\| u\|_{Z^s_\mu(T)}^4,\\
\sup _{t\in [0,T]}\big| \| u_1(t)\|_{L^2}^2-\| u_2(t) \|_{L^2}^2\big| \lesssim _{s,\mu} \| \phi_j\| _{L^2}\| \phi_1-\phi_2\|_{L^2}+T^{\delta}\| u_j\|_{Z^s_\mu (T)}^3\| u_1-u_2\|_{Z^s_\mu (T)}
\end{gather*}
for $s>\frac14$ and solutions $u,u_1,u_2\in Z^s_\mu(T)$ of \eqref{eq:u'}.
These estimates follow from the $L^2$ energy equality
\begin{gather}\label{L2eq}
\| u(t)\|_{L^2}^2=\| \phi \|_{L^2}^2+\beta \int _0^t\int _{\mathbf{T}}\big| D_x^{\frac12}(|u(t')|^2)\big| ^2\,dx\,dt',\quad t\in [0,T],
\end{gather}
which is valid for any solution $u\in Z^s_\mu(T)\subset C_TH^s\cap L^2_TH^{s+\frac12}$ to \eqref{eq:u'} and \eqref{ic} with $s>\frac14$.%
\footnote{Specifically, this can be shown by calculating the corresponding equality for $P_{\leq N}u$ and letting $N\to \infty$.
The limiting procedure is verified under the regularity $C_TH^{\frac14+}\cap L^2_TH^{\frac34+}$.}
Indeed, from the H\"older, Sobolev inequalities and interpolation, we see that
\begin{align*}
&\int _0^T\int _{\mathbf{T}}\big| D_x^{\frac12}(|u(t')|^2)\big| ^2\,dx\,dt'\lesssim \int _0^T\| u(t')\| _{L^4}^2\| D_x^{\frac12}u(t')\| _{L^4}^2\,dt' \\
&\quad \lesssim T^{0+}\| u\|_{L^{\infty}_TH^{\frac14}}^2\| u\|_{L^{2+}_TH^{\frac34}}^2\lesssim T^{0+}\| u\|_{L^{\infty}_TH^{\frac14}}^{2+}\| u\|_{L^2_TH^{\frac34+}}^{2-}\lesssim_{s,\mu} T^{0+}\| u\|_{Z^{\frac14+}_\mu (T)}^4,
\end{align*}
and similarly, that
\begin{align*}
&\int _0^T\bigg| \int _{\mathbf{T}}\big| D_x^{\frac12}(|u_1(t')|^2)\big| ^2\,dx-\int _{\mathbf{T}}\big| D_x^{\frac12}(|u_2(t')|^2)\big| ^2\,dx\bigg| \,dt'\\
&\quad \lesssim _{s,\mu}T^{0+}\| u_j\|_{Z^{\frac14+}_\mu (T)}^3\| u_1-u_2\|_{Z^{\frac14+}_\mu (T)}.
\end{align*}

(ii) We have proved the first estimate in \cite[Proposition~3.3, Lemma~3.4]{KT22}.
For the second one, it suffices to follow the proof of the first estimate, noticing that the loss of derivative is now strictly lower than the first order.
\end{proof}

For the gauge transformed equation \eqref{eq:vpm'}, we have better estimates with the factor $T^{0+}$:
\begin{lem}\label{lem:est-v}
Let $s>\frac14$ and $\mu >0$.
Then, there exist $\delta =\delta (s)>0$ and $C=C(s,\mu )>0$ such that for any $0<T\leq 1$ we have
\begin{gather*}
\| \mathcal{N}_v^\pm [u]\| _{N^s_\mu (T)}\leq CT^{\delta}\big( 1+e^{C\| u\|_{Z^s_\mu(T)}^2}\| u\|_{Z^s_\mu(T)}^2\big) \big( \| u\|_{Z^s_\mu(T)}^3+\| u\|_{Z^s_\mu(T)}^5\big) ,\\
\begin{aligned}
&\| \mathcal{N}_v^\pm [u_1]-\mathcal{N}_v^\pm [u_2]\| _{N^s_\mu (T)}\\
&\leq CT^{\delta}\big( 1+e^{C\| u_j\|_{Z^s_\mu(T)}^2}\| u_j\|_{Z^s_\mu(T)}^2\big) \big( \| u_j\|_{Z^s_\mu(T)}^2+\| u_j\|_{Z^s_\mu(T)}^4\big) \| u_1-u_2\|_{Z^s_\mu (T)}.
\end{aligned}
\end{gather*}
\end{lem}

\begin{proof}
In principle, the case $s>\frac12$ can be shown by a standard argument using Fourier restriction norm, and the parabolicity allows us to push down the regularity to $s>\frac14$, keeping the small factor $T^{\delta}$.
Multiplication by the factor $e^{\rho ^\pm [u]}$ will cause no trouble, with the aid of Lemmas~\ref{lem:ubyv} and \ref{lem:gauge}. 
We now describe how to derive the estimates of \eqref{nonl:ubarx}--\eqref{nonl:quintic}.

{\bf Estimate of \eqref{nonl:ubarx}}.
Let us consider the last term in \eqref{nonl:ubarx}, but the same argument is applicable to the other terms.
In view of Lemma~\ref{lem:gauge}, it suffices to prove
\[ \big\| P^\pm_{\alpha,\beta}\big( u_1\partial_x\bar{u}_2\big) P_\pm u_3\big\|_{N^s_\mu(T)}\lesssim _{s,\mu}T^\delta \| u_1\|_{Z^s_\mu(T)}\| u_2\|_{Z^s_\mu(T)}\| u_3\|_{Z^s_\mu(T)}\]
and
\[ \big\| P^\pm_{\alpha,\beta}\big( u_1\partial_x\bar{u}_2\big) P_\pm u_3\cdot g\big\|_{N^s_\mu(T)}\lesssim _{s,\mu}T^\delta \| u_1\|_{Z^s_\mu(T)}\| u_2\|_{Z^s_\mu(T)}\| u_3\|_{Z^s_\mu(T)}\| g\|_{G^s(T)},\]
where $g$ is a function of the form $e^{\rho^\pm [u]}-1$ or $e^{\rho^\pm [u]}-e^{\rho^\pm [u']}$ with $u,u'\in Z^s_\mu(T)$.
We then take suitable extensions and consider the estimates in $N^s_\mu$, $Z^s_\mu$ and $G^s$.
We also insert $\psi_T$'s to derive $T^\delta$ by Lemma~\ref{lem:linear2}, though we will occasionally omit writing $\psi_T$.
In addition, any function on $\mathbf{R}_t\times \mathbf{T}_x$ may be assumed to have non-negative Fourier transform: $\tilde{u}(\tau,n)\geq 0$.
Now, the claim is reduced to
\begin{gather}\label{est:ubarx}
\begin{aligned}
&\bigg\| \mathcal{F}^{-1}\bigg[ |\hat{\psi}_T|*_\tau \int_{\tau=\tau_1-\tau_2+\tau_3}\sum _{n=n_1-n_2+n_3}|n_2|\tilde{u}_1(\zeta_1)\tilde{u}_2(\zeta_2)\tilde{u}_3(\zeta_3)\,d\tau_1\,d\tau_2\bigg] \bigg\|_{N^s_\mu}\\
&\quad \lesssim _{s,\mu} T^{\delta} \| u_1\|_{Z^s_\mu}\| u_2\|_{Z^s_\mu}\| u_3\|_{Z^s_\mu}
\end{aligned}
\end{gather}
and
\[ \begin{aligned}
&\bigg\| \mathcal{F}^{-1}\bigg[ |\hat{\psi}_T|*_\tau \int_{\tau=\tau_1-\tau_2+\tau_3+\tau_4}\sum _{n=n_1-n_2+n_3+n_4}|n_2|\tilde{u}_1(\zeta_1)\tilde{u}_2(\zeta_2)\tilde{u}_3(\zeta_3)\tilde{g}(\zeta_4)\,d\tau_1\,d\tau_2\,d\tau_3\bigg] \bigg\|_{N^s_\mu}\\
&\quad \lesssim_{s,\mu} T^{\delta}\| u_1\|_{Z^s_\mu}\| u_2\|_{Z^s_\mu}\| u_3\|_{Z^s_\mu}\| g\|_{G^s},
\end{aligned}
\]
where $\zeta_j:=(\tau_j,n_j)$.
Note that the latter estimate is reduced to the former by considering $u_3g$ as $u_3$ and applying Lemma~\ref{lem:ubyv}.
Hence, we shall establish \eqref{est:ubarx} below.
We may assume $|n_1|\geq |n_3|$.
Recall the resonance inequality: writing $\lambda :=\tau+n^2$ and $\lambda _j:=\tau_j+n_j^2$, we have
\begin{gather}\label{resonanceineq}
\lambda -\lambda_1+\lambda_2-\lambda_3=n^2-n_1^2+n_2^2-n_3^2=2(n_2-n_1)(n_2-n_3)=:\Phi 
\end{gather}
under the relation $(\tau ,n)=(\tau_1,n_1)-(\tau_2,n_2)+(\tau_3,n_3)$.

\underline{(I) $|n_2|\gg |n_1|$.}
In this case, we have $|\Phi |\sim |n_2|^2$, and from \eqref{resonanceineq},
\[ \langle n\rangle ^s|n_2|\lesssim \langle n_2\rangle ^s\big( \langle \lambda_1\rangle ^{\frac12}+\langle \lambda_2\rangle ^{\frac12}+\langle \lambda_3\rangle ^{\frac12}\big) +\langle n_2\rangle ^{s+2\varepsilon} \langle \lambda \rangle ^{\frac12-\varepsilon}. \]
Using the estimate \eqref{est:Str1} for the first three cases and the H\"older, Young inequalities in $\tau$ for the last case, the left-hand side of \eqref{est:ubarx} is bounded by
\begin{align*}
&T^{\frac18-\varepsilon} \Big( \| (\Lambda^{\frac12}u_1)(\overline{J^su_2})u_3\| _{L^{\frac43}_{t,x}}+\| u_1(\overline{J^s\Lambda^{\frac12}u_2})u_3\|_{L^{\frac43}_{t,x}}+\| u_1(\overline{J^su_2})(\Lambda^{\frac12}u_3)\|_{L^{\frac43}_{t,x}}\Big) \\
&+\| (J^{\frac14}u_1)(\overline{J^{s+2\varepsilon}u_2})(J^{-\frac14}u_3)\|_{L^2_{t,x}},
\end{align*}
where $J^s,\Lambda^b$ are the same as in the proof of Lemma~\ref{lem:ubyv}.
The first term is estimated by
\[ T^{\frac18-\varepsilon}\| \Lambda^{\frac12}u_1\|_{L^2_tL^4_x}\| J^su_2\|_{L^\infty_tL^2_x}\| u_3\|_{L^4_tL^\infty_x}\lesssim T^{\frac18-\varepsilon}\| u_1\|_{X^{\frac14+,\frac12}}\| u_2\|_{Y^{s,0}}\| u_3\|_{X^{\frac14+,\frac38}},\]
where we have used \eqref{est:L4} for $u_3$.
The third term is treated similarly.
For the second term, we have the bound
\[ T^{\frac18-\varepsilon}\| u_1\|_{L^\infty_tL^4_x}\| J^s\Lambda^{\frac12}u_2\|_{L^2_{t,x}}\| u_3\|_{L^4_tL^\infty_x}\lesssim T^{\frac18-\varepsilon}\| u_1\|_{Y^{\frac14+,0}}\| u_2\|_{X^{s,\frac12}}\| u_3\|_{X^{\frac14+,\frac38}}.\]
The last term is bounded by
\[ \| J^{\frac14}u_1\|_{L^4_{t,x}}\| J^{s+2\varepsilon}u_2\|_{L^4_{t,x}}\| J^{-\frac14}u_3\|_{L^\infty_{t,x}}\lesssim \| u_1\|_{X^{\frac14,\frac38}}\| u_2\|_{X^{s+2\varepsilon,\frac38}}\| u_3\|_{Y^{\frac14+,0}}.\]
Now, we take $\varepsilon \leq \frac{1}{16}$ and use parabolicity (the effect of $\mu >0$) to bound the norm of $u_2$ as
\[ \| u_2\|_{X^{s+\frac18,\frac38}}\leq \| u_2\|_{X^{s,\frac12}}^{\frac34}\| u_2\|_{X^{s+\frac12,0}}^{\frac14}\lesssim _\mu \| u_2\|_{X^{s,\frac12}_\mu}.\]
We also derive a factor $T^{\frac18}$ from the norm of $u_1$ by using Lemma~\ref{lem:linear2}.
Gathering the estimates so far, we obtain in this case that
\[ \text{LHS of \eqref{est:ubarx}}\lesssim_{s,\mu} T^{\frac18-}\| u_1\|_{Z^s_0}\| u_2\|_{Z^s_\mu}\| u_3\|_{Z^s_0}.\]

\underline{(II) $|n_1|\gtrsim |n_2|\gg |n_3|$.}
We need further division of cases as follows.

(i) $|n_1-n_2|\sim |n_1|$.
Since $|\Phi|\sim \langle n_1\rangle \langle n_2\rangle \gtrsim \langle n_2\rangle^2$, from \eqref{resonanceineq} we have
\[ \langle n\rangle^s|n_2|\lesssim \langle n_1\rangle ^s\big( \langle \lambda_1\rangle ^{\frac12}+\langle \lambda_2\rangle ^{\frac12}+\langle \lambda_3\rangle ^{\frac12}\big) +\langle n_1\rangle ^{s+2\varepsilon} \langle \lambda \rangle ^{\frac12-\varepsilon}. \]
Then, the proof is parallel to the previous case (I), with the roles of $u_1$ and $u_2$ exchanged.

(ii) $n_1-n_2=0$.
By Lemma~\ref{lem:linear2}, we have
\[ \| \mathcal{F}^{-1}(|\hat{\psi}_T|*_\tau |\tilde{F}|)\| _{N^s_0}\lesssim T^{\frac12-\varepsilon}\| F\|_{L^2_{t,x}}.\]
We thus estimate $\| P_0(u_1\overline{Ju_2})J^su_3\|_{L^2_{t,x}}$.
This is bounded by
\[ \| J^{\frac34}u_1\|_{L^2_{t,x}}\| J^{\frac14}u_2\|_{L^\infty_tL^2_x}\| J^su_3\|_{L^\infty_tL^2_x}\lesssim _\mu \| u_1\|_{X^{\frac14,\frac12}_\mu}\| u_2\|_{Y^{\frac14,0}}\| u_3\|_{Y^{s,0}},\]
where we have used parabolicity for $u_1$.

(iii) $|n_1|\gg |n_1-n_2|\neq 0$.
In this case, we have $|n_1|\sim |n_2|$, $|\Phi|\sim \langle n_1-n_2\rangle \langle n_2\rangle$ and 
\[ \langle n\rangle^s|n_2|\lesssim \frac{\langle n_1\rangle ^{s+\frac12}}{\langle n_1-n_2\rangle^{\frac12}}\big( \langle \lambda_1\rangle ^{\frac12}+\langle \lambda_2\rangle ^{\frac12}+\langle \lambda_3\rangle ^{\frac12}\big) +\frac{\langle n_1\rangle ^{s+\frac12}\langle n_2\rangle ^{2\varepsilon}}{\langle n_1-n_2\rangle^{\frac12-\varepsilon}} \langle \lambda \rangle ^{\frac12-\varepsilon}. \]
Similarly to the case (I), we use \eqref{est:Str1} and bound the left-hand side of \eqref{est:ubarx} by
\begin{align*}
&T^{\frac18-\varepsilon}\Big( \| J^{-\frac12}[(J^s\Lambda^{\frac12}u_1)(\overline{J^{\frac12}u_2})]u_3\| _{L^{\frac43}_{t,x}}+\| J^{-\frac12}[(J^{s+\frac14}u_1)(\overline{J^{\frac14}\Lambda^{\frac12}u_2})]u_3\|_{L^{\frac43}_{t,x}} \\
&\qquad +\| J^{-\frac12}[(J^{s+\frac14}u_1)(\overline{J^{\frac14}u_2})](\Lambda^{\frac12}u_3)\|_{L^{\frac43}_{t,x}}\Big) +\| J^{-\frac12+\varepsilon}[(J^{s+\frac14}u_1)(\overline{J^{\frac14+2\varepsilon}u_2})]u_3\|_{L^2_{t,x}}.
\end{align*}
The first term is bounded by
\[ T^{\frac18-\varepsilon}\| (J^{s}\Lambda^{\frac12}u_1)(\overline{J^{\frac12}u_2})\|_{L^{\frac43}_tL^1_x}\| u_3\|_{L^\infty_tL^{4+}_x} \lesssim T^{\frac18-\varepsilon}\| u_1\|_{X^{s,\frac12}}\| J^{\frac12}u_2\|_{L^4_tL^2_x}\| u_3\|_{Y^{\frac14+,0}}.\]
The norm of $u_2$ is treated with the following estimate using parabolicity: 
\begin{gather}
\begin{aligned}
\| J^{\sigma +\frac14}u\|_{L^4_tL^2_x}&\leq \| J^{\sigma}u\|_{L^\infty_tL^2_x}^{\frac12}\| J^{\sigma+\frac12}u\|_{L^2_{t,x}}^{\frac12}\lesssim _\mu \| u\|_{Y^{\sigma ,0}}^{\frac12}\| u\|_{X^{\sigma,\frac12}_\mu}^{\frac12}\\
&\lesssim_\mu \| u\|_{Z^{\sigma}_\mu}\qquad (\sigma\in \mathbf{R}).
\end{aligned}\label{est:L4L2}
\end{gather}
In a similar manner, the second term is estimated by
\[ T^{\frac18-\varepsilon}\| J^{s+\frac14}u_1\|_{L^4_tL^2_x}\| J^{\frac14}\Lambda^{\frac12}u_2\|_{L^2_{t,x}}\| u_3\|_{L^\infty_tL^{4+}_x}\lesssim _\mu T^{\frac18-\varepsilon}\| u_1\|_{Z^s_\mu}\| u_2\|_{X^{\frac14,\frac12}}\| u_3\|_{Y^{\frac14+,0}}.\]
For the third term, we proceed as follows:
\[ T^{\frac18-\varepsilon}\| J^{s+\frac14}u_1\| _{L^4_tL^2_x}\| J^{\frac14}u_2\|_{L^\infty_tL^2_x}\| \Lambda^{\frac12}u_3\|_{L^2_tL^{4+}_x}\lesssim _\mu T^{\frac18-\varepsilon}\| u_1\|_{Z^s_\mu}\| u_2\|_{Y^{\frac14,0}}\| u_3\|_{X^{\frac14+,\frac12}}.\]
Finally, the last term is treated with \eqref{est:L4} and Lemma~\ref{lem:linear2}, bounded by
\[ \| J^{s+\frac14}u_1\|_{L^4_tL^2_x}\| J^{\frac14+2\varepsilon}u_2\|_{L^4_{t,x}}\| u_3\|_{L^\infty_tL^{4+}_x} \lesssim _\mu T^{\frac18}\| u_1\|_{Z^s_\mu}\| u_2\|_{Z^{\frac14+2\varepsilon,\frac12}}\| u_3\|_{Y^{\frac14+,0}}.\]

\underline{(III) $|n_3|\gtrsim |n_2|$.}
In this case, we apply \eqref{est:Str1mu} and divide the weight factor as
\[ \langle n\rangle ^{s-\frac18+}|n_2|\lesssim \langle n_1\rangle ^{s+\frac18}\langle n_2\rangle ^{\frac38+}\langle n_3\rangle ^{\frac38}.\]
Then, we apply \eqref{est:L4} to obtain the bound
\[ T^{0+}\| (J^{s+\frac18}u_1)(\overline{J^{\frac38+}u_2})J^{\frac38}u_3\|_{L^{\frac43}_{t,x}}\lesssim T^{0+}\| u_1\|_{X^{s+\frac18,\frac38}}\| u_2\|_{X^{\frac38+,\frac38}}\| u_3\| _{X^{\frac38,\frac38}}.\]
We obtain the desired estimate by applying the following variant of \eqref{est:L4L2}:
\[ \| u\|_{X^{\sigma+\frac18,\frac38}}\leq \| u\|_{X^{\sigma ,\frac12}}^{\frac34}\| u\|_{X^{\sigma+\frac12,0}}^{\frac14}\lesssim _\mu \| u\|_{X^{\sigma,\frac12}_\mu}\quad (\sigma \in \mathbf{R}).\]

{\bf Estimate of \eqref{nonl:u3x}.}
Arguing similarly to the estimate of \eqref{nonl:ubarx}, it suffices to prove
\begin{gather}
\begin{aligned}
&\bigg\| \mathcal{F}^{-1}\bigg[ |\hat{\psi}_T|*_\tau \int_{\tau=\tau_1-\tau_2+\tau_3}\sum _{\begin{smallmatrix}n=n_1-n_2+n_3\\ |n_1|\gtrsim |n_3|\end{smallmatrix}}|n_3|\tilde{u}_1(\zeta_1)\tilde{u}_2(\zeta_2)\tilde{u}_3(\zeta_3)\,d\tau_1\,d\tau_2\bigg] \bigg\|_{N^s_\mu}\\
&\quad \lesssim _{s,\mu} T^{0+} \| u_1\|_{Z^s_\mu}\| u_2\|_{Z^s_\mu}\| u_3\|_{Z^s_\mu}.
\end{aligned} \label{est:u3x}
\end{gather}
In fact, the commutator structure of these terms prevents the spatial derivative to fall onto the function of the highest frequency, and then the estimate is reduced to either \eqref{est:ubarx} or \eqref{est:u3x}.
(Note that we regard $\bar{u}_2g$ as $\bar{u}_2$ in this case.)
In the frequency region $\{ |n_3|\lesssim |n_2|\}$, the claimed estimate \eqref{est:u3x} is weaker than \eqref{est:ubarx}, and hence we may assume $|n_1|\gtrsim |n_3|\gg |n_2|$.
In this situation, we have $|\Phi|\sim \langle n_1\rangle \langle n_3\rangle \gtrsim \langle n_3\rangle ^2$, and thus by \eqref{resonanceineq}
\[ \langle n\rangle^s|n_3|\lesssim \langle n_1\rangle ^s\big( \langle \lambda_1\rangle ^{\frac12}+\langle \lambda_2\rangle ^{\frac12}+\langle \lambda_3\rangle ^{\frac12}\big) +\langle n_1\rangle ^{s+2\varepsilon} \langle \lambda \rangle ^{\frac12-\varepsilon}. \]
The rest of the proof is parallel to the case (I) for the estimate of \eqref{est:ubarx} above, with the roles of $u_1$ and $u_2$ exchanged.

{\bf Estimate of \eqref{nonl:pm}.}
Due to the frequency projection $P_\pm$, the estimate is reduced to showing
\begin{gather*}
\begin{aligned}
&\bigg\| \mathcal{F}^{-1}\bigg[ |\hat{\psi}_T|*_\tau \int_{\tau=\tau_1-\tau_2+\tau_3+\tau_4}\sum _{\begin{smallmatrix}n_1-n_2=0 \\ n=n_3+n_4 \\ |n_3|<|n_4|\end{smallmatrix}}|n_4|\tilde{u}_1(\zeta_1)\tilde{u}_2(\zeta_2)\tilde{u}_3(\zeta_3)\tilde{g}(\zeta_4)\,d\tau_1\,d\tau_2\,d\tau_3\bigg] \bigg\|_{N^s_\mu}\\
&\quad \lesssim _{s} T^{0+} \| u_1\|_{Z^s_\mu}\| u_2\|_{Z^s_\mu}\| u_3\|_{Z^s_\mu}\| g\|_{G^s},
\end{aligned} 
\end{gather*}
where $g$ is the same as before.
Similarly to the case (II)-(ii) for the estimate of \eqref{est:ubarx} above, the left-hand side is bounded by
\[ T^{\frac12-}\| P_0(u_1\bar{u}_2)u_3J^{s+1}g\|_{L^2_{t,x}}\lesssim T^{\frac12-}\| u_1\|_{L^\infty_tL^2_x}\| u_2\|_{L^\infty_tL^2_x}\| u_3\|_{L^4_tL^\infty_x}\| J^{s+1}g\|_{L^4_tL^2_x}.\]
The estimate follows from $\| u_3\|_{L^4_tL^\infty_x}\lesssim \| u_3\|_{X^{\frac14+,\frac38}}$ by \eqref{est:L4} and $\| J^{s+1}g\|_{L^4_tL^2_x}\lesssim \| \langle n\rangle ^{s+1}\tilde{g}\|_{\ell^2_nL^{\frac43}_\tau}$ by the Hausdorff-Young inequality.

{\bf Estimate of \eqref{nonl:p-1}.}
It suffices to prove the following:
\begin{gather}
\begin{aligned}
&\bigg\| \mathcal{F}^{-1}\bigg[ |\hat{\psi}_T|*_\tau \int_{\tau=\tau_1-\tau_2+\dots +\tau_5}\sum _{n=n_1-n_2+\dots +n_5}\frac{|n_3-n_4|\tilde{u}_1(\zeta_1)\tilde{u}_2(\zeta_2)\cdots \tilde{u}_5(\zeta_5)}{\langle n_1-n_2+n_3-n_4\rangle}\,d\tau_1\cdots d\tau_4\bigg] \bigg\|_{N^s_\mu}\\
&\quad \lesssim _{s,\mu} T^{0+} \| u_1\|_{Z^s_\mu}\| u_2\|_{Z^s_\mu}\cdots \| u_5\|_{Z^s_\mu}.
\end{aligned} \label{est:p-1}
\end{gather}
We may also restrict our attention to the case $s\in (\frac14,\frac12)$.
We only consider the case $|n_1-n_2+n_3-n_4|\ll |n_3-n_4|\,(\sim |n_1-n_2|)$, since otherwise the estimate is reduced to \eqref{est:quintic} below.
We first use \eqref{est:Str0} to obtain 
\[ \text{LHS of \eqref{est:p-1}}\lesssim T^{0+}\| J^{-1}[J^{\frac12}(u_1\bar{u}_2)J^{\frac12}(u_3\bar{u}_4)]J^su_5\|_{L^{1+}_tL^2_x} \]
for $|n|\lesssim |n_5|$ and
\[ \text{LHS of \eqref{est:p-1}}\lesssim T^{0+}\| J^{-1+s}[J^{\frac12}(u_1\bar{u}_2)J^{\frac12}(u_3\bar{u}_4)]u_5\|_{L^{1+}_tL^2_x} \]
for $|n_5|\ll |n|$.
By the Sobolev embedding and the H\"older inequality (and the condition $s\in (\frac14,\frac12)$), both of them are bounded by
\[ T^{0+}\| J^{\frac12}(u_1\bar{u}_2)\|_{L^{2+}_{t,x}}\| J^{\frac12}(u_3\bar{u}_4)\|_{L^{2+}_{t,x}}\| J^su_5\|_{L^\infty_tL^2_x}.\]
The desired estimate \eqref{est:p-1} is derived from
\[ \| (J^{\frac12}u_1)\bar{u}_2\|_{L^{2+}_{t,x}}\lesssim \| J^{\frac12+}u_1\|_{L^{4+}_tL^2_x}\| J^{\frac14+}u_2\|_{L^4_{t,x}}\lesssim _\mu \| u_1\|_{Z^{\frac14+}_\mu}\| u_2\|_{X^{\frac14+,\frac38}},\]
where at the last inequality we have used \eqref{est:L4} and a slight modification of \eqref{est:L4L2}.

{\bf Estimate of \eqref{nonl:quintic}.}
The estimate is reduced to
\begin{gather}
\begin{aligned}
&\bigg\| \mathcal{F}^{-1}\bigg[ |\hat{\psi}_T|*_\tau \int_{\tau=\tau_1-\tau_2+\dots +\tau_5}\sum _{n=n_1-n_2+\dots +n_5}\tilde{u}_1(\zeta_1)\tilde{u}_2(\zeta_2)\cdots \tilde{u}_5(\zeta_5)\,d\tau_1\cdots d\tau_4\bigg] \bigg\|_{N^s_\mu}\\
&\quad \lesssim _{s} T^{0+} \| u_1\|_{Z^s_\mu}\| u_2\|_{Z^s_\mu}\cdots \| u_5\|_{Z^s_\mu}.
\end{aligned} \label{est:quintic}
\end{gather}
We may assume $|n_1|=\max _{1\leq j\leq 5}|n_j|$ without loss of generality.
Applying \eqref{est:Str1}, the left-hand side of the above is estimated by
\[ T^{\frac18-}\| (J^su_1)\bar{u}_2u_3\bar{u}_4u_5\|_{L^{\frac43}_{t,x}}\lesssim T^{\frac18-}\| J^su_1\| _{L^4_{t,x}}\| u_2\|_{L^8_{t,x}}\cdots \| u_5\|_{L^8_{t,x}}.\]
The estimate \eqref{est:quintic} follows from \eqref{est:L4} and 
\[ \| u\|_{L^8_{t,x}}\leq \| J^{\frac14}u\|_{L^4_{t,x}}^{\frac12}\| J^{-\frac14}u\|_{L^\infty_{t,x}}^{\frac12}\lesssim \| u\|_{X^{\frac14,\frac38}}^{\frac12}\| u\|_{Y^{\frac14+,0}}^{\frac12}.\]

\noindent
Thus, the proof of Lemma~\ref{lem:est-v} is complete.
\end{proof}

\subsection{A priori bounds on $u$}
\label{subsec:apriori-lwp}

\begin{prop}\label{prop:apriori}
Let $s>1/4$ and $0<T\leq 1$.
Assume that $u\in Z^s_1(T)$ is a solution to the Cauchy problem \eqref{kdnls-r} and \eqref{ic} on $[0,T]$ with $\phi \in B_{H^s}(R,r)$, where $R\geq r>0$.
Then, there exists $\delta =\delta (s)>0$ such that we have
\[ \| u\|_{Z^s_1(T)}\leq C_\phi +T^{\delta}C_u,\]
where the constants $C_\phi ,C_u$ depend on $s, \alpha ,\beta , R,r$ (but uniform in $T$), and $C_u$ additionally depends on $\| u\|_{Z^s_1(T)}$.
\end{prop}

\begin{proof}
We actually prove the claimed estimate in the $Z^s_\mu (T)$ norm with $\mu =|\beta| P_0(|\phi |^2)$.
However, the difference of the parameter $\mu$ is not important here, because $\phi \in B_{H^s}(R,r)$ and thus $\| u\|_{Z^s_\mu (T)}\sim _{\beta ,R,r}\| u\|_{Z^s_1(T)}$ by Remark~\ref{rem:normequiv}.

We first observe that
\[ \| u\|_{Z^s_\mu (T)}\leq \| P_0u\|_{Z^s_\mu (T)} +\| e^{-\rho ^+[u]}v_+\|_{Z^s_\mu (T)} +\| e^{-\rho ^-[u]}v_-\|_{Z^s_\mu (T)},\]
where $v_\pm$ are the gauge transforms of $u$ defined by \eqref{def:vpm}.
The first term on the right-hand side is bounded by $\| u\|_{Z^0_\mu(T)}$.
We apply the following estimate to treat this term:
by the integral equation associated to \eqref{eq:u'} and Lemma~\ref{lem:est-u},
\begin{gather}\label{est:u-s'}
\begin{aligned}
\| u\| _{Z^{s'}_\mu (T)}&\lesssim \| \phi\| _{H^s}+\| \mathcal{N}_0[u,\phi ;u]\|_{N^s_\mu(T)}+\| \mathcal{N}_u[u,u,u]\|_{N^{s'}_\mu(T)}\\
&\lesssim _{s,s',\mu}\| \phi\| _{H^s}+T^{0+}\| u\| _{Z^s_\mu (T)}^5+T^{0+}\| u\| _{Z^s_\mu (T)}^3\qquad (s'<s) .
\end{aligned}
\end{gather}
To estimate the terms $e^{-\rho ^\pm[u]}v_\pm$, we use Lemmas~\ref{lem:ubyv} and \ref{lem:gauge}:
\begin{align*}
&\| e^{-\rho ^\pm [u]}v_\pm \|_{Z^s_\mu (T)}\\
&\leq C\big( \| e^{-\rho ^\pm [u]}-e^{-\rho ^\pm [U_\mu (t)\phi ]}\| _{G^s(T)}+\| e^{-\rho ^\pm [U_\mu (t)\phi]}-1\|_{G^s(T)}+1\big) \| v_\pm \|_{Z^s_\mu(T)}\\
&\leq \big( C_u\| u-U_\mu (t)\phi\|_{Z^{s-}_\mu (T)}+C_\phi\big) \| v_\pm \| _{Z^s_\mu(T)},
\intertext{and then apply a variant of \eqref{est:u-s'} to obtain}
&\leq \big( T^{0+}C_u+C_\phi\big) \| v_\pm \| _{Z^s_\mu(T)}.
\end{align*}
Hence, the claimed estimate follows if we have the same estimate for $v_\pm$: 
\[ \| v_\pm \| _{Z^s_\mu(T)}\leq C_\phi +T^{0+}C_u.\]
Note that this is better than $\| v_\pm\|_{Z^s_\mu(T)}\leq C_u$ which follows from Corollary~\ref{cor:vbyu}.
To show this, we use the integral equation associated to \eqref{eq:vpm'} (with initial data $e^{\rho ^\pm [\phi]}P_\pm \phi$) and Lemmas~\ref{lem:est-u} (i), \ref{lem:est-v}, together with Corollary~\ref{cor:vbyu},
\begin{align*}
\| v_\pm\| _{Z^s_\mu(T)}&\lesssim \| e^{\rho ^\pm [\phi ]}P_\pm \phi \|_{H^s}+\|\mathcal{N}_0[u,\phi ;v_\pm ]\| _{N^s_\mu(T)}+\| \mathcal{N}_v^\pm [u]\| _{N^s_\mu(T)} \\
&\lesssim _{s,\mu} \big( \| e^{\rho ^\pm [\phi ]}-1\| _{H^{s\vee (\frac12+)}}+1\big) \| \phi\|_{H^s}+T^{0+}\| u\|_{Z^s_\mu(T)}^4\| v_\pm\|_{Z^s_\mu(T)}+T^{0+}C_u\\
&\lesssim C_\phi +T^{0+}C_u.
\end{align*}
Here, we have estimated the term $e^{\rho ^\pm [\phi ]}P_\pm \phi$ by the Sobolev inequalities:
\[ \| fg\|_{H^s}\lesssim \| f\|_{H^{s\vee (\frac12+)}}\| g\|_{H^s},\quad \| fg\| _{H^{[s\vee (\frac12+)]-1}}\lesssim \| f\|_{H^s}\| g\|_{H^s}\quad (s>0)\]
and 
\begin{align*}
\| e^{\rho ^\pm [\phi ]}-1\| _{H^{s\vee (\frac12+)}}&\leq \sum _{m=1}^\infty \frac{C^m}{m!}\big\| \partial _x^{-1}P_{\neq 0}(|\phi |^2)\big\| _{H^{s\vee (\frac12+)}}^m\leq e^{C\| \phi\|_{H^s}^2}.
\end{align*} 
This is the end of the proof.
\end{proof}

A slight modification of the argument used so far verifies the following \emph{a~priori} estimate for difference of two solutions to \eqref{kdnls-r}:
\begin{prop}\label{prop:apriori-diff}
Let $s>1/4$ and $0<T\leq 1$.
Assume that $u_1,u_2\in Z^s_1(T)$ are solutions of \eqref{kdnls-r} on $[0,T]$ with initial data $\phi _1,\phi _2\in B_{H^s}(R,r)$, respectively.
Then, there exists $\delta =\delta (s)>0$ such that we have
\[ \| u_1-u_2\| _{Z^s_1(T)}\leq C_{u_1,u_2}\big( \| \phi _1-\phi _2\|_{H^s}+T^{\delta}\| u_1-u_2\|_{Z^s_1(T)}\big) ,\]
where $C_{u_1,u_2}>0$ is a constant depending on $s,\alpha ,\beta ,R,r$ and $\| u_1\|_{Z^s_1(T)}$, $\| u_2\|_{Z^s_1(T)}$.
\end{prop}

\begin{proof}
This time we actually prove the claimed estimate in the $Z^s_\mu (T)$ norm with $\mu =|\beta| P_0(|\phi _1|^2)$.
We divide $u_1-u_2$ as in the proof of the preceding proposition and focus on the estimates of
\[ \| u_1-u_2\| _{Z^{s'}_\mu (T)}\quad (s'<s),\qquad \| e^{-\rho ^\pm [u_1]}v_{1,\pm}-e^{-\rho^\pm [u_2]}v_{2,\pm}\|_{Z^s_\mu (T)},\]
where $v_{1,\pm},v_{2,\pm}$ denote the gauge transforms of $u_1,u_2$.
For the first one, we use the integral equation associated to \eqref{eq:u-diff} and Lemma~\ref{lem:est-u}, \eqref{est:N0-0} to obtain
\begin{align*}
\| u_1-u_2\| _{Z^{s'}_{\mu}(T)}
&\lesssim \| \phi _1-\phi _2\| _{H^s}+\| \mathcal{N}_0[u_1,\phi _1;u_1-u_2]\| _{N^s_{\mu}(T)}\\
&\quad +\| \mathcal{N}_0[u_1,u_2;u_2]\| _{N^s_{\mu}(T)}+\| [P_0(|\phi_1|^2)-P_0(|\phi_2|^2)]\partial_xu_2\| _{N^s_{\mu}(T)}\\
&\quad +\| \mathcal{N}_u[u_1,u_1,u_1]-\mathcal{N}_u[u_2,u_2,u_2]\|_{N^{s'}_{\mu}(T)}\\
&\lesssim C_{u_1,u_2} \big( \| \phi _1-\phi _2\| _{H^s}+T^{0+}\| u_1-u_2\|_{Z^s_{\mu}(T)}\big) .
\end{align*}
For the second term, we use Lemmas~\ref{lem:ubyv}, \ref{lem:gauge} to have
\begin{align*}
&\| e^{-\rho ^\pm [u_1]}v_{1,\pm}-e^{-\rho ^\pm [u_2]}v_{2,\pm}\|_{Z^s_\mu (T)}\\
&\lesssim \| e^{-\rho ^\pm [u_1]}-e^{-\rho ^\pm [u_2]}\|_{G^s(T)}\| v_{1,\pm}\|_{Z^s_\mu (T)}+\big( \| e^{-\rho ^\pm [u_2]}-1\|_{G^s(T)}+1\big) \| v_{1,\pm}-v_{2,\pm}\|_{Z^s_{\mu}(T)} \\
&\lesssim C_{u_1,u_2}\big( \| u_1-u_2\|_{Z^{s-}_{\mu}(T)}\| v_{1,\pm}\|_{Z^s_\mu(T)}+\| v_{1,\pm}-v_{2,\pm}\|_{Z^s_{\mu}(T)}\big) .
\end{align*}
The first term in the last line is treated using the previous estimates for $u_1-u_2$ and the rough bound in Corollary~\ref{cor:vbyu} for $v^\pm_1$.
To estimate $v_{1,\pm}-v_{2,\pm}$, we use the integral equation associated to \eqref{eq:vpm-diff}, together with Lemmas~\ref{lem:est-u} (i), \ref{lem:est-v}, the estimate \eqref{est:N0-0} and Corollary~\ref{cor:vbyu}, to have
\begin{align*}
\| v_{1,\pm}-v_{2,\pm}\|_{Z^s_{\mu}(T)}
&\lesssim  \| e^{\rho ^\pm [\phi _1]}P_\pm \phi _1-e^{\rho ^\pm [\phi _2]}P_\pm \phi _2\| _{H^s}+\| \mathcal{N}_0[u_1,\phi _1;v_{1,\pm}-v_{2,\pm}]\| _{N^s_{\mu}(T)}\\
&\quad +\| \mathcal{N}_0[u_1,u_2;v_{2,\pm}]\| _{N^s_{\mu}(T)}+\| [P_0(|\phi_1|^2)-P_0(|\phi_2|^2)]\partial_xv_{2,\pm}\| _{N^s_{\mu}(T)}\\
&\quad +\| \mathcal{N}^\pm_v[u_1]-\mathcal{N}^\pm_v[u_2]\|_{N^s_{\mu}(T)}\\
&\lesssim C_{u_1,u_2}\big( \| \phi _1-\phi _2\| _{H^s}+T^{0+}\| u_1-u_2\|_{Z^s_{\mu}(T)}\big) .
\end{align*}
We therefore obtain the desired estimate.
\end{proof}

By further refining the argument, we also have the estimates in higher regularities as follows:
\begin{prop}\label{prop:apriori-pr}
Let $s>1/4$, $0<T\leq 1$ and $\theta >0$.
Assume that $u,u_1,u_2\in Z^{s+\theta}_1(T)$ are solutions of \eqref{kdnls-r} on $[0,T]$ with initial data $\phi ,\phi _1,\phi _2\in H^{s+\theta}\cap B_{H^s}(R,r)$, respectively.
Then, there exists $\delta =\delta (s)>0$ such that we have
\begin{align*}
\| u\|_{Z^{s+\theta}_1(T)}&\leq \tilde{C}_u\big( \| \phi \| _{H^{s+\theta}}+T^{\delta}\| u\|_{Z^{s+\theta}_1(T)}\big) ,\\
\| u_1-u_2\| _{Z^{s+\theta}_1(T)}&\leq \tilde{C}_{u_1,u_2}\big( \| \phi _1-\phi _2\|_{H^{s+\theta}}+T^\delta \| u_1-u_2\| _{Z^{s+\theta}_1(T)}\big) \\
&\quad +\tilde{C}_{u_1,u_2}\big( \| \phi _j\|_{H^{s+\theta}}+\| u_j\|_{Z^{s+\theta}_1(T)}\big) \big( \| \phi _1-\phi _2\|_{H^s}+\| u_1-u_2\|_{Z^s_1(T)}\big) ,
\end{align*}
where the constants $\tilde{C}_u,\tilde{C}_{u_1,u_2}$ are depending on the same quantities as for $C_u,C_{u_1,u_2}$ in Propositions~\ref{prop:apriori}, \ref{prop:apriori-diff} and in addition on $\theta$, but not on $\| \phi\| _{H^{s+\theta}}$, $\| \phi_j\| _{H^{s+\theta}}$, $\| u\| _{Z^{s+\theta}_1(T)}$ and $\| u_j\| _{Z^{s+\theta}_1(T)}$. 
\end{prop}

\begin{proof}
The claim is shown by generalizing the nonlinear estimates obtained in Subsections~\ref{subsec:gaugeest}--\ref{subsec:nonlest} to the corresponding estimates at the $H^{s+\theta}$ level and repeating the argument for Propositions~\ref{prop:apriori}, \ref{prop:apriori-diff}.

Let us see the first half.
The first estimate in Lemma~\ref{lem:ubyv} is immediately generalized as 
\[ \| gw\|_{Z^{s+\theta}_\mu}\lesssim \| g\|_{G^{s+\theta}}\| w\|_{Z^s_\mu}+\| g\|_{G^s}\| w\|_{Z^{s+\theta}_\mu}\qquad (s>\tfrac14,~\theta \geq 0,~\mu >0).\]
The second one in the restricted norms,
\[ \| gw\|_{Z^{s+\theta}_\mu(T)}\lesssim \| g\|_{G^{s+\theta}(T)}\| w\|_{Z^s_\mu(T)}+\| g\|_{G^s(T)}\| w\|_{Z^{s+\theta}_\mu(T)},\]
needs more care.%
\footnote{%
The problem is that we do not know if a suitable extension $u^\dagger$ of a function $u\in Z^{s+\theta}_\mu(T)$ is also suitable as an extension from $Z^s_\mu(T)$ to $Z^s_\mu$; i.e., if there is an extension $u^\dagger$ satisfying both $\| u^\dagger \|_{Z^s_\mu}\lesssim \| u\|_{Z^s_\mu(T)}$ and $\| u^\dagger \|_{Z^{s+\theta}_\mu}\lesssim \| u\|_{Z^{s+\theta}_\mu(T)}$.
If there is a canonical extension operator $\rho_T:Z^s_\mu(T)\to Z^s_\mu$ which is bounded and commutes with $J^\theta_x$, the restricted estimate follows immediately from the non-restricted one.}
We first decompose the bilinear operator $(g,w)\mapsto gw$ as
\[ gw=B_1(g,w)+B_2(g,w),\quad \mathcal{F}_xB_1(g,w)(n):=\sum _{|n_1|\geq |n-n_1|}\hat{g}(n_1)\hat{w}(n-n_1).\]
Then, from the proof for the non-restricted estimate, we have 
\[ \| B_1(g,w)\|_{Z^{s+\theta}_\mu(T)}\leq \| B_1(g^\dagger,w^\dagger)\|_{Z^{s+\theta}_\mu}\lesssim \| g^\dagger \| _{G^{s+\theta}}\| w^\dagger \|_{Z^s_\mu}\lesssim \| g\|_{G^{s+\theta}(T)}\| w\|_{Z^s_\mu(T)},\]
where $^\dagger$ means a suitable extension.
Similarly, we have 
\[ \| B_2(g,w)\|_{Z^{s+\theta}_\mu(T)}\lesssim \| g\|_{G^s(T)}\| w\|_{Z^{s+\theta}_\mu(T)},\]
and the desired estimate is shown.
Concerning Lemma~\ref{lem:gauge}, we need only to see that \eqref{est:G1}--\eqref{est:G2} are generalized as
\begin{gather*}
\| f_1f_2\|_{G^{s+\theta}(T)}\lesssim \| f_1\|_{G^{s+\theta}(T)}\| f_2\|_{G^s(T)}+\| f_1\|_{G^s(T)}\| f_2\|_{G^{s+\theta}(T)},\\
\| \partial_x^{-1}P^\pm_{\alpha,\beta}(u_1\bar{u}_2)\|_{G^{s+\theta}(T)}\lesssim \| u_1\|_{Z^{s+\theta-}_\mu(T)}\| u_2\|_{Z^{s-}_\mu(T)}+\| u_1\|_{Z^{s-}_\mu(T)}\| u_2\|_{Z^{s+\theta-}_\mu(T)}.
\end{gather*}
These can be shown by similar decomposition of bilinear operators and the argument for the non-restricted estimate with $\theta=0$.
Consequently, we obtain the estimates
\begin{gather*}
\| e^{c\partial_x^{-1}P^\pm_{\alpha,\beta}(|u|^2)}-1\|_{G^{s+\theta}(T)}\lesssim e^{C\| u\|_{Z^{s-}_\mu(T)}^2}\| u\|_{Z^{s-}_\mu(T)}\| u\|_{Z^{s+\theta-}_\mu(T)},\\
\begin{aligned}
&\| e^{c\partial_x^{-1}P^\pm_{\alpha,\beta}(|u_1|^2)}-e^{c\partial_x^{-1}P^\pm_{\alpha,\beta}(|u_2|^2)}\|_{G^{s+\theta}(T)}\\
&\quad \lesssim e^{C\| u_j\|_{Z^{s-}_\mu(T)}^2}\big( \| u_j\|_{Z^{s-}_\mu(T)}\| u_1-u_2\|_{Z^{s+\theta-}_\mu(T)}+\| u_j\|_{Z^{s+\theta-}_\mu(T)}\| u_1-u_2\|_{Z^{s-}_\mu(T)}\big) 
\end{aligned}
\end{gather*}
as a generalization of Lemma~\ref{lem:gauge}.
Using these estimates, Corollary~\ref{cor:vbyu} is generalized as
\begin{align*}
\| v_\pm\|_{Z^{s+\theta}_\mu(T)}&\lesssim \big( 1+e^{C\| u\|_{Z^s_\mu(T)}^2}\| u\|_{Z^s_\mu(T)}^2\big) \| u\|_{Z^{s+\theta}_\mu(T)},\\
\| v_{1,\pm}-v_{2,\pm}\|_{Z^{s+\theta}_\mu(T)}
&\lesssim \big( 1+e^{C\| u_j\|_{Z^s_\mu(T)}^2}\| u_j\|_{Z^s_\mu(T)}^2\big) \| u_1-u_2\|_{Z^{s+\theta}_\mu(T)}\\
&\quad +e^{C\| u_j\|_{Z^s_\mu(T)}^2}\| u_j\|_{Z^s_\mu(T)}\| u_j\|_{Z^{s+\theta}_\mu(T)}\| u_1-u_2\|_{Z^s_\mu(T)}.
\end{align*}
By a similar argument, for Lemmas~\ref{lem:est-u} (ii) and \ref{lem:est-v} we obtain 
\begin{gather*}
\| \mathcal{N}_u[u,u,u]\|_{N^{s+\theta-}_\mu(T)}\lesssim T^\delta \| u\|_{Z^s_\mu(T)}^2\| u\|_{Z^{s+\theta}_\mu(T)},\\
\begin{aligned}
&\| \mathcal{N}_u[u_1,u_1,u_1]-\mathcal{N}_u[u_2,u_2,u_2]\|_{N^{s+\theta-}_\mu(T)}\\
&\quad \lesssim T^\delta \big( \| u_j\|_{Z^s_\mu(T)}^2\| u_1-u_2\|_{Z^{s+\theta}_\mu(T)}+\| u_j\|_{Z^s_\mu(T)}\| u_j\|_{Z^{s+\theta}_\mu(T)}\| u_1-u_2\|_{Z^s_\mu(T)}\big) ,
\end{aligned}
\end{gather*}
and
\begin{gather*}
\begin{aligned}
\| \mathcal{N}_v^\pm [u]\| _{N^{s+\theta}_\mu (T)}&\lesssim T^{\delta}\big( 1+e^{C\| u\|_{Z^s_\mu(T)}^2}\| u\|_{Z^s_\mu(T)}^2\big) \big( \| u\|_{Z^s_\mu(T)}^2+\| u\|_{Z^s_\mu(T)}^4\big) \| u\|_{Z^{s+\theta}_\mu(T)},
\end{aligned}\\
\begin{aligned}
\| \mathcal{N}_v^\pm [u_1]-\mathcal{N}_v^\pm [u_2]\| _{N^{s+\theta}_\mu (T)}
&\lesssim T^{\delta}\big( 1+e^{C\| u_j\|_{Z^s_\mu(T)}^2}\| u_j\|_{Z^s_\mu(T)}^2\big) \big( \| u_j\|_{Z^s_\mu(T)}+\| u_j\|_{Z^s_\mu(T)}^3\big) \\
\times \big( &\| u_j\|_{Z^s_\mu(T)}\| u_1-u_2\|_{Z^{s+\theta}_\mu (T)}+\| u_j\|_{Z^{s+\theta}_\mu(T)}\| u_1-u_2\|_{Z^s_\mu (T)}\big) .
\end{aligned}
\end{gather*}
Finally, we note that Lemma~\ref{lem:est-u} (i) and the estimate \eqref{est:N0-0} are ready to apply at the $H^{s+\theta}$ level, just by setting $\sigma =s+\theta$.

Now, we prove the claimed \emph{a priori} estimates.
By the same decomposition of $u$ and $u_1-u_2$ as before, it suffices to estimate $\| u\|_{Z^{s+\theta-}_\mu(T)}$, $\| e^{-\rho ^\pm[u]}v_\pm\|_{Z^{s+\theta}_\mu(T)}$ and the corresponding difference terms.
The treatment of $\| u\|_{Z^{s+\theta-}_\mu(T)}$ is similar to the argument for the $Z^{s-}_\mu(T)$ estimate, and we use the $H^{s+\theta}$-level nonlinear estimates shown above to obtain
\begin{align*}
\| u\|_{Z^{s+\theta-}_\mu(T)}&\lesssim \| \phi\|_{H^{s+\theta}}+T^{0+}\tilde{C}_u\| u\|_{Z^{s+\theta}_\mu(T)},\\
\| u_1-u_2\|_{Z^{s+\theta-}_\mu(T)}&\lesssim \| \phi_1-\phi_2\|_{H^{s+\theta}}+T^{0+}\tilde{C}_{u_1,u_2}\Big( \| u_1-u_2\|_{Z^{s+\theta}_\mu(T)}\\
&\qquad +\| u_j\|_{Z^{s+\theta}_\mu(T)}\big( \| \phi_1-\phi_2\|_{L^2}+\| u_1-u_2\|_{Z^s_\mu(T)}\big) \Big) .
\end{align*}
To estimate $e^{-\rho ^\pm[u]}v_\pm$ and the difference, we make a decomposition as follows:
\begin{gather*}
\begin{aligned}
&\| e^{-\rho ^\pm[u]}v_\pm\|_{Z^{s+\theta}_\mu(T)}\\
&\quad \lesssim \| e^{-\rho ^\pm[u]}-1\|_{G^{s+\theta}(T)}\| v_\pm\|_{Z^s_\mu(T)}+\big( \| e^{-\rho ^\pm[u]}-1\|_{G^s(T)}+1\big) \| v_\pm \|_{Z^{s+\theta}_\mu(T)},
\end{aligned}\\
\begin{aligned}
&\| e^{-\rho ^\pm[u_1]}v_{1,\pm}-e^{-\rho ^\pm[u_2]}v_{2,\pm}\|_{Z^{s+\theta}_\mu(T)}\\
&\lesssim \| e^{-\rho ^\pm[u_1]}-e^{-\rho ^\pm[u_2]}\|_{G^{s+\theta}(T)}\| v_{1,\pm}\|_{Z^s_\mu(T)}+\| e^{-\rho ^\pm[u_1]}-e^{-\rho ^\pm[u_2]}\|_{G^{s}(T)}\| v_{1,\pm}\|_{Z^{s+\theta}_\mu(T)}\\
&\quad +\| e^{-\rho ^\pm[u_2]}-1\|_{G^{s+\theta}(T)}\| v_{1,\pm}-v_{2,\pm}\|_{Z^s_\mu(T)}+\big( \| e^{-\rho ^\pm[u_2]}-1\|_{G^{s}(T)}+1\big) \| v_{1,\pm}-v_{2,\pm}\|_{Z^{s+\theta}_\mu(T)}.
\end{aligned}
\end{gather*}
For each case, the first term on the right-hand side is estimated in terms of $\| u\|_{Z^{s+\theta-}_\mu(T)}$ or $\| u_1-u_2\|_{Z^{s+\theta-}_\mu(T)}$, for which we use the above estimates to obtain the claimed bounds.
To handle the last term in each case, we estimate $\| v_\pm \|_{Z^{s+\theta}_\mu(T)}$ and $\| v_{1,\pm}-v_{2,\pm}\|_{Z^{s+\theta}_\mu(T)}$ by using the integral equations associated to \eqref{eq:vpm'} and \eqref{eq:vpm-diff} similarly to the previous argument.
Here, we only mention how the linear parts are estimated.
Using the Sobolev inequalities
\begin{align*}
\| fg\|_{H^{s+\theta}}&\lesssim \| f\|_{H^{s+\theta+\frac12}}\| g\|_{H^{0+}}+\| f\|_{H^{\frac12+}}\| g\|_{H^{s+\theta}},\\
\| fg\|_{H^{s+\theta-\frac12}}&\lesssim \| f\|_{H^{s+\theta}}\| g\|_{H^{0+}}+\| f\|_{H^{0+}}\| g\|_{H^{s+\theta}},\\
\| fg\|_{H^{-\frac12+}}&\lesssim \| f\|_{H^{0+}}\| g\|_{H^{0+}},
\end{align*}
we obtain 
\begin{gather*}
\| e^{\rho^\pm [\phi]}P_\pm \phi\|_{H^{s+\theta}}\lesssim e^{C\| \phi \|_{H^s}^2}\| \phi\|_{H^{s+\theta}},\\
\begin{aligned}
&\| e^{\rho^\pm [\phi_1]}P_\pm \phi_1 -e^{\rho^\pm [\phi_2]}P_\pm \phi_2\|_{H^{s+\theta}}\\
&\quad \lesssim e^{C\| \phi _j\|_{H^s}^2}\big( \| \phi_1-\phi_2 \|_{H^{s+\theta}}+\| \phi_j\|_{H^s}\| \phi_j\|_{H^{s+\theta}}\| \phi_1-\phi_2\|_{H^s}\big) .
\end{aligned}
\end{gather*}
This completes the proof.
\end{proof}

\subsection{Proof of local well-posedness}
\label{subsec:proof-lwp}

\begin{proof}[Proof of Theorem~\ref{thm:lwp2}]
Let $s>\frac14$, and $R\geq r>0$ be arbitrary.
For given $\phi\in B_{H^s}(R,r)$, we take a sequence of approximating initial data $\{ \phi _j\} _{j\in \mathbf{N}}\subset H^\sigma \cap B_{H^s}(2R,\frac{r}{2})$ (with $\sigma$ sufficiently large) satisfying $\phi _j\to \phi$ in $H^s$.
The local well-posedness of the Cauchy problem \eqref{kdnls-r} and \eqref{ic} in $H^s(\mathbf{T})$ is shown for $s>\frac32$ by the classical energy method (see Theorem~\ref{thm:lwp-e} and the following remark in Appendix~\ref{sec:appendix}).
Let $u_j\in C([0,T_j];H^\sigma )$ be the unique solution of \eqref{kdnls-r} with initial condition $u_j(0)=\phi _j$, and we note that $u_j\in C^1([0,T_j];H^{\frac72+})\subset Z^{\frac32+}_1(T_j)$ if $\sigma$ is large enough (see Lemma~\ref{lem:Zmu}).

\medskip
\textbf{Step 1}: Uniform existence time for $\{ u_j\}_j$.
From the local well-posedness in $H^{\frac32+}$, it suffices to show \emph{a priori} bound 
\begin{equation}\label{apriori:H32}
\| u_j(t)\| _{H^{\frac32+}}\leq C\| \phi_j\|_{H^{\frac32+}}, \qquad t\in [0,T]
\end{equation}
for some $C,T>0$ depending only on $s,R,r$.
The \emph{a~priori} bound given in Proposition~\ref{prop:apriori} and a continuity argument (with the aid of Lemma~\ref{lem:Zmu}) show that
\begin{equation}\label{apriori:Hs}
\| u_j\| _{Z^s_1(T)}\leq 2C_\phi 
\end{equation}
if $T$ is chosen as $T^\delta \cdot C_u\big| _{\| u\|_{Z^s_1(T)}=2C_\phi}\leq \frac12 C_\phi$.
Substituting the bound \eqref{apriori:Hs} into the first estimate in Proposition~\ref{prop:apriori-pr}, we have
\[ \| u_j\| _{Z^{\frac32+}_1(T)}\leq \tilde{C}_\phi \big( \| \phi _j\| _{H^{\frac32+}}+T^{\delta}\| u_j\| _{Z^{\frac32+}_1(T)}\big) ,\]
where the constant $\tilde{C}_\phi:=\tilde{C}_u\big| _{\| u\|_{Z^s_1(T)}=2C_\phi}$ still depends only on $s,R,r$.
Replacing $T$ so that $T^{\delta}\tilde{C}_\phi \leq \frac12$, we obtain
\[ \| u_j\| _{Z^{\frac32+}_1(T)}\leq 2\tilde{C}_\phi \| \phi _j\|_{H^{\frac32+}},\]
which implies \eqref{apriori:H32}.

\medskip
\textbf{Step 2}: Existence of a solution in $Z^s_1(T)$ and Lipschitz continuous dependence on initial data. 
From the \emph{a~priori} bound \eqref{apriori:Hs} and Proposition~\ref{prop:apriori-diff}, we have
\[ \| u_j-u_{j'}\| _{Z^s_1(T)}\leq C\big( \| \phi _j-\phi _{j'}\| _{H^s}+T^{\delta}\| u_j-u_{j'}\| _{Z^s_1(T)}\big) ,\]
where the constant $C$ depends only on $s,R,r$.
If $T$ is chosen as $T^{\delta}C\leq \frac12$, this bound implies that $\{ u_j\}_j$ is a Cauchy sequence in $Z^s_1(T)$, and the limit is a solution satisfying the same bound as \eqref{apriori:Hs}.
To ensure the $L^2$ lower bound, we replace $T$ again, noticing the fact that
\[ \sup_{t\in [0,T]}\big| \| u(t)\| _{L^2}^2-\| \phi\| _{L^2}^2\big| \lesssim T^{\delta}\| u\|_{Z^s_1(T)}^4,\]
which we have shown in the proof of Lemma~\ref{lem:est-u} (i).
The Lipschitz dependence follows again from Proposition~\ref{prop:apriori-diff}, since we know that the solution constructed above obeys the same bound as \eqref{apriori:Hs}.%
\footnote{Note that we did not prove the same \emph{a priori} bound as \eqref{apriori:Hs} for general solutions in $Z^s_1(T)$ with $u(0)\in B_{H^s}(R,r)$.
This is because we did not prove the continuity of the map $T'\mapsto \| u\|_{Z^s_1(T')}$ without assuming additional regularity $C^1([0,T];H^{s+2})$, and thus the continuity argument in Step~1 does not work for general solutions in $Z^s_1(T)$.
Consequently, this Lipschitz bound is not sufficient to prove uniqueness.}

At this point, we can fix the local existence time $T$ depending on $s,R,r$.

\medskip
\textbf{Step 3}: Uniqueness in $Z^s_1(T)$.
We note that Proposition~\ref{prop:apriori-diff} is valid for any solutions of \eqref{kdnls-r} belonging to $Z^s_1(T)$ that are not necessarily approximated by smooth solutions.%
\footnote{To be more precise, we need to confirm that for any solution $u\in Z^s_1(T)$ the gauge transforms $v_\pm$ satisfy the transformed equations (in a suitable sense).
This may be verified by considering the equation for $v^N_\pm :=e^{\rho^\pm[P_{\leq N}u]}P_\pm P_{\leq N}u$ and taking the limit $N\to \infty$.}
Let $u_1,u_2\in Z^s_1(T)$ be two solutions with $u_1(0)=u_2(0)=\phi \in B_{H^s}(R,r)$, and suppose for contradiction that $t_0:=\sup \{ T'\geq 0:u_1(t)=u_2(t)~\text{on}~[0,T']\} <T$.
Let $\phi':=u_1(t_0)=u_2(t_0)$. 
The possibility that $\phi'=0$ is excluded, since it would imply $u_1(t)=u_2(t)=0$ on $[t_0,T]$ by the monotonicity of the $L^2$ norm \eqref{L2eq}.
Then, $u'_j(t,x):=u_j(t+t_0,x-\nu t)$ with $\nu :=2\alpha \big( P_0(|\phi'|^2)-P_0(|\phi |^2)\big)$, $j=1,2$, are two solutions of \eqref{kdnls-r} on $[0,T-t_0]$ with the same non-zero initial data $\phi'$.
Applying Proposition~\ref{prop:apriori-diff}, we obtain
\[ \| u'_1-u'_2\| _{Z^s_1(T')}\leq C(T')^{\delta}\| u'_1-u'_2\| _{Z^s_1(T')}\]
for $T'\in (0,T-t_0]$, with the constant $C>0$ depending on $s,\| \phi'\|_{H^s}$, $\| \phi'\|_{L^2}>0$, and%
\footnote{We recall that the transformation $u(t,x)\mapsto u(t,x-\nu t)$ is a bi-Lipschitz bijection on $Z^s_1(T)$ for any $\nu \in \mathbf{R}$, $s\in \mathbf{R}$ and $T>0$.
In fact, this can be seen from $\mathcal{F}[u(t,x-\nu t)](\tau,n)=\tilde{u}(\tau -\nu n,n)$ and $\langle i(\tau +\nu n+n^2)+|n|\rangle \sim _\nu \langle i(\tau +n^2)+|n|\rangle$.}
\[ \| u'_j\|_{Z^s_1(T')}\lesssim _\nu \| u_j(\cdot +t_0)\| _{Z^s_1(T')}\leq \| u_j\|_{Z^s_1(T)}\qquad (j=1,2).\]
From this, we deduce that $u_1(t)=u_2(t)$ on $[t_0,t_0+T']$ for some $T'>0$, which contradicts the definition of $t_0$.
Hence, we see $t_0=T$, and uniqueness holds.

\medskip
\textbf{Step 4}: Persistence of regularity.
Recall that the solution $u$ on $[0,T]$ constructed above stays for all time in the set $B_{H^s}(R',r')$ for some $R'>0$ depending on $s,R,r$ and $r':=\frac{r}{2}$.
It suffices to prove the following: there exists $T'=T'(s,\theta,R',r')>0$ such that if $u(t_0)\in H^{s+\theta}$ for some $t_0\in [0,T)$, then $u\in C(I;H^{s+\theta})\cap L^2(I;H^{s+\theta+\frac12})$, $I:=[t_0,t_0+T']\cap [0,T]$.
Again considering $u(t+t_0,x-2\alpha (P_0(|u(t_0)|^2)-P_0(|\phi |^2))t)\in Z^s_1(T-t_0)$ instead of $u(t,x)$, we may set $t_0=0$.
Take an approximating sequence $\{ \phi_j\}_j\subset H^{\sigma}\cap B_{H^s}(2R',\frac{r'}{2})$ of $\phi$ in $H^{s+\theta}$ (with $\sigma$ sufficiently large).
Following the argument in Steps~1--2, we see that the corresponding classical solution $u_j$ of \eqref{kdnls-r} exists on $[0,T']$ with $T'=T'(s,\theta ,R',r')>0$ and $u_j\in C^1([0,T'];H^{s+\theta+2})\subset Z^{s+\theta}_1(T')$, satisfying a bound $\| u_j\|_{Z^{s+\theta}_1(T')}\leq C'\| \phi _j\|_{H^{s+\theta}}$ as well as a Lipschitz bound $\| u_j-u_{j'}\|_{Z^s_1(T')}\leq C'\| \phi _j-\phi _{j'}\|_{H^s}$ for some $C'>0$ depending on $s,\theta,R',r'$.
Substituting these bounds into the second estimate in Proposition~\ref{prop:apriori-pr}, we see that $\{ u_j\}_j$ is a Cauchy sequence in $Z^{s+\theta}_1(T')$:
\[ \| u_j-u_{j'}\|_{Z^{s+\theta}_1(T')}\leq C'\Big\{ \| \phi _j-\phi_{j'}\|_{H^{s+\theta}}+\Big( \sup _j\| \phi_j\|_{H^{s+\theta}}\Big) \| \phi _j-\phi _{j'}\|_{H^s}\Big\} ,\]
where $C',T'$ may be replaced but still depending only on $s,\theta,R',r'$.
The limit $u'\in Z^{s+\theta}_1(T')$ is a solution on $[0,T']$ with initial data $\phi$, and by uniqueness in $Z^s_1(T')$ we have $u'(t)=u(t)$ for $t\in [0,T']$, showing that $u\in C([0,T'];H^{s+\theta})\cap L^2([0,T'];H^{s+\theta+\frac12})$.

\medskip
\textbf{Step 5}: Smoothing property.
Smoothness in $x$ follows from uniqueness and persistence of regularity, together with an observation that if $u(t_0)\in H^{s'}$ for some $t_0\in [0,T)$ and $s'\geq s$, then $u(t)$ has higher regularity $H^{s'+\frac12}$ for almost every $t\in (t_0,T)$.
Then, smoothness in $t$ is deduced from that in $x$ and the equation \eqref{kdnls-r}.
\end{proof}

\begin{proof}[Proof of Corollary~\ref{cor:lwp2}]
All the statements in Theorem~\ref{thm:lwp2} but the Lipschitz dependence can be easily verified for the original Cauchy problem \eqref{kdnls}--\eqref{ic}.
On the other hand, it is easy to see that 
\[ \| u(t,x)-u(t,x-\nu t)\| _{Z^s_1(T)}\to 0\qquad (\nu \to 0)\]
for $u\in Z^s_1(T)$.
Then, for two solutions $u_1,u_2\in Z^s_1(T)$ of \eqref{kdnls} with $\phi_1,\phi_2\in B_{H^s}(R,r)$, we have
\begin{align*}
\| u_1-u_2\|_{Z^s_1(T)}&\leq \| u_1(t,x)-u_1(t,x-2\alpha (P_0(|\phi_1|^2)-P_0(|\phi_2|^2))t)\|_{Z^s_1(T)}\\
&\quad +\| u_1(t,x-2\alpha (P_0(|\phi_1|^2)-P_0(|\phi_2|^2))t)-u_2(t,x)\|_{Z^s_1(T)}\\
&\leq o(1) +C\| u_1(t,x-2\alpha P_0(|\phi_1|^2)t)-u_2(t,x-2\alpha P_0(|\phi_2|^2)t)\|_{Z^s_1(T)}\\
&\leq o(1) +C\| \phi_1-\phi_2\|_{H^s}\\
&=o(1)\qquad (\| \phi_1-\phi_2\|_{H^s}\to 0),
\end{align*}
which shows continuity of the solution map.
\end{proof}


\section{Global solvability}\label{sec:global}

In \cite{KTmatrix}, we established the following a priori bounds for \emph{small} solutions to \eqref{kdnls}:

\begin{prop}[\cite{KTmatrix}, Theorem~1]\label{prop:small}
Let $\alpha ,\beta \in \mathbf{R}$.
Let $u$ be a smooth solution to \eqref{kdnls}--\eqref{ic} on $[0,T]\times Z$, where $Z$ is either $\mathbf{R}$ or $\mathbf{T}$.
Then, it holds that
\begin{equation}\label{L2equality}
\| u(t)\| _{L^2}^2=\| \phi \| _{L^2}^2+\beta \int _0^t\| D_x^{\frac12}(|u(t')|^2)\| _{L^2}^2\,dt',\qquad t\in [0,T].
\end{equation}
Moreover, if $\beta <0$, there exist $C_*,C>0$ depending only on $\alpha ,\beta$ (and bounded when $\beta \to 0$) such that if $\| \phi \| _{L^2} \leq C_*^{-1}$, then
\begin{gather*}
\| u(t)\| _{H^1}^2+\frac{|\beta |}{4}\int _0^t \| D_x^{\frac12}\partial _x(|u(t')|^2)\|_{L^2}^2\,dt' \leq 4\| \phi \| _{H^1}^2e^{C\| \phi \| _{L^2}^2},\\
\| u(t)\|_{L^2}^2\geq \| \phi \| _{L^2}^2\exp \Big[ -C\| \phi \|_{H^1}e^{C\| \phi \| _{L^2}^2}|\beta |^{1/2}t^{1/2}\Big]
\end{gather*}
for $t\in [0,T]$.
\end{prop}
The above $H^1$ bound was shown by analyzing the differential equalities for the mass $\| u(t)\| _{L^2}^2$ and the energy functional
\[ \int \Big\{ |u_x|^2-\frac{3}{2}\Big( \alpha |u|^2+\beta \mathcal{H}{|u|^2}\Big) \Im (\bar{u}u_x) +\frac{1}{2}\alpha ^2|u|^6\Big\} \,dx,\]
which is a natural extension of the conserved energy (the case $\beta=0$) for the derivative nonlinear Schr\"odinger equation.
In this section, we aim to show the $H^1$ a priori bound for solutions with \emph{arbitrarily large} $L^2$ norm by modifying the previous argument.
Indeed, we have the following proposition for a large solution, the proof of which is described after the proof of Theorem~\ref{cor:gwp}.

\begin{prop}\label{prop:aprioriH1large}
Let $\alpha \in \mathbf{R}$ and $\beta <0$.
Let $u$ be a smooth solution to \eqref{kdnls}--\eqref{ic} on $[0,T]\times Z$, where $Z$ is either $\mathbf{R}$ or $\mathbf{T}$.
Then, there exists $C>0$ depending only on $\alpha ,\beta$ (and unbounded as $\beta \to 0$) such that we have
\begin{align*}
\int _0^t \| D_x^{\frac12}\partial _x(|u(t')|^2)\|_{L^2}^2\,dt' \leq C\big( 1+\| \phi\| _{L^2}^4\big) \| \phi _x\|_{L^2}^2 e^{C(1+\| \phi \| _{L^2}^6)\| \phi \|_{L^2}^2},\quad &\\
\| u_x(t)\|_{L^2}^2 \leq C\big( 1+\| \phi\| _{L^2}^8\big) \| \phi _x\|_{L^2}^2 e^{C(1+\| \phi \| _{L^2}^6)\| \phi \|_{L^2}^2},\quad& \\[5pt]
\| u(t)\|_{L^2}^2\geq \|\phi \|_{L^2}^2\exp \Big[ -Ct^{1/2}\big( 1+\| \phi\| _{L^2}^2\big) \| \phi _x\|_{L^2} e^{C(1+\| \phi \| _{L^2}^6)\| \phi \|_{L^2}^2}&\Big] ,\qquad t\in [0,T]
\end{align*}
for the case $Z=\mathbf{R}$, and also have the same estimates but with $\| \phi_x\|_{L^2}$ replaced by $\| \phi\|_{H^1}$ for the case $Z=\mathbf{T}$.
\end{prop}

From these bounds and the local well-posedness results obtained so far, we can show global well-posedness:
\begin{proof}[Proof of Theorem~\ref{cor:gwp}]
We may assume $1/4<s<1$.
Theorem~\ref{thm:lwp} shows that for any $\phi\in H^s\setminus \{ 0\}$, the Cauchy problem has a solution $u$ on some time interval $[0,T]$ and $u(T)\in H^\infty$, $\| u(T)\|_{L^2}\geq \frac12 \| \phi \|_{L^2}>0$.
By virtue of \emph{a priori} $H^1$ and $L^2$-lower bounds for smooth solutions given in Proposition~\ref{prop:aprioriH1large}, we can iterate the construction of local $H^1$ solution (which is smooth by persistence of regularity) to obtain a global solution.
\end{proof}

In the rest of this section, we shall prove Proposition~\ref{prop:aprioriH1large}.
\begin{proof}[Proof of Proposition~\ref{prop:aprioriH1large}]
We will use a different energy functional
\begin{equation}\label{defn:energy}
\begin{aligned}
E[u]&:= \int \Big| u_x-i\frac{3}{4}\Big( \alpha |u|^2+\beta \mathcal{H}(|u|^2)\Big) u\Big| ^2\,dx \\
&\;= \int \Big\{ |u_x|^2-\frac{3}{2}\Big( \alpha |u|^2+\beta \mathcal{H}(|u|^2)\Big) \Im (\bar{u}u_x) \\
&\qquad +\frac{9}{16}\Big( \alpha ^2|u|^6 +2\alpha \beta |u|^4\mathcal{H}(|u|^2)+\beta ^2 |u|^2\big[ \mathcal{H}(|u|^2)\big] ^2\Big) \Big\} \,dx.
\end{aligned}
\end{equation}
By definition, $E[u]$ is non-negative.
From now on, we write 
\[ \mathfrak{m}:=|u|^2,\qquad \mathfrak{p}:=\Im (\bar{u}u_x)\]
for simplicity.
Then, $E[u]$ is rewritten as
\[ E[u]=\int \Big\{ |u_x|^2-\frac{3}{2}\big( \alpha \mathfrak{m} +\beta \mathcal{H}\mathfrak{m}\big) \mathfrak{p} +\frac{9}{16}\big( \alpha ^2\mathfrak{m}^3+2\alpha \beta \mathfrak{m}^2\mathcal{H}\mathfrak{m} +\beta ^2\mathfrak{m}(\mathcal{H}\mathfrak{m})^2\big) \Big\} \,dx.\]
Recall \cite[Lemma~2]{KTmatrix}: let $u$ be a smooth solution to \eqref{kdnls}, then
\begin{align*}
\partial_t \int |u_x|^2\,dx
&=-3\int \big( \alpha \mathfrak{m}+\beta \mathcal{H}\mathfrak{m}\big) \partial _x(|u_x|^2) \,dx +\beta \big\| D^{1/2}\partial _x\mathfrak{m}\big\| _{L^2}^2,\\
\partial_t \int \mathfrak{m}\mathfrak{p}\,dx
&=-2\int \mathfrak{m}\partial _x (|u_x|^2)\,dx +\int \Big\{ 2\alpha \partial_x (\mathfrak{m}^2) +4\beta\mathfrak{m}D\mathfrak{m} \Big\} \mathfrak{p}\,dx,\\
\partial_t \int (\mathcal{H}\mathfrak{m})\mathfrak{p}\,dx
&=-2\int (\mathcal{H}\mathfrak{m})\partial _x(|u_x|^2)\,dx-2\big\| D^{1/2}\mathfrak{p}\big\| _{L^2}^2+\frac{1}{2}\big\| D^{1/2}\partial_x \mathfrak{m}\big\| _{L^2}^2 \\
&\qquad + \alpha \int \Big\{  \frac{3}{2}D(\mathfrak{m}^2)- \mathfrak{m}D\mathfrak{m}+2(\mathcal{H}\mathfrak{m})\partial _x\mathfrak{m} \Big\} \mathfrak{p}\,dx \\
&\qquad + \beta \int \Big\{ D[\mathfrak{m}\mathcal{H}\mathfrak{m}] +\mathcal{H}[\mathfrak{m}D\mathfrak{m}] +(\mathcal{H}\mathfrak{m})D\mathfrak{m} \Big\} \mathfrak{p}\,dx,\\
\partial_t \int \mathfrak{m}^3\,dx 
&=6\int \partial _x(\mathfrak{m}^2)\mathfrak{p} \,dx+2\beta \int \mathfrak{m}^3D\mathfrak{m}\,dx.
\end{align*}
Furthermore, we calculate 
\begin{align*}
\partial_t \int \mathfrak{m}^2\mathcal{H}\mathfrak{m}\,dx 
&=\int \Big\{ 4\partial_x [\mathfrak{m}\mathcal{H}\mathfrak{m}] -2D(\mathfrak{m}^2)\Big\} \mathfrak{p}\,dx \\
&\qquad +\alpha \int \mathfrak{m}^2 \Big\{ \frac{3}{2}D(\mathfrak{m}^2)-2\mathfrak{m}D\mathfrak{m}\Big\} \,dx \\
&\qquad +\beta \int \mathfrak{m}^2\Big\{ D[\mathfrak{m}\mathcal{H}\mathfrak{m}] +\mathcal{H}[ \mathfrak{m}D\mathfrak{m}]+2(\mathcal{H}\mathfrak{m})D\mathfrak{m} \Big\} \,dx,\\
\partial_t \int \mathfrak{m}(\mathcal{H}\mathfrak{m})^2
&=\int \Big\{ 4(\mathcal{H}\mathfrak{m})D\mathfrak{m}-4D[\mathfrak{m}\mathcal{H}\mathfrak{m}]\Big\} \mathfrak{p}\,dx \\
&\qquad +\alpha \int \mathfrak{m}(\mathcal{H}\mathfrak{m})\Big\{ 3D(\mathfrak{m}^2) -3\mathfrak{m}D\mathfrak{m} \Big\} \,dx \\
&\qquad +\beta \int \mathfrak{m}(\mathcal{H}\mathfrak{m})\Big\{ 2D[\mathfrak{m}\mathcal{H}\mathfrak{m}] +2\mathcal{H}[\mathfrak{m}D\mathfrak{m}] -(\mathcal{H}\mathfrak{m})D\mathfrak{m} \Big\} \,dx.
\end{align*}
The point is that $\partial_tE[u]$ can be written as
\begin{equation}\label{eq:E}
\partial _tE[u] = \frac{\beta}{4}\big\| D^{1/2}\partial _x\mathfrak{m}\big\| _{L^2}^2 +3\beta \big\| D^{1/2}\mathfrak{p}\big\| _{L^2}^2 + F(\mathfrak{m},\mathfrak{p}) +G(\mathfrak{m}),
\end{equation}
where $F(\mathfrak{m},\mathfrak{p})$ is a sum of terms of the form
\begin{equation}\label{defn:F}
c_{\alpha ,\beta} \int\limits _{\xi _1+\xi _2+\xi _3=0}a(\xi _1,\xi _2,\xi _3)\hat{\mathfrak{m}}(\xi _1)\hat{\mathfrak{m}}(\xi _2)\hat{\mathfrak{p}}(\xi _3),\qquad |a(\xi _1,\xi _2,\xi _3)|\lesssim |\xi _1|+|\xi _2|,
\end{equation}
and $G(\mathfrak{m})$ is a sum of terms of the form
\begin{equation}\label{defn:G}
c_{\alpha ,\beta} \int\limits_{\xi _1+\dots +\xi _4=0}b(\xi _1,\dots ,\xi _4)\hat{\mathfrak{m}}(\xi _1)\hat{\mathfrak{m}}(\xi _2)\hat{\mathfrak{m}}(\xi _3)\hat{\mathfrak{m}}(\xi _4),\qquad |b(\xi _1,\dots ,\xi _4)|\lesssim |\xi _1|+\dots +|\xi _4|.
\end{equation}

We will use the Gagliardo-Nirenberg inequalities:
\[ \| f\| _{L^\infty (Z)}^2 \lesssim \Bigg\{ \begin{aligned}
&\| f\| _{L^2}\| f_x\| _{L^2}, \\
&\| f\| _{L^2}\| f\| _{H^1}, \end{aligned}
\qquad \| f\| _{L^6(Z)}^6\lesssim \Bigg\{ \begin{alignedat}{2}
&\| f\| _{L^2}^4\| f_x\| _{L^2}^2 & &(Z=\mathbf{R}),\\
&\| f\| _{L^2}^4\| f\| _{H^1}^2&\qquad &(Z=\mathbf{T}),
\end{alignedat} \]
but the following proof does not require the use of the best constants for these inequalities.

\begin{lem}\label{lem:FG}
In the case $Z=\mathbf{R}$, the following estimates hold:
\begin{align*}
\big| F(\mathfrak{m},\mathfrak{p})\big| &\leq \frac{|\beta|}{8}\Big( \big\| D^{1/2}\partial _x\mathfrak{m}\big\| _{L^2}^2 +\big\| D^{1/2}\mathfrak{p}\big\| _{L^2}^2 \Big) +C_{\alpha,\beta} \Big( 1+ \| u\|_{L^2}^2\Big) \big\| D^{1/2}\mathfrak{m}\big\| _{L^2}^2\| u_x\| _{L^2}^2,\\
\big| G(\mathfrak{m})\big| &\leq C_{\alpha,\beta} \| u\|_{L^2}^2\big\| D^{1/2}\mathfrak{m}\big\| _{L^2}^2\| u_x\| _{L^2}^2.
\end{align*}
When $Z=\mathbf{T}$, these estimates hold if we replace $\| u_x\|_{L^2}$ by $\| u\|_{H^1}$.
\end{lem}

\begin{proof}
We focus on the case $Z=\mathbf{R}$, but the argument can be easily adapted to the case $Z=\mathbf{T}$.
From the expressions \eqref{defn:F} and \eqref{defn:G}, $F$ and $G$ can be evaluated as follows:
\begin{align*}
\big| F(\mathfrak{m},\mathfrak{p})\big| &\leq C_{\alpha,\beta}\Big( \big\| D^{1/2}\mathfrak{m}\big\| _{L^2}\big\| \mathcal{F}^{-1}|\hat{\mathfrak{m}}|\big\| _{L^\infty} \big\| D^{1/2}\mathfrak{p}\big\|_{L^2}\\
&\qquad\qquad +\big\| D^{1/2}\mathfrak{m}\big\| _{L^2}\big\| \mathcal{F}^{-1}\big| \mathcal{F}[D^{1/2}\mathfrak{m}]\big| \big\| _{L^\infty}\| \mathfrak{p}\|_{L^2}\Big) ,\\
\big| G(\mathfrak{m})\big| &\leq C_{\alpha,\beta}\big\| D^{1/2}\mathfrak{m}\big\| _{L^2}^2\big\| \mathcal{F}^{-1}|\hat{\mathfrak{m}}|\big\| _{L^\infty}^2.
\end{align*}
Using the inequalities
\begin{align*}
\big\| \mathcal{F}^{-1}|\hat{\mathfrak{m}}|\big\| _{L^\infty}&\lesssim \| \mathfrak{m}\|_{L^2}^{1/2}\| \partial_x \mathfrak{m}\|_{L^2}^{1/2}\lesssim \| u\|_{L^2}^{1/2}\| u_x\|_{L^2}^{1/2}\| u\|_{L^\infty}\lesssim \| u\|_{L^2}\| u_x\|_{L^2},\\
\big\| \mathcal{F}^{-1}\big| \mathcal{F}[D^{1/2}\mathfrak{m}]\big| \big\| _{L^\infty}&\lesssim \big\| D^{1/2}\mathfrak{m}\big\|_{L^2}^{1/2}\big\| D^{1/2}\partial_x\mathfrak{m}\big\|_{L^2}^{1/2},\\
\| \mathfrak{p}\| _{L^2}&\leq \| u\|_{L^\infty}\| u_x\| _{L^2}\lesssim \| u\|_{L^2}^{1/2}\| u_x\|_{L^2}^{3/2},
\end{align*}
we obtain the claimed estimates for $F$ and $G$ as follows:
\begin{align*}
\big| F(\mathfrak{m},\mathfrak{p})\big| 
&\leq C_{\alpha ,\beta} \Big( \big\| D^{1/2}\mathfrak{m}\big\| _{L^2}\| u\|_{L^2}\| u_x\|_{L^2}\big\| D^{1/2}\mathfrak{p}\big\|_{L^2}\\
&\qquad\qquad +\big\| D^{1/2}\mathfrak{m}\big\| _{L^2}^{3/2}\big\| D^{1/2}\partial_x\mathfrak{m}\big\|_{L^2}^{1/2}\| u\|_{L^2}^{1/2}\| u_x\|_{L^2}^{3/2}\Big) \\
&\leq \frac{|\beta |}{8}\Big( \big\| D^{1/2}\mathfrak{p}\big\|_{L^2}^2+\big\| D^{1/2}\partial_x\mathfrak{m}\big\|_{L^2}^2\Big) \\
&\qquad +C_{\alpha ,\beta} \Big( \| u\|_{L^2}^2+\| u\|_{L^2}^{2/3}\Big) \big\| D^{1/2}\mathfrak{m}\big\| _{L^2}^2\| u_x\|_{L^2}^2,\\
\big| G(\mathfrak{m})\big| &\leq C_{\alpha,\beta}\| u\|_{L^2}^2\big\| D^{1/2}\mathfrak{m}\big\| _{L^2}^2\| u_x\|_{L^2}^2.\qedhere
\end{align*}
\end{proof}

\begin{lem}\label{lem:u_x}
We have
\[ \| u_x\|_{L^2}^2\leq C_{\alpha,\beta}\Big( 1+\| u\|_{L^2}^4\Big) E[u] ,\quad E[u]\leq C_{\alpha,\beta}\Big( 1+\| u\|_{L^2}^4\Big) \| u_x\|_{L^2}^2\qquad (Z=\mathbf{R})\]
and 
\[ \begin{gathered}
\| u\|_{H^1}^2\leq C_{\alpha,\beta}\Big( 1+\| u\|_{L^2}^4\Big) \Big( E[u] +\| u\|_{L^2}^2+\| u\|_{L^2}^6\Big),\\
E[u] +\| u\|_{L^2}^2+\| u\|_{L^2}^6\leq C_{\alpha,\beta}\Big( 1+\| u\|_{L^2}^4\Big) \| u\|_{H^1}^2
\end{gathered}
\qquad (Z=\mathbf{T}).\]
\end{lem}

\begin{proof}
We first estimate $\| u_x\|_{L^2}^2$ as
\[ \| u_x\|_{L^2}^2\leq 2E[u] + C_{\alpha,\beta}\| u\| _{L^6}^6.\]
In order to estimate $\| u\|_{L^6}$ in terms of $E[u]$ and $\| u\|_{L^2}$, we follow the argument of Hayashi and Ozawa~\cite{HO92} and introduce a gauge transformation as follows:%
\footnote{%
The function $\Phi [u]$ is well-defined if $u\in H^s(Z)$ for some $s>0$.
In fact, by the Sobolev embedding, $|u|^2\in L^p$ for some $p>1$.
Since $\mathcal{H}$ is bounded on $L^p$, we have $\mathcal{H}(|u|^2)\in L^p\subset L^1_{\mathrm{loc}}$, and hence the $y$-integral in the definition of $\Phi [u]$ makes sense.
Note that the function $\Phi[u]$ may be unbounded on $\mathbf{R}$ when $Z=\mathbf{R}$.}
\[ v:=e^{i\Phi [u]}u,\quad \Phi [u](x):=\begin{cases}
-\dfrac{3}{4}\displaystyle \int _0^x \Big( \alpha |u|^2 +\beta \mathcal{H}(|u|^2)\Big) (y)\,dy &(Z=\mathbf{R}), \\[10pt]
-\dfrac{3}{4}-\hspace{-14pt}\displaystyle\int _{\mathbf{T}} \int _z^x \Big( \alpha P_{\neq 0}(|u|^2) +\beta \mathcal{H}(|u|^2)\Big) (y)\,dy\,dz\quad &(Z=\mathbf{T}),
\end{cases} \]
where $-\hspace{-11pt}\int _{\mathbf{T}}f:=\frac{1}{2\pi}\int _{\mathbf{T}}f$.
Note that $|u|=|v|$ and
\[ \begin{alignedat}{2}
|v_x| &=\Big| u_x-i\frac{3}{4}\Big( \alpha |u|^2 +\beta \mathcal{H}(|u|^2)\Big) u\Big| & &(Z=\mathbf{R}),\\
|v_x| &\leq \Big| u_x-i\frac{3}{4}\Big( \alpha |u|^2 +\beta \mathcal{H}(|u|^2)\Big) u\Big| +\frac{3|\alpha |}{8\pi} \|u\|_{L^2}^2|u| &\qquad &(Z=\mathbf{T}).
\end{alignedat} \]
In particular, we have
\[ \begin{alignedat}{2}
\| v_x\| _{L^2}^2&=E[u] & &(Z=\mathbf{R}),\\
\| v_x\| _{L^2}^2&\leq 2E[u]+ C_\alpha \| u\|_{L^2}^6 &\qquad &(Z=\mathbf{T}).
\end{alignedat} \]
When $Z=\mathbf{R}$, we see that
\[ \| u\| _{L^6}^6=\| v\|_{L^6}^6\leq C\| v\| _{L^2}^4\| v_x\|_{L^2}^2=C\| u\|_{L^2}^4E[u],\]
which is sufficient for the former estimate.
When $Z=\mathbf{T}$, the argument is modified as
\[
\| u\| _{L^6}^6=\| v\|_{L^6}^6\leq C\| v\| _{L^2}^4\| v_x\| _{L^2}^2+C\| v\|_{L^2}^6\leq C\| u\| _{L^2}^4\Big(  2E[u]+C_\alpha \| u\| _{L^2}^6\Big) +C\| u\|_{L^2}^6,
\]
which also implies the former estimate in this case.

In both cases of $Z=\mathbf{R}$ and $\mathbf{T}$, the latter estimates on $E[u]$ follows from the Gagliardo-Nirenberg inequality.
\end{proof}

From \eqref{eq:E}, Lemma~\ref{lem:FG} and Lemma~\ref{lem:u_x}, we have
\[ \partial _tE[u] + \frac{|\beta|}{8}\big\| D^{1/2}\partial _x(|u|^2)\big\| _{L^2}^2 
\leq C_{\alpha,\beta} \Big( 1+\| u\|_{L^2}^6\Big) \big\| D^{1/2}(|u|^2)\big\| _{L^2}^2E[u] \]
for $Z=\mathbf{R}$, and (together with $\partial_t\| u(t)\|_{L^2}^2\leq 0$)
\begin{align*}
&\partial _t\Big( E[u] +\| u\|_{L^2}^2+\| u\|_{L^2}^6\Big) + \frac{|\beta|}{8}\big\| D^{1/2}\partial _x(|u|^2)\big\| _{L^2}^2 \\
&\quad \leq C_{\alpha,\beta} \Big( 1+\| u\|_{L^2}^6\Big) \big\| D^{1/2}(|u|^2)\big\| _{L^2}^2\Big( E[u] +\| u\|_{L^2}^2+\| u\|_{L^2}^6\Big) 
\end{align*}
for $Z=\mathbf{T}$.
Applying the Gronwall inequality and the $L^2$ bound \eqref{L2equality}, we obtain
\[ E[u(t)]+\frac{|\beta|}{8}\int _0^t\big\| D^{1/2}\partial _x(|u(\tau )|^2)\big\| _{L^2}^2\,d\tau \leq E[\phi ]\exp \Big[ C_{\alpha,\beta}\Big( 1+\| \phi\|_{L^2}^6\Big) \| \phi \|_{L^2}^2\Big] \]
for $Z=\mathbf{R}$, and 
\begin{align*}
&E[u(t)] +\| u(t)\|_{L^2}^2+\| u(t)\|_{L^2}^6 + \frac{|\beta|}{8}\int _0^t\big\| D^{1/2}\partial _x(|u(\tau )|^2)\big\| _{L^2}^2\,d\tau \\
&\quad \leq \Big( E[\phi ] +\| \phi \|_{L^2}^2+\| \phi \|_{L^2}^6\Big) \exp \Big[ C_{\alpha,\beta}\Big( 1+\| \phi\|_{L^2}^6\Big) \| \phi \|_{L^2}^2\Big]
\end{align*}
for $Z=\mathbf{T}$.
These two inequalities show the first inequality in Proposition~\ref{prop:aprioriH1large}.
Finally, the second inequality in Proposition~\ref{prop:aprioriH1large} is obtained from Lemma~\ref{lem:u_x} and \eqref{L2equality}.

For the lower bound of $\| u(t)\| _{L^2}^2$, we just repeat the argument in \cite{KTmatrix} based on the differential inequality
\[ \partial _t\| u(t)\|_{L^2}^{-2}\leq C|\beta |\big\| D^{1/2}\partial _x(|u(t)|^2)\big\| _{L^2}\| u(t)\|_{L^2}^{-2}. \]

This completes the proof of Proposition~\ref{prop:aprioriH1large}.
\end{proof}


\appendix

\section{Energy method for KDNLS}
\label{sec:appendix}

In the appendix, we give a proof of local well-posedness for \eqref{kdnls} in high regularity by the classical energy method.
Since the proof does not rely on the parabolicity or the gauge transformation, we can treat both the periodic and non-periodic cases.
Hence, we consider
\begin{alignat}{2}
   \partial_tu - i \partial_x^2u &= \alpha \partial_x \big[ |u|^2u\big] +\beta \partial_x\big[\mathcal{H}(|u|^2)u\big] , &\qquad t > 0, ~ &x \in \Omega, \label{kdnls-e} \\
   u(0, x) &= \phi (x), & &x \in \Omega , \label{ic-e}
\end{alignat}
where $\alpha \in \mathbf{R}$, $\beta \leq 0$, and $\Omega$ is either $\mathbf{R}$ or $\mathbf{T}$.

\begin{thm} \label{thm:lwp-e}
Let $s>\frac32$, and $\Omega =\mathbf{R}$ or $\mathbf{T}$.
Then, the Cauchy problem \eqref{kdnls-e}--\eqref{ic-e} is locally well-posed in $H^s(\Omega)$.
More precisely, for any $\phi \in H^s(\Omega)$ there exist $T=T(\| \phi \|_{H^{\frac32+}})>0$ and a unique solution $u$ on $[0,T]$ in the class $C([0,T];H^s(\Omega))$.
Moreover, the solution map $\phi \mapsto u$ from $H^s$ to $C([0,T]; H^s)$ is continuous.
\end{thm}
\begin{rem}
(i) Theorem \ref{thm:lwp-e} is not a straightforward result following Bona and Smith \cite{BS75}.
Indeed, we use the argument in \cite{BS75}, but we need to take advantage of the dissipative nature of the system to obtain \emph{a priori} estimates of solution.
\par (ii) The same results as in the above theorem holds for the renormalized equation \eqref{kdnls-r}, because the transformation $u(t,x)\mapsto u(t,x-\nu t)$ is a homeomorphism on $C([0,T];H^s)$.
\end{rem}

In the following, we only consider the case $\alpha =0$ and $\beta =-1$ for simplicity, but the argument can be easily adapted to the case $\alpha \neq 0$ and $\beta <0$.
To prove the theorem, we consider the following regularized equation:
\begin{align}
   &\partial_t u_\varepsilon = i \partial_x^2 u_\varepsilon + \varepsilon \partial_x^2 u_\varepsilon - \partial_x \bigl [ \mathcal{H} (|u_\varepsilon|^2) u_\varepsilon \bigr ] , \quad t > 0, ~ x \in \Omega, \quad \varepsilon > 0. \label{kdnls-e2}
\end{align}
It is easy to show that the Cauchy problem \eqref{kdnls-e2} and \eqref{ic-e} is locally well-posed in $H^s(\Omega)$ for $s>\frac12$, with the local existence time $T=T(s, \| \phi \|_{H^s}, \varepsilon )>0$ vanishing as $\varepsilon \to 0$.

\subsection{Energy estimates}

In this subsection, we prove various energy-type estimates for a smooth solution and the difference of two smooth solutions of the regularized equation \eqref{kdnls-e2}.
We begin with the following:
\begin{prop}\label{prop:Hs-e}
For any $s\geq s_0>\frac32$ and $\phi \in H^\infty (\Omega)$, there exist $T_0=T_0(s_0, \| \phi \| _{H^{s_0}})>0$ and $C=C(s,s_0)>0$ such that the following holds.
Let $\varepsilon >0$ be arbitrary, $0<T\leq T_0$, and assume that the smooth solution $u_\varepsilon$ to \eqref{kdnls-e2} and \eqref{ic-e} exists on $[0,T]$.
Then, it holds that
\[ \| u_\varepsilon(t) \|_{H^s}\leq C\| \phi \| _{H^s},\qquad t\in [0,T].\] 
\end{prop}

\begin{proof}
It suffices to prove the estimate in the equivalent norm
\[ \| u\| _{\tilde{H}^s}^2:=\| u\|_{L^2}^2+\| D_x^su\|_{L^2}^2, \quad D_x=(-\partial _x^2)^{1/2}=\partial _x\mathcal{H}.\]

For the $L^2$ norm, we have
\begin{align*}
\frac{d}{dt}\| u_\varepsilon (t)\|_{L^2}^2&=-2\varepsilon \| \partial_xu_\varepsilon (t)\|_{L^2}^2-2\Re \int _\Omega \partial_x \big[ \mathcal{H}(|u_\varepsilon (t)|^2)u_\varepsilon (t)\big] \bar{u}_\varepsilon (t) \,dx \\
&\leq -2\int _\Omega D_x(|u_\varepsilon (t)|^2)|u_\varepsilon (t)|^2 \,dx -\int _\Omega \mathcal{H}(|u_\varepsilon (t)|^2)\partial_x (|u_\varepsilon (t)|^2) \,dx\\
&=-\int _\Omega D_x(|u_\varepsilon (t)|^2)|u_\varepsilon (t)|^2 \,dx
=-\big\| D_x^{\frac12}(|u_\varepsilon (t)|^2)\big\|_{L^2}^2
~\leq 0.
\end{align*}
Next, we have
\begin{align*}
\frac{d}{dt}\| D_x^su_\varepsilon (t)\| _{L^2}^2&=-2\varepsilon \| \partial_xD_x^su_\varepsilon (t)\|_{L^2}^2-2\Re \int _\Omega D_x^s\partial_x \big[ \mathcal{H}(|u_\varepsilon (t)|^2)u_\varepsilon (t)\big] D_x^s\bar{u}_\varepsilon (t) \,dx \\
&\leq -2\Re \int _\Omega D_x^s\big( D_x(|u_\varepsilon (t)|^2)u_\varepsilon (t) + \mathcal{H}(|u_\varepsilon (t)|^2)\partial_xu_\varepsilon (t)\big) D_x^s\bar{u}_\varepsilon (t) \,dx.
\end{align*}

Since 
\[ \sigma \geq 1,~\xi ,\eta \in \mathbf{R}\quad \Rightarrow \quad \big| |\xi +\eta |^\sigma -|\xi |^\sigma -|\eta |^\sigma \big| \leq C(|\xi |\vee |\eta |)^{\sigma -1}(|\xi |\wedge |\eta |),\]
we see that (for $s>3/2$)
\begin{align*}
\frac{d}{dt}\| D_x^su_\varepsilon (t)\| _{L^2}^2&\leq -\int _\Omega D_x^{s+1}(|u_\varepsilon (t)|^2)\cdot 2\Re \big( u_\varepsilon (t)D_x^s\bar{u}_\varepsilon (t) \big) \,dx \\
&\quad -\int _\Omega \mathcal{H}(|u_\varepsilon (t)|^2)\cdot 2\Re \big( D_x^s\partial_xu_\varepsilon (t)D_x^s\bar{u}_\varepsilon (t)\big) \,dx
~+\mathcal{R}[u_\varepsilon (t)] ,
\end{align*}
where $\mathcal{R}$ denotes the integrals of products of four functions with $2s+1$ derivatives such that each function has at most $s$ derivatives.
The first integral on the right-hand side is treated as 
\begin{align*}
&-\int _\Omega D_x^{s+1}(|u_\varepsilon (t)|^2)\cdot 2\Re \big( u_\varepsilon (t)D_x^s\bar{u}_\varepsilon (t) \big) \,dx \\
&\quad =-\int _\Omega D_x^{s+1}(|u_\varepsilon (t)|^2)\cdot D_x^s( |u_\varepsilon (t)|^2) \,dx\\
&\qquad +\int _\Omega D_x^s(|u_\varepsilon (t)|^2)\cdot D_x\big( D_x^s(|u_\varepsilon (t)|^2)-\bar{u}_\varepsilon (t)D_x^su_\varepsilon (t) -u_\varepsilon (t)D_x^s\bar{u}_\varepsilon (t)\big) \,dx\\
&\quad =-\big\| D_x^{s+\frac{1}{2}}(|u_\varepsilon (t)|^2)\big\| _{L^2}^2+\mathcal{R}[u_\varepsilon (t)] 
~ \leq \mathcal{R}[u_\varepsilon (t)] ,
\end{align*}
while the second integral is already $\mathcal{R}[u_\varepsilon (t)]$ by an integration by parts.
Hence, using interpolation and Sobolev embedding, and then combining with the $L^2$ estimate, we obtain
\[ \frac{d}{dt}\| u_\varepsilon (t)\| _{\tilde{H}^s}^2\leq C_s\| u_\varepsilon (t)\|_{\tilde{H}^{s_0}}^2\| u_\varepsilon (t)\|_{\tilde{H}^s}^2\]
for $s_0>3/2$, where $C_s$ is a positive constant depending on $s$ (and $s_0$).
By the Gronwall inequality, we have
\begin{equation}
\| u_\varepsilon (t)\| _{\tilde{H}^s}^2\leq \| \phi \|_{\tilde{H}^s}^2\exp \Big( C_s\int _0^t\| u_\varepsilon (\tau )\|_{\tilde{H}^{s_0}}^2\,d\tau \Big) .\label{est:Hs-e}
\end{equation}
Take $T_0>0$ such that $\exp \big( C_{s_0}\cdot 2\| \phi\| _{\tilde{H}^{s_0}}^2T_0\big) <2$.
Then, from \eqref{est:Hs-e} with $s=s_0$ and a continuity argument, we see $\| u_\varepsilon (t)\|_{\tilde{H}^{s_0}}^2\leq 2\| \phi\| _{\tilde{H}^{s_0}}^2$ for $t\leq T_0$.
Substituting this bound into \eqref{est:Hs-e}, we have $\| u_\varepsilon (t)\|_{\tilde{H}^s}^2\leq 2^{C_s/C_{s_0}}\| \phi\| _{\tilde{H}^s}^2$ for $t\leq T_0$.
\end{proof}

Next, we give \emph{a priori} estimates for the difference of two solutions of the regularized equations \eqref{kdnls-e2} with different $\varepsilon$'s.
Due to the lack of symmetry, the difference estimates require more careful use of the dissipative structure of KDNLS.
Let $\varepsilon \geq \varepsilon' >0$, and let $u_\varepsilon, u_{\varepsilon'}$ be smooth solutions of \eqref{kdnls-e2} with parameters $\varepsilon ,\varepsilon'$ starting from the data $\phi _{\varepsilon}$ and $\phi_{\varepsilon'}$, respectively.
By Proposition~\ref{prop:Hs-e}, these solutions exist on a time interval $[0,T_0]$ depending on $M:=\max \{ \| \phi _{\varepsilon}\|_{H^{s_0}},\,\| \phi _{\varepsilon'}\|_{H^{s_0}}\}$ but not on $\varepsilon$, $\varepsilon'$, and satisfy
\begin{equation}\label{apriori-e}
\| u_\varepsilon (t)\|_{H^s}\leq C_s\| \phi _{\varepsilon}\|_{H^s},\quad \| u_{\varepsilon'}(t)\|_{H^s}\leq C_s\| \phi _{\varepsilon'}\|_{H^s},\quad t\in [0,T_0]
\end{equation}
for any $s>\frac32$.
Define $w:=u_\varepsilon -u_{\varepsilon'}$, then $w$ is a solution of 
\begin{align}
&\partial_t w = i \partial_x^2 w +\varepsilon'\partial_x^2w +(\varepsilon -\varepsilon') \partial_x^2u_\varepsilon \notag \\
&\qquad\quad - \partial_x \bigl [ \mathcal{H} (|u_\varepsilon|^2) u_\varepsilon - \mathcal{H}(|u_{\varepsilon'}|^2) u_{\varepsilon'} \bigr ] , \quad (t,x)\in [0,T_0]\times \Omega, \label{kdnls-e3} \\
&w(0,x)=\phi _{\varepsilon}(x)-\phi _{\varepsilon'}(x), \quad x\in \Omega .\label{ic-e3}
\end{align}
Note that
\begin{align}
&\mathcal{H} (|u_\varepsilon|^2) u_\varepsilon - \mathcal{H}(|u_{\varepsilon'}|^2) u_{\varepsilon'}=\mathcal{H} (|u_\varepsilon|^2) u_\varepsilon - \mathcal{H}(|u_\varepsilon -w|^2) (u_\varepsilon -w) \notag \\
&=\mathcal{H}(|w|^2)w-\big( \mathcal{H}(2\Re(u_\varepsilon \bar{w}))w+\mathcal{H}(|w|^2)u_\varepsilon \big) +\big( \mathcal{H}(|u_\varepsilon |^2)w+\mathcal{H}(2\Re(u_\varepsilon \bar{w}))u_\varepsilon \big) .\label{nonl-e}
\end{align}

Before proving the difference estimate in $H^s$, we state and prove the easier $L^2$ estimate as the following lemma:
\begin{lem}\label{lem:L2diff-e}
Let $s_0>\frac32$.
With the above notation, we have
\[ \| w(t)\|_{L^2}^2\lesssim \| \phi _{\varepsilon}-\phi _{\varepsilon'} \|_{L^2}^2+\varepsilon ^2\| \phi _{\varepsilon}\|_{H^2}^2,\qquad t\in [0,T_0],\]
where the implicit constant depends only on $M:=\max \{ \| \phi _{\varepsilon}\|_{H^{s_0}},\,\| \phi _{\varepsilon'}\|_{H^{s_0}}\}$ (and $s_0$) and not on $\varepsilon$, $\varepsilon'$.
\end{lem}

\begin{proof}
The equation \eqref{kdnls-e3} gives
\begin{align*}
\frac{d}{dt}\| w\|_{L^2}^2&=-2\varepsilon' \| \partial_xw\|_{L^2}^2+2(\varepsilon -\varepsilon')\Re \int _{\Omega} \partial _x^2u_\varepsilon \cdot \bar{w}\,dx \\
&\quad -2\Re \int _\Omega \partial _x\big( \mathcal{H} (|u_\varepsilon|^2) u_\varepsilon - \mathcal{H}(|u_{\varepsilon'}|^2) u_{\varepsilon'}\big) \bar{w}\,dx .
\end{align*}
We dispose of the first term and apply the Schwarz inequality to the second term to obtain the bound 
\[ \varepsilon \| u_\varepsilon\|_{H^2}\| w\|_{L^2}\leq \| w\|_{L^2}^2+ C\varepsilon ^2\| \phi _{\varepsilon}\|_{H^2}^2,\]
where we have used the \emph{a priori} estimate \eqref{apriori-e}.
For the last integral, we apply \eqref{nonl-e} and observe that the terms which are cubic or higher with respect to $w$ can be estimated (using \eqref{apriori-e}) by
\[ \big( \| u_\varepsilon\|_{H^{s_0}}+\| u_{\varepsilon'}\|_{H^{s_0}}\big) ^2\| w\| _{L^2}^2\lesssim _M\| w\|_{L^2}^2.\]
The remaining terms, which are quadratic in $w$, can be treated as follows:
\begin{gather*}
\begin{aligned}
-2\Re \int _\Omega \partial _x\big( \mathcal{H}(|u_\varepsilon |^2)w\big) \bar{w}\,dx &=\int _\Omega \mathcal{H}(|u_\varepsilon |^2)\partial _x(|w|^2)\,dx \\
&\leq \big\| D_x(|u_\varepsilon |^2)\big\| _{L^\infty}\| w\| _{L^2}^2
~\lesssim _M\| w\|_{L^2}^2,
\end{aligned}\\
\begin{aligned}
&-2\Re \int _\Omega \partial _x\big( \mathcal{H}(2\Re(u_\varepsilon \bar{w}))u_\varepsilon \big) \bar{w}\,dx\\
&\quad =-\int _\Omega D_x(2\Re(u_\varepsilon \bar{w}))\cdot 2\Re (u_\varepsilon \bar{w})\,dx -\int _\Omega \mathcal{H}(2\Re(u_\varepsilon \bar{w}))\cdot 2\Re (\partial _xu_\varepsilon \cdot \bar{w})\,dx\\
&\quad \leq -\big\| D_x^{\frac{1}{2}}(2\Re (u_\varepsilon \bar{w}))\big\|_{L^2}^2+4\| u_\varepsilon\|_{L^\infty}\| \partial _xu_\varepsilon\|_{L^\infty}\| w\|_{L^2}^2
~ \lesssim _M\| w\|_{L^2}^2.
\end{aligned}
\end{gather*}
From the estimates obtained above, we have
\begin{equation}\label{est:L2diff-e}
\frac{d}{dt}\| w(t)\|_{L^2}^2\lesssim _M\| w(t)\| _{L^2}^2+\varepsilon ^2\| \phi _{\varepsilon}\|_{H^2}^2,\qquad t\in [0,T_0].
\end{equation}
The Gronwall inequality implies the claimed estimate.
\end{proof}

Now, we prove the $H^s$ difference estimate.
\begin{prop}\label{prop:Hsdiff-e}
Let $s\geq s_0>\frac32$.
We have
\begin{align*}
\| w(t)\|_{H^s}^2&\lesssim \| \phi _{\varepsilon}-\phi _{\varepsilon'} \|_{H^s}^2+\varepsilon ^2\| \phi _{\varepsilon}\|_{H^{s+2}}^2\\
&\quad +\big( \| \phi _{\varepsilon}-\phi _{\varepsilon'} \|_{L^2}^2+\varepsilon ^2\| \phi _{\varepsilon}\|_{H^2}^2\big) \| \phi _{\varepsilon}\|_{H^{s+1}}^{\frac{2s}{s+1-s_0}},\qquad t\in [0,T_0],
\end{align*}
where the implicit constant depends only on $s$, $s_0$, $M$ but not on $\varepsilon$, $\varepsilon'$.
\end{prop}

\begin{proof}
The equation \eqref{kdnls-e3} gives
\begin{align*}
\frac{d}{dt}\| D_x^sw\|_{L^2}^2&=-2\varepsilon' \| \partial_xD_x^sw\|_{L^2}^2+2(\varepsilon -\varepsilon')\Re \int _{\Omega} \partial _x^2D_x^su_\varepsilon \cdot D_x^s\bar{w}\,dx \\
&\quad -2\Re \int _\Omega D_x^s\partial _x\big( \mathcal{H} (|u_\varepsilon|^2) u_\varepsilon - \mathcal{H}(|u_{\varepsilon'}|^2) u_{\varepsilon'}\big) D_x^s\bar{w}\,dx .
\end{align*}

We treat the first two terms on the right-hand side as we did in the proof of Lemma~\ref{lem:L2diff-e}.
In view of \eqref{apriori-e}, these terms are bounded by
\begin{equation}\label{est0-e}
\| w(t)\|_{H^s}^2+C\varepsilon ^2\| \phi _{\varepsilon}\|_{H^{s+2}}^2.
\end{equation}

Let us consider the last integral based on the expression \eqref{nonl-e}.
The term corresponding to $\mathcal{H}(|w|^2)w$ is estimated in exactly the same way as we did in the proof of Proposition~\ref{prop:Hs-e}, which gives
\begin{equation}\label{est1-e}
\begin{split}
&-2\Re \int _\Omega D_x^s\partial _x\big( \mathcal{H}(|w(t)|^2)w(t)\big) D_x^s\bar{w}(t)\,dx \\
&\quad \leq -\big\| D_x^{s+\frac12}(|w(t)|^2)\big\|_{L^2}^2+C\| w(t)\| _{H^{s_0}}^2\| w(t)\|_{H^s}^2\\
&\quad \leq -\big\| D_x^{s+\frac12}(|w(t)|^2)\big\|_{L^2}^2+C_M\| w(t)\|_{H^s}^2,
\end{split}
\end{equation}
by \eqref{apriori-e}.
For the other terms
\begin{align*}
&+2\Re \int _\Omega D_x^s\partial _x\big( \mathcal{H}(2\Re(u_\varepsilon \bar{w}))w+\mathcal{H}(|w|^2)u_\varepsilon\big) D_x^s\bar{w}\,dx \\
&-2\Re \int _\Omega D_x^s\partial _x\big( \mathcal{H}(|u_\varepsilon |^2)w+\mathcal{H}(2\Re(u_\varepsilon \bar{w}))u_\varepsilon\big) D_x^s\bar{w}\,dx\\
=\,&+2\Re \int _\Omega D_x^s\big( D_x(2\Re(u_\varepsilon \bar{w}))w+ \mathcal{H}(2\Re(u_\varepsilon \bar{w})) \partial_xw+D_x(|w|^2)u_\varepsilon+\mathcal{H}(|w|^2)\partial _xu_\varepsilon \big) D_x^s\bar{w}\,dx\\
&-2\Re \int _\Omega D_x^s\big( D_x(|u_\varepsilon |^2)w+\mathcal{H}(|u_\varepsilon |^2)\partial _xw+D_x(2\Re(u_\varepsilon \bar{w}))u_\varepsilon+\mathcal{H}(2\Re(u_\varepsilon \bar{w}))\partial_xu_\varepsilon\big) D_x^s\bar{w}\,dx,
\end{align*}
we first consider the lower-order part by subtracting the terms
\begin{align*}
I:=\,&+2\Re \int _\Omega \big( D_x^{s+1}(2\Re(u_\varepsilon \bar{w}))w+ \mathcal{H}(2\Re(u_\varepsilon \bar{w})) D_x^s\partial_xw\\
&\qquad\qquad\qquad\quad +D_x^{s+1}(|w|^2)u_\varepsilon+\mathcal{H}(|w|^2)D_x^s\partial _xu_\varepsilon \big) D_x^s\bar{w}\,dx\\
&-2\Re \int _\Omega \big( D_x^{s+1}(|u_\varepsilon |^2)w+\mathcal{H}(|u_\varepsilon |^2)D_x^s\partial _xw\\
&\qquad\qquad\quad +D_x^{s+1}(2\Re(u_\varepsilon \bar{w}))u_\varepsilon+\mathcal{H}(2\Re(u_\varepsilon \bar{w}))D_x^s\partial_xu_\varepsilon\big) D_x^s\bar{w}\,dx,
\end{align*}
in which all the derivatives $D_x^s\partial_x$ fall onto one function.
Using a commutator estimate
\[ \| D_x^s(fg)-gD_x^sf \|_{L^2}\lesssim \| f\|_{H^{s-1}}\| g\| _{H^{\frac{3}{2}+}}+\| f\| _{H^{\frac{1}{2}+}}\| g\|_{H^s},\]
interpolation, \eqref{apriori-e}, and Lemma~\ref{lem:L2diff-e}, the remaining lower-order part is bounded by
\begin{align}
&\| D_x^sw\|_{L^2}\Big( \big\| 2\Re (u_\varepsilon \bar{w})\big\| _{H^s}\big( \| w\|_{H^{s_0}}+\| u_\varepsilon\|_{H^{s_0}}\big) +\big\| 2\Re (u_\varepsilon \bar{w})\big\| _{H^{s_0}}\big( \| w\|_{H^{s}}+\| u_\varepsilon\|_{H^{s}}\big) \notag \\
&\qquad +\| |w|^2\|_{H^s}\| u_\varepsilon\|_{H^{s_0}}+\| |w|^2\|_{H^{s_0}}\| u_\varepsilon\|_{H^{s}}+\| |u_\varepsilon|^2\|_{H^s}\| w\|_{H^{s_0}}+\| |u_\varepsilon|^2\|_{H^{s_0}}\| w\|_{H^{s}}\Big) \notag \\
&\lesssim \| w\|_{H^s}\Big( \| w\|_{H^s}\big( \| w\|_{H^{s_0}}+\| u_\varepsilon\|_{H^{s_0}}\big) \| u_\varepsilon\|_{H^{s_0}} + \| w\|_{H^{s_0}}\big( \| w\|_{H^{s_0}}+\| u_\varepsilon\|_{H^{s_0}}\big) \| u_\varepsilon\|_{H^s}\Big) \notag \\
&\lesssim _M\| w\|_{H^s}^2+\| w\|_{H^s}\| w\|_{H^{s_0}}\| u_\varepsilon\|_{H^s}\notag \\
&\lesssim \| w\|_{H^s}^2+\| w\|_{H^s}^{1+\frac{s_0}{s}}\| w\|_{L^2}^{\frac{s-s_0}{s}}\| u_\varepsilon\|_{H^{s_0}}^{\frac{1}{s+1-s_0}}\| u_\varepsilon\|_{H^{s+1}}^{\frac{s-s_0}{s+1-s_0}}\notag \\
&\lesssim _M\| w\|_{H^s}^2+\big( \| w\|_{H^s}^{1+\frac{s_0}{s}}\big) ^{\frac{2s}{s+s_0}}+\big( \| w\|_{L^2}^{\frac{s-s_0}{s}}\| u_\varepsilon\|_{H^{s+1}}^{\frac{s-s_0}{s+1-s_0}}\big) ^{\frac{2s}{s-s_0}}
\lesssim \| w\|_{H^s}^2+\| w\|_{L^2}^2\| u_\varepsilon\|_{H^{s+1}}^{\frac{2s}{s+1-s_0}}\notag \\
&\lesssim _M\| w\|_{H^s}^2+\big( \| \phi_{\varepsilon}-\phi_{\varepsilon'} \|_{L^2}^2+\varepsilon ^2\| \phi_{\varepsilon}\|_{H^2}^2\big) \| \phi_{\varepsilon}\|_{H^{s+1}}^{\frac{2s}{s+1-s_0}}.\label{est2-e}
\end{align}
In what follows, we write $\mathcal{R}[w,u_\varepsilon]$ to denote any terms bounded by \eqref{est2-e}.

We next consider 
\begin{align*}
I=\,&2\Re \int _\Omega \big( \mathcal{H}(|w|^2-2\Re (u_\varepsilon\bar{w}))D_x^s\partial_xu_\varepsilon -D_x^{s+1}(|u_\varepsilon|^2)w\big) D_x^s\bar{w}\,dx \\
&+2\Re \int _\Omega \mathcal{H}(2\Re (u_\varepsilon \bar{w})-|u_\varepsilon|^2)D_x^s\partial _xw\cdot D_x^s\bar{w}\,dx\\
&+2\Re \int_\Omega \big( D_x^{s+1}(2\Re(u_\varepsilon \bar{w}))w +D_x^{s+1}(|w|^2)u_\varepsilon\big) D_x^s\bar{w}\,dx\\
&-2\Re \int _\Omega D_x^{s+1}(2\Re(u_\varepsilon \bar{w}))u_\varepsilon D_x^s\bar{w}\,dx\\
=:\,&I_1+I_2+I_3+I_4.
\end{align*}
For $I_1$, we use interpolation as in \eqref{est2-e} to obtain
\begin{align*}
I_1&\lesssim \| w\|_{H^s}\| w\|_{H^{s_0-1}}\big( \| w\|_{H^{s_0-1}}+\| u_\varepsilon\|_{H^{s_0-1}}\big) \| u_\varepsilon\|_{H^{s+1}}\\
&\lesssim _M\| w\|_{H^s}^{1+\frac{s_0-1}{s}}\| w\|_{L^2}^{\frac{s-s_0+1}{s}}\| u_\varepsilon\|_{H^{s+1}}\\
&\lesssim \| w\| _{H^s}^2+\| w\|_{L^2}^2\| u_\varepsilon\|_{H^{s+1}}^{\frac{2s}{s-s_0+1}}
~=\mathcal{R}[w,u_\varepsilon] .
\end{align*}
We also observe that an integration by parts implies
\[ I_2=-\int _\Omega D(2\Re (u_\varepsilon \bar{w})-|u_\varepsilon|^2)|D^sw|^2\,dx =\mathcal{R}[w,u_\varepsilon].\]
We estimate $I_3$ as follows:%
\footnote{In the estimate of $2\Re (u_\varepsilon D_x^s\bar{w})-D_x^s(2\Re (u_\varepsilon \bar{w}))$ the highest derivative $D_x^{s+1}u_\varepsilon$ appears (after transferring one derivative from $|w|^2$), which however can be treated as for $I_1$.}
\begin{align*}
I_3&=\int _\Omega \big( D_x^{s+1}(2\Re(u_\varepsilon \bar{w}))\cdot 2\Re (wD_x^s\bar{w}) +D_x^{s+1}(|w|^2)\cdot 2\Re (u_\varepsilon D_x^s\bar{w})\big) \,dx \\
&=\int _\Omega \big( D_x^{s+1}(2\Re(u_\varepsilon \bar{w}))\cdot D_x^s(|w|^2) +D_x^{s+1}(|w|^2)\cdot D_x^s(2\Re (u_\varepsilon \bar{w}))\big) \,dx +\mathcal{R}[w,u_\varepsilon] \\
&\leq 2\big\| D_x^{s+\frac{1}{2}}(|w(t)|^2)\big\| _{L^2}\big\| D_x^{s+\frac12}(2\Re(u_\varepsilon(t)\bar{w}(t)))\big\|_{L^2}+\mathcal{R}[w,u_\varepsilon] .
\end{align*}
The first term on the right-hand side is problematic, since it includes the highest derivative on $w$ and it does not have the negative sign.
Fortunately, this term can be absorbed into the negative term in \eqref{est1-e} and another one arising in $I_4$:
\begin{align*}
I_4&=-\int _\Omega D_x^{s+1}(2\Re(u_\varepsilon \bar{w}))\cdot 2\Re (u_\varepsilon D_x^s\bar{w})\,dx\\
&=-\int _\Omega D_x^{s+1}(2\Re(u_\varepsilon \bar{w}))\cdot D_x^s(2\Re (u_\varepsilon \bar{w}))\,dx +\mathcal{R}[w,u_\varepsilon] \\
&=-\big\| D_x^{s+\frac12}(2\Re(u_\varepsilon(t)\bar{w}(t)))\big\|_{L^2}^2+\mathcal{R}[w,u_\varepsilon] .
\end{align*}
From the above estimates, we obtain
\begin{align*}
I&\leq \big\| D_x^{s+\frac{1}{2}}(|w(t)|^2)\big\| _{L^2}^2+C_M\Big( \| w\|_{H^s}^2+\big( \| \phi_{\varepsilon}-\phi_{\varepsilon'} \|_{L^2}^2+\varepsilon ^2\| \phi_{\varepsilon}\|_{H^2}^2\big) \| \phi_{\varepsilon}\|_{H^{s+1}}^{\frac{2s}{s+1-s_0}}\Big) .
\end{align*}
Combining it with \eqref{est0-e}--\eqref{est2-e}, we arrive at
\begin{align*}
\frac{d}{dt}\| D_x^sw(t)\|_{L^2}^2&\lesssim _M \| w(t)\|_{H^s}^2+\varepsilon ^2\| \phi_{\varepsilon}\|_{H^{s+2}}^2\\
&\qquad +\big( \| \phi_{\varepsilon}-\phi_{\varepsilon'} \|_{L^2}^2+\varepsilon ^2\| \phi_{\varepsilon}\|_{H^2}^2\big) \| \phi_{\varepsilon}\|_{H^{s+1}}^{\frac{2s}{s+1-s_0}},\quad t\in [0,T_0].
\end{align*}
The claim follows from the $L^2$ estimate \eqref{est:L2diff-e} and the Gronwall inequality.
\end{proof}

\subsection{Proof of Theorem~\ref{thm:lwp-e}}

We are in a position to prove Theorem~\ref{thm:lwp-e}.

\begin{proof}[Proof of Theorem~\ref{thm:lwp-e}]
We divide the proof into several steps.

{\bf Step 1}: Approximation of the initial datum.
Let $s>\frac32$, and $s_0\in (\frac32,s\wedge 2)$ be fixed.
Take an initial datum $\phi\in H^s(\Omega)$.
We use the Bona-Smith type approximation of $\phi$ (see \cite{BS75}): for $\varepsilon\in (0,1)$, define $\phi_\varepsilon:=P_{\leq \varepsilon^{-\lambda}}\phi$ with $\lambda \in (0,\frac12)$.
We see that $\{ \phi _\varepsilon\}_\varepsilon \subset H^\infty (\Omega)$, $\phi_\varepsilon \to \phi$ in $H^s$, 
\[ \sup_{\varepsilon \in (0,1)}\| \phi_\varepsilon\|_{H^s}=\| \phi\|_{H^s},\qquad \sup_{\varepsilon \in (0,1)}\| \phi_\varepsilon\|_{H^{s_0}}=\| \phi\|_{H^{s_0}}.\]
Moreover, noticing
\begin{gather*}
\| \phi_{\varepsilon}\|_{H^{\sigma}}\lesssim \varepsilon ^{-\lambda (\sigma -s)}\| \phi \|_{H^s}\qquad (\sigma >s),\\
\| \phi_{\varepsilon}-\phi_{\varepsilon'} \|_{L^2}\lesssim \varepsilon^{\lambda s}\| \phi \|_{H^s}\qquad (0<\varepsilon'<\varepsilon<1),
\end{gather*}
we can show that
\begin{gather}\label{est:BonaSmith}
\varepsilon ^2\| \phi_{\varepsilon}\|_{H^{s+2}}^2+\big( \| \phi_{\varepsilon}-\phi_{\varepsilon'} \|_{L^2}^2+\varepsilon ^2\| \phi_{\varepsilon}\|_{H^2}^2\big) \| \phi_{\varepsilon}\|_{H^{s+1}}^{\frac{2s}{s+1-s_0}}\ \lesssim \varepsilon^\gamma
\end{gather}
for some $\gamma>0$, where the implicit constant depends only on $s$ and $\|\phi\|_{H^s}$.

{\bf Step 2}: Existence.
For $\varepsilon\in (0,1)$, let $u_\varepsilon$ be the smooth solution of the regularized equation \eqref{kdnls-e2} with $u_\varepsilon (0)=\phi_{\varepsilon}$.
From the local well-posedness for \eqref{kdnls-e2} and the \emph{a priori} estimate given in Proposition~\ref{prop:Hs-e}, the solution $u_\varepsilon$ exists on a time interval $[0,T_0]$ depending only on $s_0,\| \phi \|_{H^{s_0}}$ and not on $s,\| \phi\|_{H^s},\varepsilon$.
Moreover, $\{ u_\varepsilon \}_\varepsilon$ is bounded in $C([0,T_0];H^s)$.
By the difference estimate in Proposition~\ref{prop:Hsdiff-e} and \eqref{est:BonaSmith}, we have
\[ \| u_{\varepsilon}-u_{\varepsilon'}\|_{C([0,T_0];H^s)}^2\lesssim \| \phi_{\varepsilon}-\phi_{\varepsilon'}\|_{H^s}^2+\varepsilon^\gamma \qquad (0<\varepsilon'<\varepsilon<1),\]
where the implicit constant depends only on $s,s_0,\| \phi\|_{H^{s}}$.
This shows that $\{ u_\varepsilon \}_\varepsilon$ is Cauchy in $C([0,T_0];H^s)$, and hence converges to some $u\in C([0,T_0];H^s)$ as $\varepsilon \to 0$ and 
\begin{gather}\label{contdep-e}
\| u_{\varepsilon}-u\|_{C([0,T_0];H^s)}^2\lesssim \| \phi_{\varepsilon}-\phi\|_{H^s}^2+\varepsilon^\gamma \qquad (0<\varepsilon<1).
\end{gather}
It is straightforward to show that the limit $u$ solves \eqref{kdnls-e}--\eqref{ic-e}, either in the sense of distributions or as a solution to the associated integral equation.

{\bf Step 3}: Continuous dependence on initial data.
Let us take any sequence $\{ \phi^j\} _{j\in \mathbf{N}}\subset H^s$ converging to $\phi$. 
Let $u^j,u_\varepsilon^j$ be the solutions of \eqref{kdnls-e}, \eqref{kdnls-e2} with $u^j(0)=\phi^j$ and $u_\varepsilon^j(0)=P_{\leq \varepsilon ^{-\lambda}}\phi^j$, respectively.
Note that these solutions exist on the common interval $[0,T_0]$ corresponding to the data size $2\| \phi\|_{H^{s_0}}$ for sufficiently large $j$.
Applying Proposition~\ref{prop:Hsdiff-e} to the difference $u_\varepsilon ^j-u_\varepsilon$ (for fixed $\varepsilon =\varepsilon'\in (0,1)$) and using \eqref{est:BonaSmith}, we have
\[ \| u_\varepsilon ^j-u_\varepsilon\|_{C([0,T_0];H^s)}^2\lesssim \| P_{\leq \varepsilon ^{-\lambda}}(\phi^j-\phi)\| _{H^s}^2+\| P_{\leq \varepsilon ^{-\lambda}}(\phi^j-\phi)\| _{L^2}^2\varepsilon ^{-\delta}+\varepsilon ^\gamma ,\]
where $\delta :=\frac{2\lambda s}{s+1-s_0}>0$.
Combining it with \eqref{contdep-e}, we have
\begin{align*}
&\limsup _{j\to \infty}\| u^j-u\|_{C([0,T_0];H^s)}^2\\
&\lesssim \limsup _{j\to \infty}\big( \| P_{>\varepsilon ^{-\lambda}}\phi^j\|_{H^s}^2+ \| P_{>\varepsilon ^{-\lambda}}\phi\|_{H^s}^2+\varepsilon ^\gamma +\| \phi^j-\phi\|_{H^s}^2+\| \phi^j-\phi\|_{L^2}^2\varepsilon ^{-\delta}\big) \\
&\lesssim \sup _{j\in \mathbf{N}}\| P_{>\varepsilon ^{-\lambda}}\phi^j\|_{H^s}^2+ \| P_{>\varepsilon ^{-\lambda}}\phi\|_{H^s}^2+\varepsilon ^\gamma \qquad (0<\varepsilon<1).
\end{align*}
The first term on the right-hand side vanishes as $\varepsilon \to 0$, since $\{ \phi^j\}_j$ is precompact in $H^s$.
Therefore, letting $\varepsilon \to 0$ shows the convergence $u^j\to u$ in $C([0,T_0];H^s)$.

{\bf Step 4}: Uniqueness.
Let $u,v\in C([0,T];H^s)$ be two solutions of \eqref{kdnls-e} on $[0,T]$ with the common initial datum.
For $N\in \mathbf{N}$, let $u_N:=P_{\leq N}u$ and $v_N:=P_{\leq N}v$.
Then, $u_N$ and $v_N$ are smooth and satisfy
\[ \partial_tu_N=i\partial_x^2u_N-P_{\leq N}\partial_x\big[ \mathcal{H}(|u|^2)u\big] ,\quad (t,x)\in [0,T]\times \Omega .\]
Similarly to the proof of Lemma~\ref{lem:L2diff-e}, we estimate the time derivative of the $L^2$ norm of $w_N(t):=u_N(t)-v_N(t)$ as
\begin{align*}
&\frac{d}{dt}\| w_N\| _{L^2}^2~=~-2\Re \int _\Omega P_{\leq N}\partial_x \big( \mathcal{H}(|u|^2)u-\mathcal{H}(|v|^2)v\big) \bar{w}_N\,dx\\
&\leq 2\big\| P_{\leq N}\partial_x \big( \mathcal{H}(|u|^2)u-\mathcal{H}(|v|^2)v\big) -\partial_x \big( \mathcal{H}(|u_N|^2)u_N-\mathcal{H}(|v_N|^2)v_N\big) \big\| _{L^2}\| w_N\| _{L^2}\\
&\quad -2\Re \int _\Omega \partial_x \big( \mathcal{H}(|u_N|^2)u_N-\mathcal{H}(|v_N|^2)v_N\big) \bar{w}_N\,dx\\
&\leq 2\big\| P_{>N}\big( \mathcal{H}(|u|^2)u-\mathcal{H}(|v|^2)v\big) \big\| _{H^1}\| w_N\| _{L^2}\\
&\quad +2\big\| \big( \mathcal{H}(|u|^2)u-\mathcal{H}(|v|^2)v\big) -\big( \mathcal{H}(|u_N|^2)u_N-\mathcal{H}(|v_N|^2)v_N\big) \big\| _{H^1}\| w_N\| _{L^2}\\
&\quad -\big\| D_x^{\frac{1}{2}}(2\Re (u_N\bar{w}_N))\big\| _{L^2}^2+C\big( \| u_N\|_{H^{s_0}}+\| v_N\|_{H^{s_0}}\big) ^2\| w_N\|_{L^2}^2.
\end{align*}
For the first two terms on the right-hand side, we use the regularity of $u,v$ to have
\begin{gather*}
\begin{aligned}
&\big\| P_{>N}\big( \mathcal{H}(|u|^2)u-\mathcal{H}(|v|^2)v\big) \big\| _{H^1}\\
&\quad \lesssim N^{-(s-1)}\big\| \mathcal{H}(|u|^2)u-\mathcal{H}(|v|^2)v\big\|_{H^s}~\lesssim N^{-(s-1)}M^3,
\end{aligned}\\
\begin{aligned}
&\big\| \big( \mathcal{H}(|u|^2)u-\mathcal{H}(|v|^2)v\big) -\big( \mathcal{H}(|u_N|^2)u_N-\mathcal{H}(|v_N|^2)v_N\big) \big\| _{H^1}\\
&\quad \lesssim \big( \| u\|_{H^1}^2\| P_{>N}u\|_{H^1}+\| v\|_{H^1}^2\| P_{>N}v\|_{H^1}\big) ~\lesssim N^{-(s-1)}M^3,
\end{aligned}
\end{gather*}
where $M:=\max \{ \| u\|_{C([0,T];H^s)},\| v\|_{C([0,T];H^s)}\}<\infty$.
Consequently, we have
\begin{align*}
\frac{d}{dt}\| w_N\| _{L^2}^2&\lesssim N^{-(s-1)}M^3\| w_N\| _{L^2}+M^2\| w_N\| _{L^2}^2 \\
&\lesssim M^2\| w_N\| _{L^2}^2 +N^{-2(s-1)}M^4,\qquad t\in [0,T].
\end{align*}
Applying the Gronwall inequality (and noting that $u(0)=v(0)$), we obtain
\[ \sup_{t\in [0,T]}\big\| P_{\leq N}\big( u(t)-v(t)\big) \big\|_{L^2}^2\leq N^{-2(s-1)}M^2e^{CM^2T}.\]
Letting $N\to \infty$, we conclude that $u(t)=v(t)$ on $[0,T]$.

This is the end of the proof of Theorem~\ref{thm:lwp-e}.
\end{proof}


\bigskip
\noindent
{\bf Acknowledgements}.
The first author N.K is partially supported by JSPS KAKENHI Grant-in-Aid for Young Researchers (B) (16K17626) and Grant-in-Aid for Scientific Research (C) (20K03678).
The second author Y.T is partially supported by JSPS KAKENHI Grant-in-Aid for Scientific Research (B) (17H02853).


\end{document}